\documentclass[10pt]{scrartcl}

\usepackage{amsmath, amssymb, amsthm, marginnote, ltxcmds, adjustbox, stmaryrd, graphicx, bbold, stackengine}

\usepackage{mathtools, stackengine, changepage, ragged2e}

\usepackage{todonotes}

\definecolor{cite}{HTML}{11871E}
\definecolor{url}{HTML}{698996}
\definecolor{link}{HTML}{912F1B}

\usepackage[pdfencoding=unicode, colorlinks=true, linkcolor=link, citecolor=cite, urlcolor=url, linktocpage]{hyperref}

\usepackage[backend=biber, style=alphabetic, maxnames=10, maxalphanames=3, minalphanames=3]{biblatex}

\usepackage{tikz}
\usetikzlibrary{cd}
\usetikzlibrary{arrows, arrows.meta, positioning, calc}
\tikzcdset{arrow style=tikz, diagrams={>={Straight Barb[scale=0.8]}}}

\tikzstyle{arrow} = [-{Straight Barb[scale=0.8]}, line width=0.2mm]

\usepackage{enumitem}
{\end{enumerate}%
}

\usepackage[
	letterpaper,
	twoside=false,
	textheight=22cm,
	textwidth=14.4cm,
	marginparsep=0.75cm,
	marginparwidth=2.5cm,
	heightrounded,
	centering
]{geometry}

\usepackage{stmaryrd}

\usepackage[capitalise]{cleveref}

\Crefname{prop}{Proposition}{Propositions}
\Crefname{lem}{Lemma}{Lemmas}
\Crefname{cor}{Corollary}{Corollaries}
\Crefname{thm}{Theorem}{Theorems}
\Crefname{alphThm}{Theorem}{Theorems}
\Crefname{alphCor}{Corollary}{Corollaries}

\Crefname{defn}{Definition}{Definitions}
\Crefname{notation}{Notation}{Notations}
\Crefname{cons}{Construction}{Constructions}
\Crefname{rmk}{Remark}{Remarks}
\Crefname{obs}{Observation}{Observations}
\Crefname{trick}{Trick}{Tricks}
\Crefname{warning}{Warning}{Warnings}
\Crefname{conj}{Conjecture}{Conjectures}
\Crefname{assump}{Assumption}{Assumptions}
\Crefname{recollect}{Recollection}{Recollections}
\Crefname{terminology}{Terminology}{Terminologies}
\Crefname{hypothesis}{Hypothesis}{Hypotheses}

\Crefname{question}{Question}{Questions}
\Crefname{example}{Example}{Examples}

\Crefname{figure}{Figure}{Figures}

\crefformat{equation}{(#2#1#3)}
\crefformat{section}{\S#2#1#3}
\crefmultiformat{section}{\S\S#2#1#3}{and~#2#1#3}{, #2#1#3}{, and~#2#1#3}

\newtheorem{thm}[subsubsection]{Theorem}
\newtheorem{prop}[subsubsection]{Proposition}
\newtheorem{lem}[subsubsection]{Lemma}
\newtheorem*{lem*}{Lemma}
\newtheorem{cor}[subsubsection]{Corollary}

\newtheorem{alphThm}{Theorem}

\newcommand{\neutralize}[1]{\expandafter\let\csname c@#1\endcsname\count@}
\makeatother

\newtheorem{alphCor}{Corollary}



\theoremstyle{definition}
\newtheorem{defn}[subsubsection]{Definition}
\newtheorem{cons}[subsubsection]{Construction}
\newtheorem{nota}[subsubsection]{Notation}

\newtheorem{recollect}[subsubsection]{Recollections}
\newtheorem{terminology}[subsubsection]{Terminology}

\newtheorem{hypothesis}[subsubsection]{Hypothesis}

\theoremstyle{remark}
\newtheorem{rmk}[subsubsection]{Remark}
\newtheorem{obs}[subsubsection]{Observation}
\newtheorem{example}[subsubsection]{Example}
\newtheorem{fact}[subsubsection]{Fact}

\newcommand{\baseCat}{{\mathcal{T}}}

\newcommand{\unstable}{\mathrm{un}}

\newcommand{\tkappa}{^{\underline{\kappa }}}
\newcommand{\tomega}{^{\underline{\omega}}}
\newcommand{\perfect}{^{\mathrm{perf}}}
\newcommand{\norm}{\mathrm{N}}
\newcommand{\calg}{\mathrm{CAlg}}
\newcommand{\spc}{\mathcal{S}}
\newcommand{\vop}{^{{\myuline{\mathrm{op}}}}}
\newcommand{\Tperfect}{^{{\mathrm{perf}}}}
\newcommand{\ind}{\mathrm{Ind}}

\newcommand{\cocartesianCategory}{\mathrm{coCart}}

\DeclareMathOperator{\cmonoid}{\mathrm{CMon}}
\newcommand{\canonical}{\mathrm{can}}

\DeclareMathOperator{\forget}{\mathrm{fgt}}

\DeclareMathOperator{\rfunc}{\mathrm{RFun}}
\DeclareMathOperator{\poset}{\mathrm{Pos}}
\newcommand{\loops}{\Omega^{\infty}}

\DeclareMathOperator{\lfunc}{\mathrm{LFun}}
\newcommand{\simplex}{\mathbf{\Delta}}

\newcommand{\eval}{\mathrm{ev}}

\newcommand{\constant}{\operatorname{const}}
\DeclareMathOperator{\tconstant}{\underline{\mathrm{const}}}
\newcommand{\cat}{\mathrm{Cat}}

\newcommand{\presentable}{\mathrm{Pr}}

\newcommand{\nmot}{\mathrm{NMot}}

\newcommand{\Texact}{^{{\mathrm{ex}}}}
\newcommand{\Tlexact}{^{{\mathrm{lex}}}}
\newcommand{\Trexact}{^{{\mathrm{rex}}}}
\newcommand{\Sdot}{\mathrm{S}}
\newcommand{\Qcons}{\mathrm{Q}}

\newcommand{\A}{\mathcal{A}}
\newcommand{\sC}{{\mathcal C}}
\newcommand{\D}{{\mathcal D}}
\newcommand{\decalage}{\mathrm{D\acute{e}c}}
\newcommand{\B}{\mathcal{B}}
\newcommand{\op}{^{\mathrm{op}}}

\DeclareMathOperator{\mapsp}{\mathrm{map}}
\newcommand{\sphere}{\mathbb{S}}

\newcommand{\M}{\mathcal{M}}
\newcommand{\Z}{\mathcal{Z}}
\newcommand{\sU}{{\mathcal U}}
\newcommand{\K}{\mathcal{K}}
\newcommand{\R}{\mathcal{R}}
\newcommand{\E}{\mathcal{E}}

\newcommand{\KTheory}{\mathrm{K}}

\DeclareMathOperator{\res}{\mathrm{Res}}
\DeclareMathOperator{\mackey}{\mathrm{Mack}}
\DeclareMathOperator{\presheaf}{\mathrm{PSh}}

\newcommand{\add}{^{\mathrm{add}}}
\newcommand{\exact}{^{\mathrm{ex}}}
\newcommand{\spectra}{\mathrm{Sp}}
\newcommand{\finite}{\mathrm{Fin}}

\DeclareMathOperator{\effBurn}{\mathrm{Span}}

\newcommand{\id}{\mathrm{id}}
\DeclareMathOperator{\map}{\mathrm{Map}}
\DeclareMathOperator{\fib}{\operatorname{fib}}
\DeclareMathOperator{\cofib}{\operatorname{cofib}}
\DeclareMathOperator{\func}{\mathrm{Fun}}

\newcommand{\unit}{\mathbb{1}}
\newcommand{\catinf}{\mathrm{Cat}}
\DeclareMathOperator{\twistedArrow}{\operatorname{TwAr}}
\newcommand{\orbit}{\mathcal{O}}

\DeclareMathOperator{\cofree}{\operatorname{\underline{Cofree}}}

\newcommand{\tcone}{^{\underline{\triangleleft}}}
\newcommand{\tcocone}{^{\underline{\triangleright}}}
\newcommand{\terminalTCat}{\underline{\ast}}

\newcommand{\fanpsi}{\Psi}

\newcommand{\pointwise}{^{\mathrm{pw}}}
\newcommand{\botPointwise}{_{\mathrm{pw}}}
\newcommand{\botPointwiseKappa}{_{\mathrm{pw},\kappa}}
\newcommand{\normed}{\mathrm{nm}}

\newcommand{\botNormedKappa}{_{\mathrm{nm},\kappa}}
\newcommand{\totimes}{^{\underline{\otimes}}}

\newcommand{\Tcore}{^{{\simeq}}}
\newcommand{\tstable}{\underline{\mathrm{st}}}
\newcommand{\proper}{\mathcal{P}}

\newcommand{\localisation}{{\lambda}}

\newcommand{\borel}{\mathrm{Bor}}

\newcommand{\target}{\mathrm{tgt}}

\newcommand{\trivial}{\mathrm{triv}}

\def\colim{\qopname\relax m{colim}}

\newcommand{\arrdisp}{0.33ex}
\newcommand{\arrdisplacementsp}{0.72ex}

\newcommand{\ardis}{\ar@<\arrdisp>}
\newcommand{\ardissp}{\ar@<\arrdisplacementsp>}

\usepackage{contour}
\usepackage{ulem}

\contourlength{0.5pt}

\newcommand{\myuline}[1]{%
  \uline{\phantom{#1}}%
  \llap{\contour{white}{#1}}%
}

\makeatletter
\newcommand*{\saved@myuline}{}
\let\saved@myuline\myuline

\newcommand*{\mathuline}{%
  \mathpalette{\math@myuline\saved@myuline}%
}
\newcommand*{\math@myuline}[3]{%
  \mbox{#1{$#2#3\m@th$}}%
}

\renewcommand*{\myuline}{%
  \relax  
  \ifmmode
    \expandafter\mathuline
  \else
    \expandafter\saved@myuline
  \fi
}
\makeatother

\title{\LARGE Parametrised noncommutative motives and equivariant cubical descent in  algebraic K-theory}
\date{ \today}
\author{\textsc{Kaif Hilman}\thanks{University of Bonn\\kaif@math.uni-bonn.de}}

\begin{document}
\maketitle

\vspace{-6mm}

\begin{abstract}
    For an atomic orbital base category in the sense of \cite{parametrisedIntroduction}, we introduce the category of parametrised perfect--stable categories and use it to construct the parametrised version of noncommutative motives in which algebraic $\KTheory$-theory is corepresented. Furthermore, we initiate a rudimentary theory of parametrised cubes which could be of independent interest, generalising some of the elements in \cite{dottoExcision} beyond the equivariant case. Using this cubical theory, we show that in the equivariant case for finite 2--groups $G$, the parametrised noncommutative motives canonically refine to $G$--symmetric monoidal categories. Consequently, this endows the equivariant algebraic $\KTheory$-theory spectra for these groups with the structure of $\mathbb{E}_{\infty}$--ring spectra equipped with multiplicative norms in the sense of \cites{HHR}. Along the way, we will also provide a machine to manufacture $G$--symmetric monoidal categories from symmetric monoidal categories equipped with $G$--actions and elucidate how the aforementioned parametrised perfect--stable categories relate to Mackey functors valued in perfect--stable categories.
\end{abstract}

\tableofcontents

\section{Introduction}
Algebraic K-theory, as a functor $\KTheory : \cat\perfect\rightarrow \spectra$, is an additive spectral invariant on the $\infty$-category $\cat\perfect$ of small perfect stable $\infty$-categories by the work of \cite{BGT13} and moreover admits a lax symmetric monoidal refinement by \cite{BGTMult}. The methods of these papers were to construct a ``universal'' category of \textit{noncommutative motives} $\nmot$ through which the functor $\KTheory$ above factors. Using this lax symmetric monoidal structure, \cite{barwick2} has been able to show that equivariant algebraic K--theory for finite groups $G$  naturally admits the so--called \textit{Green functor} structure, in the sense of Mackey functor theory (i.e. $\mathbb{E}_{\infty}$--algebras in the Day convolution structure on $\spectra_G\coloneqq \mackey_G(\spectra)$).

On the other hand, there are more refined multiplicative structures in equivariant homotopy theory in the form of  the \textit{multiplicative norms} which were exploited in the celebrated resolution of the Kervaire invariant one problem in \cite{HHR}. The goal of this paper is to investigate the equivariant analogues of all the elements involved in the functor $\KTheory\colon\cat\perfect\rightarrow \nmot\rightarrow\spectra$ and to lay the groundwork in studying the equivariant multiplicative refinements of all these. It turns out that  the question of these  refinements is intimately related to that of \textit{cubical descent} in algebraic K--theory. Therefore, in order to handle these cubical matters in the equivariant setting, we will initiate the study of a theory on what we term as \textit{parametrised cubes}, which could be of independent interest. In particular, this serves also as the foundations for a theory of \textit{parametrised functor calculus} introduced in \cite{kaifSil}. Employing this rudimentary theory, our main result will then be that equivariant algebraic K--theory  can indeed be refined to admit multiplicative norms in the special case when $G$ is a group of order $2^n$ for any $n$.

\textbf{Convention:} This paper is written in the language of $\infty$--categories as developed in \cite{lurieHTT,lurieHA}. As such, in order not to encumber the exposition, by a ``category'' we will always mean an ``$\infty$--category''. Hence, for example, we will write $\cat$ for the $\infty$--category of small $\infty$--categories, usually written as $\cat_{\infty}$.

\subsection*{Motivations for equivariant multiplicative norms}

Let us first recall the notion of multiplicative norms and why it can be a useful piece of extra structure. Let $R\in\calg(\spectra_G^{\otimes})$ be an $\mathbb{E}_{\infty}$--ring in genuine $G$--spectra. In particular, this means that $R$ is equipped with the structure of a multiplication $\otimes\colon R\otimes R\rightarrow R$ which participates in many coherence diagrams. For a subgroup $H\leq G$, there is a functor $\norm^G_H\colon \spectra_H\rightarrow \spectra_G$ called the \textit{norm}, which intuitively is the multiplicative version of the usual additive induction functor $\ind^G_H\colon \spectra_H\rightarrow \spectra_G$ which satisfies a multiplicative version of the usual double--coset formula for $\ind^G_H$. We may then ask if we could endow $R$ with the structure of an ``equivariant multiplication'', i.e. a ring map
$\otimes^G_H\colon \norm^G_H\res^G_HR \longrightarrow R$ participating in appropriate coherence diagrams. These maps are the \textit{multiplicative norms} alluded to above. In the presence of all these maps together with coherences, we will say that $R$ has been enhanced to the structure of a $G$--$\mathbb{E}_{\infty}$--ring spectrum. 

The multiplicative norms provide a lot extra structure and constraints that can be productively exploited. Hill, Hopkins, and Ravenel   \cite{HHR} used these in an essential way to construct a bordism spectrum  to obstruct the existence of Kervaire invariant one elements for all but six of the infinitely many cases. For an example of the rigid demands the multiplicative norms put on a ring $G$--spectrum, see for instance \cite[Ex. 3.28]{irakliDerivedMackey} where it is argued that the so--called ``inflated''  $C_2$--spectral Mackey functor $H\mathbb{Z}$ \textit{cannot} admit such structures. Apart from that, \cite[Rmk. III.1.5]{nikolausScholze}  pointed out how the norms can be used to recover their  Tate diagonal. Yet another use of this structure is as a ``trick tool'' for a technical, but important, step in proving ``completion theorems'' such as the Atiyah--Segal completion theorem  and the Segal conjecture. See for example \cite[Lem. 7.4.17]{kaifThesis} for more details on this.

Unsurprisingly, these structures turn out to be rather tricky to formulate precisely and are yet trickier to construct. It was in large part to this end that the theory of \textit{parametrised higher categories} were introduced and studied by Barwick, Dotto, Glasman, Nardin, and Shah in \cite{parametrisedIntroduction,expose1Elements,shahThesis,shahPaperII,nardinShah}. We now recall the basic philosophy of this setup and explain what we mean by equivariant algebraic K--theory in this language.

\subsection*{Parametrised higher categories and equivariant algebraic K--theory}

The primary ingredient of the parametrised formalism aforementioned is a fixed small base category $\baseCat$, and the basic category of interest is $\cat_{\baseCat}\coloneqq \func(\baseCat
\op,\cat)$ whose objects are termed as $\baseCat$\textit{--categories}. In fact, for most of this article we will restrict to the case when $\baseCat$ is \textit{atomic orbital} (cf. \cref{DefinitionAtomicOrbital}), which is a class of categories isolated by \cite{expose1Elements} that are particularly suited to higher algebraic considerations such as parametrised semiadditivity and multiplicative structures. The most important example of an atomic orbital base category for us will be the case of $\baseCat$ being the orbit category $\orbit_G$ for a finite group $G$, in which case we will write $\cat_G\coloneqq \func(\orbit_G\op,\cat)$ where the objects are called $G$\textit{--categories}. 
With this notion, a theory of \textit{parametrised (co)limits} can be developed which affords precise meanings to many analogies between equivariant higher algebra and ordinary higher algebra. 

As hinted at above, one salient feature of this theory is that it supports the theory of $G$\textit{--symmetric monoidal categories} by which we may make sense of $G$--$\mathbb{E}_{\infty}$--algebras $\calg_G$, i.e. $\mathbb{E}_{\infty}$--algebras equipped with multiplicative norms. For concreteness, \cite[Thm. 2.3.9]{nardinShah} provides an equivalent description of the category of $G$--symmetric monoidal categories as $\mackey_G(\cat)$, i.e. $G$--Mackey functors (in the sense of \cite{barwick1})  valued in $\cat$. 

Notwithstanding this pleasant generality,  constructing examples of $G$--$\mathbb{E}_{\infty}$--algebras remains by and large a tricky one.  It is  in this context that we offer our first main result:

\begin{alphThm}[Precise and full statement in {\cref{monoidalBorelificationPrinciple}}]\label{alphaThm:monoidalBorelificationPrinciple}
Let  $\D^{\otimes}$ be a symmetric monoidal category with a $G$--action. From this, we may construct a $G$--symmetric monoidal category $\underline{\borel}(\D^{\otimes})$ with underlying $G$--category  $\underline{\borel}(\D)\coloneqq \{G/H\mapsto \D^{hH}\}$. The  norm map $\norm^G_H\colon \D^{hH}\rightarrow\D^{hG}$ is concretely described as follows: for $X \in \D^{hH}$ a $H$--object in $\underline{\borel}(\D)$, the $G$--object $\norm^G_HX \in \D^{hG}$ is given by $\bigotimes_{g\in G/H}gX$. Moreover, the category of $G$--$\mathbb{E}_{\infty}$--algebras in $\underline{\borel}(\D^{\otimes})$ is equivalent to $\calg(\D^{\otimes})^{hG}$, i.e. $\mathbb
{E}_{\infty}$--algebras in $\D^{\otimes}$ equipped with a $G$--action.\footnote{Since the appearance of our work, there are now also different proofs of this result, cf. for instance \cite{phil}.}
\end{alphThm}
\noindent We call the construction $\underline{\borel}(-)$ the \textit{Borelification} functor, and in the full version, we also show that \textit{any} $G$--symmetric monoidal category admits a natural $G$--symmetric monoidal functor to its Borelification. The proof of this will proceed by first placing everything at the appropriate categorical level  and later extracting the desired statements by a process of decategorification.\footnote{In hindsight, this was very much inspired by the philosophy of \cite{GGN} (which, in turn, was inspired by \cite[ $\S4.8.2$]{lurieHA}) in dealing with monoidal \textit{structures} via the \textit{properties} of categorical products.}

\vspace{1mm}

Next, we work towards explaining what we shall mean by equivariant (or more generally, $\baseCat$--parametrised) algebraic K--theory in this article. For a fixed finite group $G$ and a category with finite products $\sC$, \cite{barwick1} supplies us  the $\infty$--categorical version of $G$--Mackey functors valued in $\sC$ denoted $\mackey_G(\sC)$.  Since $\mackey_G(-)$ is functorial on product--preserving functors, we may apply it to the functor $\KTheory : \cat\perfect\rightarrow\spectra$ to obtain $\mackey_G(\KTheory)\colon \mackey_G(\cat\perfect)\longrightarrow \mackey_G(\spectra)=\spectra_G$.
The lax symmetric monoidal refinement of $\KTheory$ then induces a lax symmetric monoidal refinement of $\mackey_G(\KTheory)$, and this is  the content of the Green functor refinement of equivariant algebraic K--theory due to \cite{barwick2}. This construction has been used to great effect, for example, in \cite{CMMN2}, where they provided one of the key ingredients for the proof of the chromatic redshift conjecture for $\mathbb{E}_{\infty}$--rings.

Unfortunately, this $\mathbb{E}_{\infty}$--structure lacks the desired multiplicative norms.  Worse still, we do not even know if $\mackey_G(\cat\perfect)$ may be endowed with a reasonable $G$--symmetric monoidal structure with which to even begin to speak of $G$--$\mathbb{E}_{\infty}$--algebras. Therefore, as a first step, we provide a replacement for $\mackey_G(\cat\perfect)$ by introducing the $\baseCat$--category $\underline{\cat}\Tperfect_{\baseCat}$ of $\baseCat$--perfect--stable categories for an arbitrary atomic orbital $\baseCat$.  The benefit of this  is that, at least when $\baseCat=\orbit_G$ which is our chief case of interest, there is already a natural $\baseCat$--symmetric monoidal structure\footnote{At least under \cref{hypothesis:normed_spectra,hypothesis2:norm_finiteness} which are satisfied in the equivariant setting. See the end of the introduction for an explanation of this.}  on $\underline{\cat}\Tperfect_{\baseCat}$. As a crucial bridge between this notion and the version of equivariant K--theory in \cite{barwick2,CMMN2}, we provide  our next main result:

\begin{alphThm}[See {\cref{mainMackeyInclusion}}] \label{alphaThm:mainMackeyInclusion}\label{alphaThm:MackeyInclusionCreatesFibreCofibres}
We have a conservative $\baseCat$--faithful inclusion $\underline{\cat}\Tperfect_{\baseCat} \subset\underline{\mackey}_{\baseCat}(\cat\perfect)$. Moreover, this inclusion preserves and reflects parametrised (co)limits.
\end{alphThm} 
\noindent Intuitively, in the case of $\baseCat=\orbit_G$, the image of the faithful inclusion $\cat\Tperfect_G\subset \mackey_G(\cat\perfect)$ consists of those $\cat\perfect$--valued $G$--Mackey functors where the abstract transfer maps are both left and right adjoint to the restrictions. This not only allows us to \textit{define} $\baseCat$--parametrised algebraic K--theory as the composite
$\underline{\KTheory}_{\baseCat}\colon \cat\Tperfect_{\baseCat}\hookrightarrow \mackey_{\baseCat}(\cat\perfect)\xlongrightarrow{\mackey_{\baseCat}(\KTheory)} \mackey_{\baseCat}(\spectra)$, but
 also  to port many known concepts and results about $\cat\perfect$ to the setting of $\cat\Tperfect
_{\baseCat}$, the most important of which is the theory of \textit{split Karoubi sequences}. Recall that a {split Karoubi sequence in $\cat\perfect$} is a sequence of objects $\sC\hookrightarrow \D\twoheadrightarrow \E$ in $\cat\perfect$ which is both a cofibre and a fibre sequence, and where both functors admit both left and right adjoints. Thanks to \cref{alphaThm:MackeyInclusionCreatesFibreCofibres}, we can and will define a sequence in $\cat\Tperfect_{\baseCat}$ to be a split Karoubi sequence if it is so when viewed as a sequence in $\mackey_{\baseCat}(\cat\perfect)$ under the faithful inclusion, and we deduce various properties of this definition in \cref{subsection:splitKaroubiSequences} from the unparametrised theory.

\subsection*{Multiplicativity and cubical descent}
We now introduce the twin problem of endowing algebraic K--theory with equivariant multiplicative structures and that of equivariant cubical descent. First, given a symmetric monoidal structure $\sC^{\otimes}$ on $\sC$, recall from \cite[Def. 2.2.1.6, Prop. 2.2.1.9]{lurieHA} that we can enhance a Bousfield localisation $L \colon \sC \rightarrow \D$ to a symmetric monoidal functor by verifying that for any finite collection $\{f_i\colon x_i\rightarrow y_i\}_{i\in I}$ of  $L$--equivalences in $\sC$, the morphism $\otimes_if_i\colon \otimes_ix_i\rightarrow \otimes_iy_i$ is also an $L$--equivalence, in which case we  say that $L$ is \textit{compatible with the symmetric monoidal structure} $\sC^{\otimes}$. In fact, by \cite[Ex. 2.2.1.7]{lurieHA}, this condition can be drastically simplified just to checking that  if $f\colon x\rightarrow y$ is an $L$--equivalence, then so is $\id_z\otimes f\colon z\otimes x\rightarrow z\otimes y$. This is because the map $f_1\otimes f_2\colon x_1\otimes x_2\rightarrow y_1\otimes y_2$ can be factored as $(f_1\otimes \id_{y_2})\circ(\id_{x_1}\otimes f_2)$ where both maps are of the form specified in the simplified condition.

Next,  it was shown in \cite{BGT13} that algebraic K--theory may be factored as the composition
 $ \KTheory\colon \cat\perfect \xlongrightarrow{y}\presheaf(\cat\perfect) \xlongrightarrow{\mathcal{L}} \nmot \xlongrightarrow{\mapsp(\Z(\spectra^{\omega}),-)}\spectra$ where $y$ is the Yoneda embedding, $\Z\coloneqq \mathcal{L}y$, and $\nmot$ is the so--called stable category of \textit{noncommutative motives} obtained by stabilising the  localisation of $\presheaf(\cat\perfect)$ against the split Karoubi sequences, i.e. inverting the maps $y(\D)/y(\sC)\rightarrow y(\E)$ for split Karoubi sequences as above. We call maps in $\presheaf(\cat\perfect)$ that get inverted by $\mathcal{L}$ the \textit{motivic equivalences}. 
We then upgrade the functor K to a lax symmetric monoidal one by showing that $\mathcal{L}$ enhances to a symmetric monoidal functor using the Day convolution structure on $\presheaf(\cat\perfect)$. By Lurie's criterion above, it suffices to show that $y(\A)\otimes[y(\D)/y(\sC)\rightarrow y(\E)]\simeq [y(\A\otimes \D)/y(\A\otimes \sC)\rightarrow y(\A\otimes \E)]$. Given this, since applying $\A\otimes-$ preserves split Karoubi sequences, the right--hand side (and hence also the left--hand side) is a motivic equivalence, as required.

Coming back to the parametrised setup,  it turns out that one may also make sense of the notion of a $\baseCat$--category of $\baseCat$\textit{--noncommutative motives}  using $\underline{\cat}\Tperfect_{\baseCat}$. In fact, in \cref{subsection:noncommutativeMotives}, we construct \textit{two} variants $\underline{\nmot}\pointwise_{\baseCat}$ and $\underline{\nmot}^{\mathrm{nm}}_{\baseCat}$ of $\baseCat$--noncommutative motives called the pointwise and normed variants, respectively.\footnote{Where the normed variant requires \cref{hypothesis:normed_spectra} on $\baseCat$ which is satisfied by $\orbit_G$.} As suggested by the name, we prove in \cref{monoidalityOfTMotives} that the functor $\Z_{\mathrm{nm}}\colon \underline{\cat}\Tperfect_{\baseCat}\rightarrow\underline{\nmot}^{\mathrm{nm}}_{\baseCat}$ refines to a $\baseCat$--symmetric monoidal functor. On the other hand, $\underline{\nmot}^{\mathrm{pw}}_{\baseCat}$ is more related to K--theory since we have:

\begin{alphThm}[See {\cref{univPropStableMotives,MotivicCorepresentability}}]\label{alphaThm:universalPropertyMotives}
    For any ${\baseCat}$-presentable-stable category $\underline{\E}$, the precomposition
$\Z^*_{\mathrm{pw}} \colon \underline{\func}^L_{\baseCat}(\underline{\nmot}_{\baseCat}^{\mathrm{pw}}, \underline{\E}) \rightarrow \underline{\func}_{\baseCat}\add(\underline{\cat}\Tperfect_{\baseCat}, \underline{\E})$ is an equivalence, where $\underline{\func}_{\baseCat}\add(\underline{\cat}\Tperfect_{\baseCat}, \underline{\E})\subseteq \underline{\func}_{\baseCat}(\underline{\cat}\Tperfect_{\baseCat}, \underline{\E})$ denotes the full subcategory of additive functors, i.e. those that send split Karoubi sequences to fibre sequences. Moreover, we have the  factorisation 
$\underline{\KTheory}_{\baseCat}\colon \underline{\cat}\Tperfect_{\baseCat}\xlongrightarrow{\Z_{\mathrm{pw}}}\underline{\nmot}\pointwise_{\baseCat} \xlongrightarrow{\myuline{\mapsp}(\Z_{\mathrm{pw}}\myuline{\spectra},-)} \myuline{\spectra}_{\baseCat}$
of $\baseCat$--parametrised algebraic K--theory.
\end{alphThm}
\noindent By the universal property from \cref{alphaThm:universalPropertyMotives}, we obtain a  comparison map $\Psi\colon \underline{\nmot}_{\baseCat}^{\mathrm{pw}}\rightarrow \underline{\nmot}_{\baseCat}^{\mathrm{nm}}$. This map is an equivalence if and only if the functor $\Z_{\mathrm{pw}}$ is compatible with the $\baseCat$--symmetric monoidal structures in a sense analogous to the unparametrised situation sketched above.\footnote{This criterion is given for example in \cite[Lem. 5.27]{quigleyShahParametrisedTate}; see also \cref{TDwyerKanSymmetricMonoidality,parametrisedPresentableDwyerKanLocalisation} where we provide a different proof.} If this happens, then  $\underline{\KTheory}_{\baseCat}$ refines to a $\baseCat$--lax symmetric monoidal functor. 

\vspace{1mm}

Unfortunately, this is where the breezy transferability from the nonequivariant setting to the equivariant one ends.  The core issue in this setting is the lack of a \textit{currying} manoeuvre for $G$--tensor products: nonequivariantly, we were able to simplify  the sufficient condition for symmetric monoidality of localisations by \textit{currying} and separating the problem into each tensor component. This led to a rather easy check to  enhance the motivic localisation $\Z$ multiplicatively. With $G$--tensor products, we have no such luxury since tensoring an object $G$--times and remembering the permutation $G$--equivariant structure inextricably links the tensor components. This forces us to deal with tensor powers of localisation equivalences head--on, leading naturally to the phenomenon of \textit{K--theoretic cubical descent}, as we now explain.

Given two split Karoubi sequences  $\{\sC_i\hookrightarrow \D_i\twoheadrightarrow \E_i\}_{i=0,1}$ in $\cat\perfect$, we would like to show that the tensored map 
$y(\D_1)/y(\sC_1)\otimes y(\D_2)/y(\sC_2)\longrightarrow y(\E_1)\otimes y(\E_2)$
is still a motivic equivalence. Since $y(-)$ was symmetric monoidal and $-\otimes-$ commutes with colimits in each variable, this map is equivalent to the composite
\begin{equation}\label{eqn:factoringPowersOfMotivicEquivalences}
    \begin{tikzcd}
        \frac{y(\D_1\otimes \D_2)}{y(\D_1\otimes \sC_2)\cup_{y(\sC_1\otimes \sC_2)}y(\sC_1\otimes \D_2)}\rar & \frac{y(\D_1\otimes \D_2)}{y(\D_1\otimes \sC_2\cup_{\sC_1\otimes \sC_2}\sC_1\otimes \D_2)}\rar & y(\E_1\otimes \E_2).
    \end{tikzcd}
\end{equation}
This breaks up the problem  into two parts: (a) if the induced map $\D_1\otimes \sC_2\cup_{\sC_1\otimes \sC_2}\sC_1\otimes \D_2\rightarrow \D_1\otimes \D_2$ is fully faithful, then by general nonsense,  $\D_1\otimes \sC_2\cup_{\sC_1\otimes \sC_2}\sC_1\otimes \D_2\hookrightarrow \D_1\otimes \D_2\twoheadrightarrow \E_1\otimes \E_2$ will again be a split Karoubi sequence. In this case,  the second map in  \cref{eqn:factoringPowersOfMotivicEquivalences} is a motivic equivalence; (b) if  the functor $\Z \colon \cat\perfect\rightarrow\nmot$ preserves pushouts of the form
    \begin{center}
        \begin{tikzcd}
            \sC_1\otimes \sC_2\rar[hook]\dar[hook] \ar[dr, phantom ,very near end, "\ulcorner"] & \sC_1\otimes \D_2\dar\\
            \D_1\otimes \sC_2 \rar & \D_1\otimes \sC_2\cup_{\sC_1\otimes \sC_2}\sC_1\otimes \D_2
        \end{tikzcd}
    \end{center}
    then the first map in \cref{eqn:factoringPowersOfMotivicEquivalences} will be seen to be a motivic equivalence.
Analogous questions involving cubes of higher dimensions can be formulated with more than two split Karoubi sequences.  As such, upon forgoing currying manoeuvres, endowing  multiplicative structures on K--theory  is strongly related  to  certain descent questions against special types of cubes as in problem (b). 
\vspace{1mm}

Motivated by these questions, we initiate in \cref{section:theoryOfParametrisedCubes} the study of \textit{parametrised cubes}.  More concretely, said cubes are parametrised categories obtained by taking finite indexed products of $\Delta^1$. The hypothesis that $\baseCat$ is atomic orbital will be used in an essential way in order to define the ``singletons'' in a parametrised cube. As far as we know, a cubical theory in this level of generality is new and could be of independent interest, cf. for example \cite{kaifSil} where they were used to develop a theory of parametrised Goodwillie calculus.

The guiding example   will be that of $C_2$\textit{--pushouts}, assuming that $G=C_2$ for simplicity. Namely, suppose in the setting of problem (b) above, instead of tensoring two different split Karoubi sequences, we took its $C_2$--norm instead. As in (b), this  induces a $C_2$--colimit diagram 
\begin{equation}\label{eqn:introC_2Pushouts}
    \begin{tikzcd}
        \norm^{C_2}_e\sC \ar[dr, very near end, "{\ulcorner}",phantom]\rar[hook]\dar[hook] & \sC\otimes \D \dar\\
        \D\otimes\sC \rar \ar[ur,no head, dashed]& \D\otimes\sC\underline{\cup}_{\norm^{C_2}_e\sC}\sC\otimes\D
    \end{tikzcd}
\end{equation}
in $\cat\Tperfect_{C_2}$ which we term as a $C_2$--pushout diagram. In \cref{eqn:introC_2Pushouts}, the top left and bottom right terms are $C_2$--objects in $\underline{\cat}\Tperfect_{C_2}$, but the other two terms are merely objects in $\cat\perfect$. The point is that the group $C_2$ acts on the \textit{entire} diagram  by swapping $\sC\otimes\D$ with $\D\otimes\sC$.   For example, in $C_2$--spaces, $C_2$--pushouts yield the \textit{sign suspension} $\Sigma^{\sigma}$ of $C_2$--spaces. 

The idea for the $C_2$--equivariant case now is to show that the square descent in K--theory induces descent with respect to $C_2$--colimits of the form \cref{eqn:introC_2Pushouts} by  re--expressing  the $C_2$--pushout diagram \cref{eqn:introC_2Pushouts} into an \textit{ordinary} pushout diagram of $C_2$--objects in $\cat\Tperfect_{C_2}$. We then show that additive functors satisfy descent with respect to such squares by standard methods, and so solve the $C_2$--analogue \cref{eqn:introC_2Pushouts} of problem (b) above. Together with this, a d\'{e}vissage--type argument via the solvability of $p$--groups (in the case $p=2$) then yields the following main theorem of this article, providing in the special case of $G$ being a 2--group the desired refinement of the Green functor  structure from \cite{barwick2} to $G$--$\mathbb{E}_{\infty}$--algebras.

\begin{alphThm}[See {\cref{mainTheorem:pointwiseEqualsNormedFor2-Groups,mainCorollary:K-theoryOf2-Groups}}]\label{alphaThm:pointwiseEqualsNormedFor2-Groups}
Let $G$ be a group with $|G|=2^n$ for some $n$. The comparison map $\fanpsi : \underline{\nmot}_G\pointwise \rightarrow \underline{\nmot}_G^{\mathrm{nm}}$ in this case is an equivalence. Consequently, $\underline{\KTheory}_G\colon \underline{\cat}\Tperfect_G\rightarrow \myuline{\spectra}_G$ canonically refines to a $G$--lax symmetric monoidal functor for such $G$ and so induces the functor $\underline{\KTheory}_G\colon \calg_G(\underline{\cat}\Tperfect_G)\rightarrow \calg_G(\myuline{\spectra}_G)$.
\end{alphThm}

\noindent This in particular means that for such groups, the equivariant K--theory spectrum $\{\KTheory(\sphere_H)\}_{H\leq G}$ of the equivariant sphere spectrum canonically assembles to a normed $\mathbb{E}_{\infty}$--ring spectrum in $\spectra_G$. Lastly, a straightforward combination of \cref{alphaThm:monoidalBorelificationPrinciple,alphaThm:pointwiseEqualsNormedFor2-Groups} gives us the following:

\begin{alphCor}[See {\cref{penultimateGLaxSymmetricMonoidal}}]\label{alphaCor:SwanKTheory}
    Let $G$ be a group with $|G|=2^n$ for some $n$ and $\sC^{\otimes}$ be a small symmetric monoidal perfect--stable category  with a $G$--action. Then $\{\KTheory(\sC^{hH})\}_{H\leq G}$ naturally assemble to a spectral $G$--Mackey functor equipped with multiplicative norms. 
\end{alphCor}

The  equivariant K--theory spectra considered in \cref{alphaCor:SwanKTheory} represents an extremely interesting class of examples and are sometimes called Swan K--theory. Classically, its importance has been recognised as early as Swan's groundbreaking work \cite{swan} together with its axiomatisation and hermitian elaboration \cite{dress}. More recently, they have also been considered, for instance, in \cite[$\S8$]{barwick2} and \cite{merlingMalkiewich}. \\

\textbf{Relation to other work.} Most results in this paper are corrections and expansions of Chapters 2--4 of the author's PhD thesis \cite{kaifThesis}. Equivariant algebraic K-theory is not a new subject and much work has been done in this area, see for example \cites{guillou2021multiplicative, lenz2021gglobal, merling2016equivariant, schwede2019global}. A slightly over-simplified but helpful view is that there are two versions of higher algebraic K-theory: on the one hand, there is the \textit{group-completion} K\textit{-theory}, classically related to Quillen's +-construction, cf. \cites{GGN} for an $\infty$-categorical treatment; on the other hand, there is the Quillen/Segal/Waldhausen K{-theory} whose input is a small stable $\infty$-category,  corresponding to Quillen's Q-construction and Segal and Waldhausen's $\mathrm{S}_{\bullet}$-construction. All the literature cited above dealt with the equivariant enhancement of the group-completion K-theory. In this paper, we treat the latter version of K-theory, and is a further refinement of the multiplicative structures treated in \cites{barwick2,CMMN2} to include the multiplicative norms. 

Before our work, Elmanto and Haugseng \cite{elmantoHaugseng} have shown that for \textit{all} finite groups $G$, the equivariant algebraic K--theory \textit{space} enhances to one with norms, using as a key input the deep result on K--theoretic power operations in \cite{barwickMathewNikolausPolynomiality}. Subsequently, after the appearance of our article, the work of Cnossen, Haugseng, Lenz, and Linskens \cite{cnossennorms} on bispan functors in particular  upgraded Elmanto--Haugseng's result to spectral norms. 

Our methods, which are independent of \cite{barwickMathewNikolausPolynomiality}, are  distinct from theirs. Especially in light of recent breakthroughs in studying localising motives by Efimov, we think that the motivic universal property of equivariant K--theory is an interesting point in its own right. For instance, it  guarantees that the multiplicative norm structure on equivariant algebraic K--theory is unique and it can be used to construct a highly structure equivariant Dennis trace map. In joint work \cite{hilmanRamzi} with Maxime Ramzi building on the foundations and the general approach herein, we construct the category of equivariant \textit{localising} motives associated to nonconnective algebraic K--theory and show that these admit multiplicative norms for all finite groups $G$. We then also deduce from this the case of connective algebraic K--theory.

Finally,  in the intervening years since this work first appeared, we have noticed a potential gap in the parametrised categories literature as well as one oversight  in an earlier version of this work. Namely, on the one hand, we were unfortunately not able to satisfy ourself that the proof of  \cite[Prop. 3.27]{nardinThesis} guaranteeing that the parametrised category of parametrised genuine spectra is parametrised idempotent (and we do not know if it should hold in general); on the other hand, we have also glossed over the fact that there might be a problem in passing the multiplicative norms structure from the presentable setting to the perfect--stable setting since it could be possible  that   norms of a left adjoint functor preserving compact objects can fail to preserve compact objects. We have singled out the assumptions needed in light of these points in \cref{hypothesis:normed_spectra,hypothesis2:norm_finiteness} respectively. 

Interdependently with \cite{hilmanRamzi,brankoKaifNatalie}, we settle these issues in the  equivariant situation which is, in any case, the main goal of the present paper. In more detail, using the theory of parametrised cubes in \cref{section:theoryOfParametrisedCubes} which is orthogonal to and independent of \cite[Prop. 3.27]{nardinThesis}, we prove in \cite{brankoKaifNatalie} that the base category $\baseCat=\orbit_G$ does indeed satisfy the idempotence property of \cref{hypothesis:normed_spectra}. We also show in \textit{loc. cit.} that \cref{hypothesis2:norm_finiteness} is satisfied when $\baseCat=\orbit_G$ by invoking  the purely combinatorial statement of equivariant finiteness of equivariant cubes proved in \cite[App. C]{hilmanRamzi}. \\

\textbf{Outline of paper.} Since our K--theoretic goals will require a fair bit of parametrised machinery, we will take the opportunity in  \cref{section:aspectsOfParametrisedCategories} to contribute to what may be classified as general parametrised theory where we will in particular prove \cref{alphaThm:monoidalBorelificationPrinciple,alphaThm:mainMackeyInclusion}. Next, we introduce and develop the basics of the theory of parametrised cubes in \cref{section:theoryOfParametrisedCubes}. In the final \cref{section:twoVariantsK-theory}, we will apply all the theory above to construct the parametrised version of noncommutative motives and prove \cref{alphaThm:universalPropertyMotives,alphaThm:pointwiseEqualsNormedFor2-Groups,alphaCor:SwanKTheory}. Each section will be prefaced with a more detailed outline of its contents.\\

\textbf{Acknowledgements.} I am grateful to Jesper Grodal, Markus Land,  Emanuele Dotto, Maxime Ramzi,  Asaf Horev, and Sil Linskens for useful comments, sanity checks, and many hours of enlightening conversations. Special thanks are due to Marc Hoyois for catching a serious mistake in the first version, around which much of the work in this revised version is based, and to Greg Arone who first suggested that one might be able to re--express equivariant pushouts in terms of ordinary pushouts, which proved to be the decisive technique driving our main result. We thank also Sil Linskens for reading a draft of a version of the article as well as for the expositional improvements and minor corrections suggested. This article is based to a large extent on work done in the author's PhD thesis \cite{kaifThesis} which was supported by the Danish
National Research Foundation through the Copenhagen Centre for Geometry and Topology
(DNRF151) as well as by the Swedish Research Council (grant no. 2016-06596) through the research program “Higher algebraic structures in algebra, topology and geometry” held at Institut Mittag–Leffler, Sweden in the spring of 2022. Furthermore, substantial improvements and revision work  have  been carried out at the Max Planck Institute for Mathematics in Bonn, Germany and the University of Bonn, and I was also supported by the European Research Council (ERC) under Horizon Europe (GeoCats, grant No.~101042990 and BorSym, grant No.~101163408).

\section{Aspects of parametrised category theory}\label{section:aspectsOfParametrisedCategories}

This section pertains to supplying miscellaneous  results in parametrised higher category theory. We  recollect the foundations of the parametrised theory that we shall need in \cref{subsection:basicSetup} as well as take the opportunity to prove some basic categorical generalities in \cref{subsection:miscelleneousPreliminaries}; in \cref{subsection:symmetricMonoidalityLocalisations} (from which point on we will always assume that the base category is atomic orbital), we will prove various elements in the interaction between localisations and multiplicative structures; next, we will elucidate in \cref{subsection:BorelEquivariantTheory} the $G$--symmetric monoidal theory associated to so--called ``Borel'' $G$--categories, proving \cref{alphaThm:monoidalBorelificationPrinciple}; we then introduce and study in \cref{sec4:perfectStables} the ``internal'' notion of perfect--stable categories in the parametrised setting, proving \cref{alphaThm:mainMackeyInclusion}.

\subsection{Basic setup}\label{subsection:basicSetup}

We provide here an overview of the basic theory. The original sources for these are \cite{parametrisedIntroduction,expose1Elements,shahThesis,shahPaperII,nardinExposeIV,nardinThesis,nardinShah,kaifPresentable}. For a one--stop survey for many of the basic theory, see for instance \cite{kaifThesis}. Expert readers should feel free to skip this subsection.

\begin{defn}[$\baseCat$--categories]
    Let $\baseCat$ be a small category. The category of $\baseCat$\textit{--categories} $\cat_{\baseCat}$ is defined to be $\func(\baseCat\op,\cat)$. An object in $\cat_{\baseCat}$ will be indicated with the underline notation $\underline{\sC}$. Under Lurie's straightening--unstraightening equivalence $\func(\baseCat\op,\cat)\simeq \cocartesianCategory(\baseCat\op)$, we will denote by $\int\underline{\sC}\rightarrow\baseCat\op$ the cocartesian unstraightening of $\underline{\sC}\in\func(\baseCat\op,\cat)$. Morphisms in $\cat_{\baseCat}$ are called $\baseCat$\textit{--functors}.
\end{defn}

\begin{terminology}[Objects]
    By an \textit{object} in a parametrised category $\underline{\sC}$, we will mean a $\baseCat$--functor $\terminalTCat\rightarrow\underline{\sC}$ where $\terminalTCat\in\cat_{\baseCat}$ is the terminal $\baseCat$--category which is constant with value $\ast$. For a fixed $V\in\baseCat$, writing $v\colon \baseCat_{/V}\rightarrow \baseCat$ for the canonical functor, we may then define a $V$--object in $\underline{\sC}$ to be a $\baseCat$--functor $v_!\terminalTCat\rightarrow \underline{\sC}$. By adjunction, this is the same datum as a $\baseCat_{/V}$--functor $\terminalTCat\rightarrow v^*\underline{\sC}$. Hence, the datum of a $V$--object in $\underline{\sC}$ is the same datum as an object (in the sense of the first sentence above) in $v^*\underline{\sC}$.
\end{terminology}

\begin{nota}
    For a morphism $f\colon W \rightarrow V$ in $\baseCat$ and $\underline{\sC}\in\cat_{\baseCat}=\func(\baseCat\op,\cat)$, we will write $f^*\colon \sC_W\rightarrow \sC_V$ for the structure map encoded by $\underline{\sC}$.
\end{nota}

\begin{nota}
    We will write $(-)\vop\colon \cat_{\baseCat}\xrightarrow{\simeq}\cat_{\baseCat}$ for the self--equivalence induced by applying $\func(\baseCat\op,-)$ to the self--equivalence $(-)\op\colon \cat\xrightarrow{\simeq}\cat$.
\end{nota}

\begin{cons}[Cofree parametrisation, {\cite[Thm. 2.8]{nardinExposeIV}}]\label{cofreeParametrisations}
Let $\D$ be a category. Then there is a $\baseCat$-category $\cofree_{\baseCat}(\D)$ classified by the functor $\baseCat\op \rightarrow \catinf$ given by $V \mapsto \func((\baseCat_{/V})\op, \D)$ called the \textit{cofree $\baseCat$--category on $\D$}. This has the following universal property: if $\underline{\sC}\in\cat_{\baseCat}$, then there is a natural equivalence $\func_{\baseCat}(\underline{\sC}, \cofree_{\baseCat}(\D)) \simeq \func(\int\underline{\sC}, \D)$ of unparametrised categories, where $\func_{\baseCat}$ denotes the category of $\baseCat$--functors.
\end{cons}

For underlined objects, we will often omit the subscripts $\baseCat$ for readability.

\begin{example}[Spaces and categories]
    An important example is the $\baseCat$\textit{--category of $\baseCat$--spaces} $\underline{\spc}$, defined as $\cofree(\spc)$. There is then a parametrised mapping space $\myuline{\map}(-,-)$ functor landing in $\underline{\spc}$ associated to any $\baseCat$--category which induces a parametrised Yoneda embedding satisying the usual universal property of presheaves. We refer the reader to \cite[$\S10$]{shahThesis} for more details. Similarly, we define the $\baseCat$\textit{--category of $\baseCat$--categories} $\underline{\cat}$ as $\cofree(\cat)$.
\end{example}

\begin{terminology}
    A $\baseCat$--functor $\underline{\sC}\rightarrow \underline{\D}$ is said to be $\baseCat$\textit{--fully faithful} if it is so fibrewise. Via the notion of parametrised mapping spaces from the example above, this can also be formulated as saying the $\baseCat$--functor induces equivalences on the parametrised mapping spaces.
\end{terminology}

\begin{recollect}[Adjunctions]
    Since $\cat_{\baseCat}$ is naturally a 2--category, we also have a notion of $\baseCat$--adjunctions, which one can also describe concretely in the unstraightened perspective using the notion of relative adjunctions from \cite[$\S7.3.2$]{lurieHA}, cf. \cite[$\S8$]{shahThesis}. By \cite[Lem. 2.2.9]{kaifPresentable} for example, this may also be phrased more internally as the datum of a natural equivalence $\myuline{\map}_{\underline{\D}}(L-,-)\simeq \myuline{\map}_{\underline{\sC}}(-,R-)\colon \underline{\sC}\vop\times \underline{\D}\longrightarrow \underline{\spc}$.
\end{recollect}

\begin{nota}[Cotensors]\label{parametrisedCotensors}
There is an internal hom functor $\underline{\func}$ (also written as $\underline{\func}_{\baseCat}$ when we want to be explicit with the base category we are working over) equipped with natural equivalence $\func_{\baseCat}(-\times\underline{\sC},-)\simeq\func_{\baseCat}(-,\underline{\func}(\underline{\sC},-))$ for any $\underline{\sC}\in\cat_{\baseCat}$.  Write $p\colon \baseCat\op\rightarrow \ast$ for the unique functor and let $I$ be a small unparametrised category. Then the adjunction $-\times I : \cat \rightleftarrows \cat : \func(I,-)$ induces the adjunction
$(-\times I)_* : \func(\baseCat\op,\cat) \rightleftarrows \func(\baseCat\op, \cat) : \func(I,-)_*$.
Under the identification $\func(\baseCat\op,\cat)= \cat_{\baseCat}$, it is clear that $(-\times I)_*$ corresponds to the $\baseCat$-functor $-\times p^*I$, whose right adjoint  is $\underline{\func}(p^*I, -)$. Therefore $\underline{\func}(p^*I, -)\simeq \func(I,-)_*$ implements the \textit{fibrewise functor construction}. We will often write $\underline{\func}(I,-)$ for $\underline{\func}(p^*I, -)$. This satisfies the following properties whose proofs are immediate.
\begin{enumerate}
    \item $\underline{\cat}_{\baseCat}$ is cotensored over $\cat$ in the sense that for any $\baseCat$-categories $\underline{\sC},  \underline{\D}$ we have 
\[\underline{\func}(\underline{\sC},  \underline{\func}(I, \underline{\D})) \simeq \underline{\func}(I, \underline{\func}(\underline{\sC},  \underline{\D}))\]
\item $\underline{\func}(I, -)$ preserves $\baseCat$-adjunctions. This is straightforward to deduce from \cite[$\S8$]{shahThesis}.
\end{enumerate}
\end{nota}

\begin{obs}\label{obs:leftKanExtensionAsA2Adjunction}   Suppose $\baseCat$ has a final object and let $s\colon \ast\hookrightarrow \baseCat$ be      the inclusion of the final object, so that upon passing to the opposites, we have the adjunction $s\colon \ast\rightleftharpoons \baseCat\op : p$. Using that  $s_!\colon \cat\rightarrow\cat_{\baseCat}$ preserves products  since $s_!\simeq p^*$ is a right adjoint by virtue of the adjunction $s\dashv p$, we also have an equivalence $s^*\underline{\func}(s_!I,-)\simeq \func(I,s^*-)\colon \cat_{\baseCat}\rightarrow \cat$.
\end{obs}

\begin{recollect}[(Co)limits and indexed (co)products]\label{recollect:indexedCoproductsAndFibrewiseColimits}
    The notion of parametrised adjunctions and $\underline{\func}$ afford us the key concept of \textit{parametrised (co)limits}. That is, for any $\underline{I}, \underline{\sC}\in\cat_{\baseCat}$ and writing $\pi\colon \underline{I}\rightarrow\terminalTCat$ for the unique $\baseCat$--functor, the $\underline{I}$\textit{--shaped (co)limit in} $\underline{\sC}$ functor, if it exists, may be defined as the parametrised right (resp. left) adjoint $\pi_*$ (resp. $\pi_!$) to the restriction functor $\pi^* \colon \underline{\sC}\rightarrow\underline{\func}(\underline{I},\underline{\sC})$.  Be warned that these are \textit{not} given by fibrewise taking (co)limit, although this \textit{is} so when the indexing shape is a constant $\baseCat$--category (i.e. those of the form $p^*I$ as in \cref{parametrisedCotensors}). Furthermore, one can also develop the notion of parametrised Kan extensions etc., and we refer the reader to \cite[$\S\S9, 10$]{shahThesis} for more  details.

    A very important part of the general theory that we will use often in our arguments later is that every parametrised colimit may be decomposed into an unparametrised part and a truly parametrised part. To add precision to this, it would be helpful first to recall some standard terminologies: it is common in the literature to term as \textit{fibrewise (co)limits} those parametrised (co)limits which are indexed by a constant diagram (in the sense explained above). On the other hand, for any fixed $V\in\baseCat$, by basechanging from $\baseCat$ to $\baseCat_{/V}$ (i.e. by considering the functor $\cat_{\baseCat}\rightarrow \cat_{\baseCat_{/V}}$ induced by the canonical functor $\baseCat_{/V}\rightarrow\baseCat$), we may without loss of generality assume that $V$ was a final object in $\baseCat$. In this case, for any $U\in\baseCat$, writing  $f\colon U\rightarrow
     V$ in $\baseCat$ for the unique map, we also write $f\colon f_!f^*\underline{\ast}\rightarrow\terminalTCat$ for the unique map of $\baseCat$--categories (note the intentional abuse of the notation $f$). Now for any $\underline{\sC}\in\cat_{\baseCat}$, we shall term the left (resp. right) adjoint $f_! \:\mathrm{ (resp. } \:f_*\mathrm{)}\colon \underline{\func}(f_!f^*\terminalTCat,\underline{\sC})\simeq f_*f^*\underline{\sC}\rightarrow \underline{\sC}$ to the functor $f^*$, if it exists, as the \textit{$f$--indexed coproduct (resp. product)}. These (co)limits play a distinguished role in the parametrised theory: akin to the unparametrised setting where every colimit can be rewritten as a geometric realisation all of whose terms are coproducts, we know by \cite[$\S12$]{shahThesis} that any parametrised colimit may be rewritten as a fibrewise geometric realisation all of whose terms are indexed coproducts. The upshot of this is that we can often divide a proof into dealing with fibrewise (co)limits and indexed (co)products separately.
\end{recollect}

\begin{recollect}[Adjointed squares]
    We recall the notion of Beck--Chevalley transformations and adjointability from \cite[$\S7.3.1$]{lurieHTT}. Suppose we are given a commuting square
    \begin{center}
        \begin{tikzcd}
            \underline{\sC} \rar["\varphi"]\dar["L"'] & \widetilde{\underline{\sC}}\dar["\widetilde{L}"']\\
            \underline{\D} \rar["\psi"] & \widetilde{\underline{\D}}
        \end{tikzcd}
    \end{center}
    such that $L,\widetilde{L}$ admit $\baseCat$--right adjoints $R,\widetilde{R}$ respectively. We may then obtain a natural transformation $\varphi R\Rightarrow \widetilde{R}\psi$ via
    \[\varphi R \xlongrightarrow{\widetilde{\eta}_{\varphi R}} \widetilde{R}\widetilde{L}\varphi R\simeq \widetilde{R}\psi LR \xlongrightarrow{\widetilde{R}\psi\varepsilon} \widetilde{R}\psi\]
    This canonically constructed transformation is called the \textit{Beck--Chevalley transformation}. If this map is an equivalence (hence we get an equivalence $\varphi R\simeq \widetilde{R}\psi$), then we say that the original square is \textit{right adjointable}. Similarly, we may also define the notion of left adjointability.
\end{recollect}

\begin{prop}[Fibrewise criteria for $\baseCat$-adjunctions, {\cite[Cor. 2.2.7]{kaifPresentable}}]\label{criteriaForTAdjunctions}
Let $F : \underline{\sC} \rightarrow \underline{\D}$ be a $\baseCat$-functor. Then it admits a $\baseCat$-right adjoint if and only if it admits fibrewise right adjoints $G_V$ for all $V\in \baseCat$ and for all morphism $f : W \rightarrow V$ in $\baseCat$, the Beck--Chevalley square
\begin{center}
    \begin{tikzcd}
    \sC_W & \D_W \lar["G_W"']\\
    \sC_V \uar["f^*"]& \D_V\lar["G_V"]\uar["f^*"']
    \end{tikzcd}
\end{center}
commutes. Similarly for $\baseCat$-left adjoints.
\end{prop}

\begin{prop}[(Co)limit preservation, {\cite[{Prop. 2.4.2}]{kaifPresentable}}]\label{characterisationStrongPreservations}
Let $\underline{\sC},\underline{\D}$ be $\baseCat$-cocomplete categories and $F : \underline{\sC} \rightarrow \underline{\D}$ a $\baseCat$-functor. Then $F$  preserves $\baseCat$-colimits if and only if it preserves colimits in each fibre and for all $f : W \rightarrow V$ in $\baseCat$, the Beck--Chevalley square 
\begin{center}
    \begin{tikzcd}
    \sC_W \dar["f_!"]\rar["F_W"] &  \D_W \dar["f_!"]\\
    \sC_V\rar["F_V"] & \D_V
    \end{tikzcd}
\end{center}
commutes. Similarly for $\baseCat$-limits.
\end{prop}

\begin{defn}[{\cite[Def. 4.1]{nardinExposeIV}}]\label{DefinitionAtomicOrbital}
Let $\baseCat$ be a small category. We say that it is \textit{atomic} if whenever we have $f : W\rightarrow V$ and $g : V \rightarrow W$ in ${\baseCat}$ such that $g\circ f$ is an equivalence, then $f$ and $g$ were already inverse equivalences. And we say that it is \textit{orbital} if the finite coproduct cocompletion $\finite_{\baseCat}$ admits finite pullbacks. Here, by finite coproduct cocompletion, we mean the full subcategory of the presheaf category $\func({\baseCat}\op,\spc)$ spanned by finite coproduct of representables. See \cref{recollect:finiteTSets} for more details on $\finite_{\baseCat}$.
\end{defn}

\begin{example}
    The orbit category $\orbit_G$ for a finite group $G$ is atomic orbital.
\end{example}

\begin{obs}
    It is straightforward to argue by unwinding the definitions that atomic orbitality ensures that for any map $f\colon U \rightarrow V$ in $\baseCat$, the orbital decomposition of the pullback $U\times_VU$ contains a copy of $U$. That is, we have the pullback diagram
    \begin{equation}\label{eqn:complementationSquare}
        \begin{tikzcd}
            U \coprod Z \rar["\id\sqcup c"]\dar["\id\sqcup \overline{c}"']\ar[dr, phantom, very near start, "\lrcorner"] & U\dar["f"]\\
            U\rar["f"] & V
        \end{tikzcd}
    \end{equation}    
    This is the key property enjoyed by atomic orbital categories that is crucial for algebraic considerations such as  Nardin's definition of $\baseCat$--semiadditivity (which we will recall shortly) as well as the theory of parametrised cubes which we introduce in \cref{section:theoryOfParametrisedCubes}.
\end{obs}

\begin{obs}\label{obs:topFixedPointsOfIndexedCoproducts}
   Suppose $\baseCat$ is atomic  and has a final object and let $s\colon \ast\hookrightarrow \baseCat$ be      the inclusion of the final object $T\in\baseCat$, and $w\colon W \rightarrow T$ be a map in $\baseCat$ which is not an equivalence. For an arbitrary $\underline{\sC}\in\cat_{\baseCat_{/W}}$, we claim that there is a natural equivalence
   $s^*w_!\underline{\sC}\simeq \emptyset\in\cat$. To see this, note by the usual pointwise right Kan extension formula that, for $\D\in\cat$, $s_*\D\in\cat_{\baseCat}$ has value $\D$ at $T\in\baseCat$ and $\ast$ everywhere else. Hence, since $w$ was not an equivalence, by atomicity we know that there is no morphism $T\rightarrow W$ and so pulling back along $w\colon \baseCat_{/W}\rightarrow \baseCat$ yields $w^*s_*\D\simeq \terminalTCat\in\cat_{\baseCat_{/W}}$ for any $\D\in\cat$. Therefore, we obtain
   $\map_{\cat}(s^*w_!\underline{\sC},\D)\simeq \map_{\cat_{\baseCat_{/W}}}(\underline{\sC},w^*s_*\D)\simeq \map_{\cat_{\baseCat_{/W}}}(\underline{\sC},\terminalTCat)\simeq \ast$,
   and so $s^*w_!\underline{\sC}$ satisfies the universal property of the initial object in $\cat$.
\end{obs}

\begin{terminology}[Left/right Beck--Chevalley conditions]\label{beckChevalleyMeaning}
Let $\baseCat$ be an orbital category and let $\underline{\sC}$ be a $\baseCat$-category that admits finite fibrewise coproducts (resp. products) such that for each $f : W \rightarrow V$ in $\baseCat$, the pullback $f^* : \sC_V \rightarrow \sC_W$ admits a left adjoint $f_!$ (resp. right adjoint $f_*$). We say that $\underline{\sC}$ satisfies the \textit{left Beck-Chevalley condition} (resp. \textit{right Beck-Chevalley condition}) if for every pair of edges $f: W\rightarrow V$ and $g : Y\rightarrow V$ in $\baseCat$, if we write the pullback (whose orbital decomposition exists by orbitality of $\baseCat$) as 
        \begin{center}
            \begin{tikzcd}
            \coprod_a R_a = Y\times_VW\dar[dr, phantom, very near start, "\scalebox{1.5}{$\lrcorner$}"] \rar["\coprod_af_a"]\dar["\coprod_ag_a"'] & Y \dar["g"]\\
            W \rar["f"] & V
            \end{tikzcd}
        \end{center}
        then the canonical map 
        $\coprod_{a}g_{a_!}f^*_{a} \Longrightarrow f^{*}g_{!}$ (resp. $f^*g_* \Longrightarrow \prod_{a}g_{a*}f^{*}_{a}$) is an equivalence.
\end{terminology}

We now recall the algebraic aspects of parametrised higher category theory in the presence of the atomic orbitality assumption on $\baseCat$. These were first introduced and studied in \cite{nardinExposeIV,nardinThesis}, and later revisited with further developments in \cite{nardinShah}.

\begin{recollect}[Finite $\baseCat$--sets]\label{recollect:finiteTSets}
    For every $V\in\baseCat$, we may define the category $\finite_{/V}\subseteq \presheaf(\baseCat_{/V})$ given by the finite coproduct cocompletion of $\baseCat_{/V}$. When no $V\in\baseCat$ is specified, we write $\finite_{\baseCat}$ for the finite coproduct cocompletion of $\baseCat$. By general category theory, we know that given a map $f\colon U \rightarrow V$, the left Kan extension $f_!\colon \presheaf(\baseCat_{/U})\rightarrow\presheaf(\baseCat_{/V})$ restricts to a functor $f_!\colon \finite_{/U}\rightarrow\finite_{/V}$. And in our setting, by the orbitality assumption, the right adjoint $f^*\colon \presheaf(\baseCat_{/V})\rightarrow\presheaf(\baseCat_{/U})$ also restricts to a right adjoint $f^*\colon\finite_{/V}\rightarrow
     \finite_{/U}$. These assemble to a $\baseCat$--category $\underline{\finite}_{\baseCat}$. Similarly, we may construct the pointed version $\underline{\finite}_{*\baseCat}$ whose fibre over $V\in\baseCat$ is given by $(\finite_{/V})_{[V=V]/}$. For details, see for example {\cite[Def. 2.1.1]{nardinShah}}
\end{recollect}

Write $\cat_*\subset \cat$ for the non--full subcategory of pointed categories and morphisms the functors which preserve these.

\begin{defn}[Pointedness]
    A $\baseCat$--category $\underline{\sC}\in\cat_{\baseCat}$ is said to be $\baseCat$\textit{--pointed} if it lies in the non--full subcategory $\func(\baseCat\op,\cat_*)\subset \func(\baseCat\op,\cat)=\cat_{\baseCat}$. That is, it is a $\baseCat$--category all of whose fibres are pointed and such that the structure maps preserve the zero objects.
\end{defn}

\begin{recollect}[Semiadditivity norm maps, {\cite[Cons. 5.2]{nardinExposeIV}}]
    Let $\underline{\sC}$ be $\baseCat$--pointed and have finite $\baseCat$-coproducts, and $\underline{\D}$ have finite $\baseCat$-products. Let $F\colon \underline{\sC}\rightarrow\underline{\D}$ be a $\baseCat$--functor and $f\colon U \rightarrow V$ be a map in $\underline{\finite}_{\baseCat}$. We would like now to construct a canonical transformation 
    \begin{equation}\label{eqn:semiadditivityNormMap}
        F\circ f_! \Rightarrow f_*\circ f^*F \colon f_*f^*\underline{\sC}\longrightarrow \underline{D}
    \end{equation}
    called the \textit{semiadditivity norm map}. By atomic orbitality,  the pullback square \cref{eqn:complementationSquare} gives us a natural equivalence $f^*f_!\simeq \id\sqcup \overline{c}_!c^*$ and so since $\underline{\sC}$ was $\baseCat$--pointed, we may postcompose this with the map $\pi\coloneqq \id\sqcup 0\colon \id\sqcup\overline{c}_!c^*\rightarrow \id\sqcup 0\simeq \id$. The map \cref{eqn:semiadditivityNormMap} is then defined as  
    \[Ff_! \xlongrightarrow{\eta_{F{f_!}}} f_*f^*Ff_!\simeq f_*Ff^*f_!\simeq f_*F(\id\sqcup \overline{c}_!c^*) \xlongrightarrow{f_*F\pi} f_*F.\]
    A little unwinding of definitions shows that, when $U\simeq V\sqcup V$ and $f\colon U \rightarrow V$ is the fold map, the construction above specialises to the usual unparametrised canonical comparison map $\coprod\rightarrow \prod$  required to be an equivalence in the definition of  semiadditive categories. 
\end{recollect}

\begin{defn}[Semiadditivity and stability, {\cite[Def. 5.3, Def. 7.1]{nardinExposeIV}}]
Let $\underline{\sC}$ be $\baseCat$--pointed and have finite $\baseCat$-coproducts, and $\underline{\D}$ have finite $\baseCat$-products. Let $F\colon \underline{\sC}\rightarrow\underline{\D}$ be a $\baseCat$--functor. We say that it is $\baseCat$\textit{-semiadditive} if for all $f\colon U \rightarrow V$ in $\underline{\finite}_{\baseCat}$, the semiadditivity norm map constructed above is an equivalence. We say that a pointed $\baseCat$-category $\underline{\sC}$ with finite $\baseCat$--(co)products is \textit{$\baseCat$-semiadditive} if the identity functor is $\baseCat$-semiadditive. A $\baseCat$--semiadditive category is said to be $\baseCat$\textit{--stable} if it is furthermore fibrewise stable.
\end{defn}

We write $\underline{\func}^{\times}_{\baseCat}(\underline{\sC},  \underline{\D})$  for the $\baseCat$-full subcategories of $\underline{\func}_{\baseCat}(\underline{\sC}, \underline{\D})$ consisting of the finite $\baseCat$--product--preserving functors. The following statements about commutative monoids have been stated by Nardin in \cite{nardinExposeIV}, but we have chosen to cite \cite{CLLSpans} since the proofs in the former reference are rather sketchy.

\begin{defn}[Commutative monoids, {\cite[Def. 7.1]{CLLSpans}}]\label{nota:TCommutativeMonoidsDefinition}
For $\underline{\sC}$ with finite $\baseCat$-limits we will denote \textit{$\baseCat$--commutative monoids} by $\underline{\cmonoid}_{\baseCat}(\underline{\sC}) := \underline{\func}^{{\times}}_{\baseCat}(\myuline{\effBurn}(\underline{\finite}),\underline{\sC})$.
\end{defn}

\begin{rmk}
    By \cite[Thm. D]{philMasters} (which was implicitly claimed in \cite[Thm. 2.3.9]{nardinShah} without proof), we see that evaluating at any $V\in\baseCat$, we get an equivalence
    $\underline{\cmonoid}_{\baseCat}(\cofree(\sC))_V\simeq \func^{\times}(\effBurn(\finite_{/V}),\sC) =: \mackey_{\baseCat_{/V}}(\sC)$.
\end{rmk}

\begin{example}
    A key instance of this construction is the $\baseCat$\textit{--category of $\baseCat$--spectra}, defined as $\myuline{\spectra}_{\baseCat}\coloneqq \underline{\cmonoid}_{\baseCat}(\cofree_{\baseCat}(\spectra))$. By the preceding remark, we see that $\myuline{\spectra}_{\baseCat}$ is fibrewise given by spectral Mackey functors  $\mackey_{\baseCat_{/V}}(\spectra)$. Equivalently, $\myuline{\spectra}_{\baseCat}$ may also  be written as $\myuline{\spectra}^{\mathrm{fib}}\underline{\cmonoid}_{\baseCat}(\underline{\spc})$ where $\myuline{\spectra}^{\mathrm{fib}}$ is the fibrewise stabilisation functor.
\end{example}

\begin{cons}[Forgetful functor, {\cite[Def. 7.1]{CLLSpans}}]\label{cons:forgetfulFunctor}
    For all $V\in\baseCat$, $\baseCat_{/V}$ has a terminal object. Hence so does $\finite_{\baseCat_{/V}}$. Now, $\myuline{\effBurn}(\underline{\finite})(V)\simeq \effBurn(\finite_{\baseCat_{/V}})$. Hence, by picking out the terminal object from $\finite_{\baseCat_{/V}}$ in $\effBurn(\finite_{\baseCat_{/V}})$ for all $V\in\baseCat$, we obtain a $\baseCat$--point $\underline{\ast}\rightarrow \myuline{\effBurn}(\underline{\finite})$. Precomposing along this map give us a transformation $\underline{\cmonoid}_{\baseCat}(-)\Rightarrow \id\simeq \underline{\func}(\terminalTCat,-)$ of endofunctors on $\cat_{\baseCat,\underline{\prod}}$ which we call the forgetful functor $\forget$. Here, $\cat_{\baseCat,\underline{\prod}}$ is the non--full subcategory of $\cat_{\baseCat}$ whose objects are $\baseCat$--categories with finite indexed products and morphisms which preserve these.
\end{cons}

\begin{thm}[$\baseCat$--semiadditivisation, {\cite[Thm. 7.4]{CLLSpans}}] \label{semiadditivisation}
Let $\underline{\sC}$ be a $\baseCat$-category with finite $\baseCat$-products. Then the forgetful functor $\underline{\cmonoid}_{\baseCat}(\underline{\sC}) \rightarrow \underline{\sC}$  is an equivalence if and only if $\underline{\sC}$ was $\baseCat$-semiadditive.
\end{thm}

\begin{nota}
    For $\underline{\sC}\in\cat_{\baseCat}$, we write $\underline{\presheaf}(\underline{\sC})$ and $\underline{\presheaf}^{\underline{\mathrm{st}}}(\underline{\sC})$ for the presheaf categories $\underline{\func}(\underline{\sC}\vop,\underline{\spc})$ and $\underline{\func}(\underline{\sC}\vop,\myuline{\spectra})$, respectively.
\end{nota}

Moreover, Nardin in his thesis \cite{nardinThesis} has also introduced the notion of $\baseCat$--symmetric monoidal structures, upon which one may extract $\baseCat$--commutative algebra objects. In the case when $\baseCat=\orbit_G$, such commutative algebra objects encode precisely the multiplicative norms of \cite{greenleesMayMU,HHR}. Much like  the unparametrised notion from \cite{lurieHA}, the notion of $\baseCat$--operads was defined in \cite{nardinThesis} as certain fibrations over $\underline{\finite}_{*\baseCat}$ and   $\baseCat$--symmetric monoidal categories are then the $\baseCat$--operads such that this fibration is $\baseCat$--cocartesian. The $\baseCat$--commutative algebras $\myuline{\calg}_{\baseCat}(\underline{\sC}\totimes)$ of a $\baseCat$--symmetric monoidal category $\underline{\sC}\totimes$ is then defined to be the $\baseCat$--category of $\underline{\finite}_{*\baseCat}$--sections of $\underline{\sC}\totimes$ which are maps of $\baseCat$--operads. We refer the reader to \cite[$\S2$]{nardinShah} for a more recent account of this theory.

\begin{nota}
    As in usual symmetric monoidal structures which in particular supply us with a ``multiplication''  $-\otimes-\colon \sC\times \sC\rightarrow \sC$, we also have an indexed version of such maps. In more detail, for each map $f\colon U\rightarrow V$ in $\finite_{/V}$, we also have an indexed multiplication map
    \[f_{\otimes}\colon f_*f^*\underline{\sC}\longrightarrow \underline{\sC}\] associated to a $\baseCat$--symmetric monoidal structure $\underline{\sC}\totimes$ on $\underline{\sC}$. Here we have implicitly basechanged the $\baseCat$--category $\underline{\sC}$ to a $\baseCat_{/V}$--category, which we also write $\underline{\sC}$. This will be a convenience  employed throughout the document to lighten our notational burdens. 
\end{nota}

\begin{recollect}[Distributivity]
    Another notion that will be important in our work is that of \textit{distributivity} which generalises  the idea of tensor products which are bicocontinuous into the parametrised setting. This was first defined by Nardin in his thesis \cite{nardinThesis}, but see also \cite[Def. 3.2.3]{nardinShah}  where the theory is further developed. A $\baseCat$--symmetric monoidal structure $\underline{\sC}\totimes$ on a $\baseCat$--category with all $\baseCat$--colimits is said to be $\baseCat$--\textit{distributive} if the following holds: for any map $f\colon U\rightarrow V$ in $\finite_{/V}$ and a $\baseCat_{/U}$--colimit diagram $\partial \colon \underline{K}\tcocone\rightarrow f^*\underline{\sC}$, the diagram
    \[f_{\otimes}\partial\colon (f_*\underline{K})\tcocone\xrightarrow{\canonical} f_*(\underline{K}\tcocone) \xlongrightarrow{f_*\partial} f_*f^*\underline{\sC} \xlongrightarrow{f_{\otimes}} \underline{\sC}\]
    is a $\baseCat_{/V}$--colimit diagram in $\underline{\sC}$.
\end{recollect}

Crucial to our work will be two results by Nardin--Shah  which we collect here as:

\begin{thm}[{\cite[Thm. 2.3.9, Prop. 2.8.7]{nardinShah}, \cite[Thm. D]{philMasters}}] \label{thm:nardinShahEnvelopes}
Let $\underline{\cat}_{\baseCat}^{\otimes}$ be the $\baseCat$--category of $\baseCat$--symmetric monoidal categories and $\baseCat$--symmetric monoidal functors. Let $\underline{\sC}\totimes\in\cat_{\baseCat}^{\otimes}$. Then (1) there are equivalences $\underline{\cat}^{\otimes}_{\baseCat}\simeq \underline{\cmonoid}_{\baseCat}(\underline{\cat})$ and $\cat_{\baseCat}^{\otimes}\simeq \mackey_{\baseCat}(\cat)$, and (2) there is a $\baseCat$--functor $\underline{\mathrm{Env}}$ from $\baseCat$--operads to $\baseCat$--symmetric monoidal categories which participates in a natural equivalence $\func\totimes_{\baseCat}(\underline{\mathrm{Env}}(\underline{\finite}_*),\underline{\sC}\totimes)\simeq \calg_{\baseCat}(\underline{\sC}\totimes)$, where $\func\totimes_{\baseCat}$ is the category of $\baseCat$--symmetric monoidal functors.
\end{thm}

In particular, when $\baseCat=\orbit_G$, we see that the category $G$--symmetric monoidal categories is given by $\mackey_G(\cat)\simeq \cmonoid_G(\underline{\cat}_G)$, analogous to the fact that the category of symmetric monoidal categories may equivalently be described as $\cmonoid(\cat)$.

\begin{recollect}[Pointwise symmetric monoidal structures]\label{recollect:universalPropertyPointwiseMonoidal}
    Let $\underline{J}\in\cat_{\baseCat}$ and $\underline{\D}\totimes\in \underline{\cmonoid}_{\baseCat}(\underline{\cat})$. In the same way that one can equip the pointwise symmetric monoidal structure on the functor category $\func(I,\sC)$ for an arbitrary $I\in\cat$ and $\sC^{\otimes}\in\cmonoid(\cat)$, we can also construct a ${\baseCat}$--symmetric monoidal structure $\underline{\func}(\underline{J},\underline{\D}\totimes)$ on the ${\baseCat}$--functor category $\underline{\func}(\underline{J},\underline{\D})$ thanks to \cite[$\S3.3$]{nardinShah}. As in the unparametrised case, this construction enjoys the following cotensor universal property: for any $\underline{\sC}\totimes$, there is a natural equivalence
    \[\myuline{\map}_{\underline{\cmonoid}(\underline{\cat})}(\underline{\sC}\totimes,\underline{\func}(\underline{J},\underline{\D}\totimes))\simeq \myuline{\map}_{\underline{\cat}}(\underline{J},\underline{\func}\totimes(\underline{\sC}\totimes,\underline{\D}\totimes))\]
    where $\underline{\func}\totimes(\underline{\sC}\totimes,\underline{\D}\totimes)$ is the ${\baseCat}$--category of ${\baseCat}$--symmetric monoidal functors. We refer the reader to \cite[$\S2, \S3.3$]{nardinShah}, especially \cite[Thm. 3.3.3]{nardinShah}, for more details on this.
\end{recollect}

Next, we recall the notion of parametrised presentability.  This notion was first defined in \cite[$\S1.4$]{nardinThesis} and further developed in \cite{kaifPresentable}. A pleasant feature of the theory is that parametrised Ind--completions and parametrised accessibility are fibrewise notions (cf. \cite[$\S3.5, \S5.2$]{kaifPresentable}). For example, the $\kappa$--Ind--completion functor $\underline{\ind}_{\kappa}\colon \cat_{\baseCat}\rightarrow {\widehat{\cat}}_{\baseCat}$ is given just by applying $\func(\baseCat\op,-)$ to the usual functor $\ind_{\kappa}\colon \cat\rightarrow\widehat{\cat}$. Consequently, so are the notions of parametrised compactness and parametrised idempotent--completeness (cf. \cite[$\S5.1, \S5.3$]{kaifPresentable}). For instance, an object $X$ in $\underline{\sC}$ is said to be parametrised $\kappa$--compact if it is so fibrewise, and we showed in \cite[Prop. 5.1.4]{kaifPresentable} that this definition can equivalently be characterised by saying that the parametrised mapping space functor $\myuline{\map}_{\underline{\sC}}(X,-)$ preserves fibrewise (here fibrewise is in the sense of \cref{recollect:indexedCoproductsAndFibrewiseColimits}) $\kappa$--filtered colimits. Moreover, we have a ``Lurie--Simpson'' style characterisation of parametrised presentability:

\begin{thm}[Omnibus presentability, {\cite[Thm. 6.1.2, Prop. 6.3.3]{kaifPresentable}}]\label{simpsonTheorem}
Let $\baseCat$ be an orbital category and $\underline{\sC}$ a $\baseCat$-category. Then $\underline{\sC}$ is $\baseCat$-presentable if and only if $\underline{\sC}$ satisfies the left Beck-Chevalley condition (cf. \cref{beckChevalleyMeaning}) and there is a regular cardinal $\kappa$ such that the straightening $\underline{\sC} : \baseCat\op \longrightarrow \widehat{\cat}$ factors through  $\underline{\sC}\colon \baseCat\op \longrightarrow \presentable_{L, \kappa}$. Moreover, $\baseCat$--presentable categories are also $\baseCat$--complete.
\end{thm}

\begin{thm}[Parametrised adjoint functor theorem, {\cite[{Thm. 4.2.1}]{kaifPresentable}}]\label{parametrisedAdjointFunctorTheorem}
Let $F : \underline{\sC}\rightarrow \underline{\D}$ be a $\baseCat$-functor between $\baseCat$-presentable categories. If $F$  preserves $\baseCat$-colimits, then $F$ admits a $\baseCat$-right adjoint; If $F$  preserves $\baseCat$-limits and is $\baseCat$-accessible, then $F$ admits a $\baseCat$-left adjoint.
\end{thm}

\begin{nota}
    For a fixed regular cardinal $\kappa$, we write $\underline{\presentable}_{\baseCat, L,\kappa}\subset \underline{\widehat{\cat}}$ for the non--full $\baseCat$--subcategory of $\kappa$--accessible parametrised presentable categories and morphisms the left adjoint $\baseCat$--functors which preserve $\kappa$--compact objects; we write $\underline{\presentable}_{\baseCat, R,\kappa\text{-filt}}\subset \underline{\widehat{\cat}}$ for the non--full subcategory of $\kappa$--accessible parametrised presentable categories and morphisms the right adjoint $\baseCat$--functors which preserve $\kappa$--filtered colimits. Furthermore, we also write $\underline{\cat}_{\baseCat}^{\underline{\mathrm{idem}}(\kappa)}\subset \underline{\cat}_{\baseCat}$ for the non--full $\baseCat$--subcategory of small parametrised-idempotent-complete $\baseCat$--categories which are parametrised-$\kappa$-cocomplete and morphisms the functors which  preserve $\kappa$-small parametrised colimits.
\end{nota}

\begin{prop}[{\cite[Thm. 4.5.3]{kaifPresentable}}]\label{TPresentableIdempotentCorrespondence}
Let $\kappa$ be a regular cardinal.  Then we have an equivalence of $\baseCat$--categories
$(-)^{\underline{\kappa }} : \underline{\presentable}_{\baseCat, L, \kappa} \rightleftarrows \underline{\cat}_{\baseCat}^{\underline{\mathrm{idem}}(\kappa)} : \underline{\ind}_{\kappa}$
where $(-)^{\underline{\kappa }}$ and $\underline{\ind}_{\kappa}$ denote taking fibrewise $\kappa$--compact objects and fibrewise $\kappa$--Ind--completion, respectively.
\end{prop}

The following object -- whose categorical properties we shall work out in \cref{subsec:SemiadditivityStability} -- will be one of the main players in this paper as it will be the domain of our parametrised  K--theory functor. 

\begin{nota}[$\baseCat$--perfect--stable categories]\label{nota:CatTPerfectStable}
We will use the notation $\underline{\cat}^{\myuline{\mathrm{perf}}(\kappa)}_{\baseCat} \coloneqq \underline{\cat}^{\tstable,\underline{\mathrm{idem}}(\kappa)}_{\baseCat} \subseteq \underline{\cat}_{\baseCat}^{\underline{\mathrm{idem}}(\kappa)}$ for the $\baseCat$\textit{--perfect-stable categories} which are parametrised $\kappa$--cocomplete, where the word perfect is standard terminology for being idempotent-complete. When $\kappa=\omega$, we will often use the abbreviation $\underline{\cat}\Tperfect_{\baseCat}\coloneqq \underline{\cat}^{\myuline{\mathrm{perf}}(\omega)}_{\baseCat}$.  By \cref{TPresentableIdempotentCorrespondence}, we have $\underline{\presentable}_{\baseCat, \mathrm{st}, L, \omega} \simeq \underline{\cat}\Tperfect_{\baseCat} $ implemented by taking fibrewise compact objects and fibrewise Ind--completion. Moreover, by \cref{characterisationStrongPreservations}, we see that the faithful inclusion $\underline{\cat}\Tperfect_{\baseCat}\subset \underline{\cat}_{\baseCat}$ factors through $\cofree(\cat\perfect)\subset \cofree(\cat)= \underline{\cat}_{\baseCat}$.
\end{nota}

\subsection{Miscellaneous preliminaries}\label{subsection:miscelleneousPreliminaries}

We consign to this subsection various elementary miscellany about parametrised categories that we shall need for main body of this article. 

\begin{recollect}\label{fact:cocartesianFibrationsAgree}
    An elegant definition of parametrised cocartesian fibrations has been given by Martini, and we refer the reader to \cite[Def. 3.1.1]{martiniCocartesian} for details. 
\end{recollect}

\noindent With this, we have a parametrised version of Lurie's straightening-unstraightening equivalence.

\begin{thm}[{\cite[Thm. 6.3.1]{martiniCocartesian}}]\label{thm:parametrisedStraighteningUnstraightening}
    Let $\underline{S}\in\cat_{\baseCat}$. There is a natural equivalence $\func_{\baseCat}(\underline{S},\underline{\cat}_{\baseCat})\simeq \cocartesianCategory_{\baseCat}(\underline{S})$.
\end{thm}

\begin{cons}\label{cons:cartesianStraightening}
    There is also a cartesian version of \cref{thm:parametrisedStraighteningUnstraightening}: by inspecting \cite[Def. 3.1.1]{martiniCocartesian}, a $\baseCat$--map $F\colon \underline{\sC}\rightarrow\underline{\D}$ is a $\baseCat$--cocartesian fibration if and only if $F\vop\colon \underline{\sC}\vop\rightarrow\underline{\D}\vop$ is a $\baseCat$--cartesian fibration, i.e. for a fixed $\underline{S}\in\cat_{\baseCat}$, the $\baseCat$--functor $(-)\vop\colon \cocartesianCategory_{\baseCat}(\underline{S})\rightarrow \mathrm{Cart}_{\baseCat}(\underline{S}\vop)$ is an equivalence. Hence, by the parametrised \textit{cartesian} straightening--unstraightening equivalence, we will mean the composite
    \[\func_{\baseCat}(\underline{S},\underline{\cat}_{\baseCat})\xrightarrow[\simeq]{(-)\vop}\func_{\baseCat}(\underline{S},\underline{\cat}_{\baseCat}) \simeq\cocartesianCategory_{\baseCat}(\underline{S}) \xrightarrow[\simeq]{(-)\vop}\mathrm{Cart}_{\baseCat}(\underline{S}\vop) \]
    where the middle equivalence is by \cref{thm:parametrisedStraighteningUnstraightening}.
\end{cons}

\begin{nota}\label{LFunRFunNotations}
We write $\underline{\rfunc}_{\baseCat}$ (resp. $\underline{\lfunc}_{\baseCat}$) for the $\baseCat$-full subcategories of $\underline{\func}_{\baseCat}$ consisting of $\baseCat$-right adjoint functors (resp. $\baseCat$-left adjoint functors); we write $\underline{\func}^R_{\baseCat}$ (resp.  $\underline{\func}^L_{\baseCat}$) for the $\baseCat$-full subcategories of $\underline{\func}_{\baseCat}$ consisting of strongly $\baseCat$-limit-preserving functors (resp. strongly $\baseCat$-colimit-preserving functors).
\end{nota}

\begin{prop}[{\cite[{Prop. 2.5.10}]{kaifPresentable}}]\label{HTT5.2.6.2}
Let $\underline{\sC},  \underline{\D}\in\cat_{\baseCat}$. Then $\underline{\lfunc}_{\baseCat}(\underline{\D}, \underline{\sC})\simeq \underline{\rfunc}_{\baseCat}(\underline{\sC}, \underline{\D})\vop$.
\end{prop}

Next, we now record several facts about Beck--Chevalley (i.e. adjointed) squares.

\begin{lem}\label{prop:adjointabilityPreservesAdjunction}
    Suppose we have $\baseCat$--adjunctions $L\colon\underline{\sC}\rightleftharpoons \underline{\D}: R$ and $\widetilde{L}\colon \widetilde{\underline{\sC}}\rightleftharpoons \widetilde{\underline{\D}} : \widetilde{R}$ with adjunction (co)units $\eta, \varepsilon$ and $\widetilde{\eta},\widetilde{{\varepsilon}}$ respectively, together with a right adjointable square
    \begin{center}
        \begin{tikzcd}
            \underline{\sC} \rar["\varphi"]\dar["L"'] & \widetilde{\underline{\sC}}\dar["\widetilde{L}"']\\
            \underline{\D} \rar["\psi"] \ar[u, bend right = 50, dashed ,"R"']& \widetilde{\underline{\D}}\ar[u, bend right = 50, dashed ,"\widetilde{R}"']
        \end{tikzcd}
    \end{center}
    For every $X\in \underline{\sC}$, we then have a natural identification of $\varphi\eta_X\colon \varphi X \rightarrow \varphi RLX$ with $\widetilde{\eta}_{\varphi X} \colon \varphi X \rightarrow \widetilde{R}\widetilde{L}\varphi X$ via the identification $\varphi RLX\simeq \widetilde{R}\widetilde{L}\varphi X$ coming from adjointability. Similarly, for $A\in \underline{\D}$, we have a natural identification of $\psi\varepsilon_A\colon \psi LRA\rightarrow\psi A$ with $\widetilde{\varepsilon}_{\psi A}\colon \widetilde{L}\widetilde{R}\psi A\rightarrow\psi A$.
\end{lem}
\begin{proof}
    Adjointability means the Beck--Chevalley transformation
    $\varphi R Y\xlongrightarrow{\widetilde{\eta}_{\varphi RY}} \widetilde{R}\widetilde{L}\varphi RY\simeq_{\widetilde{R}\sigma_R}\widetilde{R}\psi L RY \xlongrightarrow{\widetilde{R}\psi\varepsilon_{Y}}\widetilde{R}\psi Y$
    is an equivalence, where $\sigma$ is the datum of the commutation $\psi L\simeq \widetilde{L}\varphi$. Hence,  it would suffice to show that the diagram
    \begin{equation}\label{eqn:unitIdentificationCommutation}
        \begin{tikzcd}
            \varphi X \ar[rrr, "\widetilde{\eta}_{\varphi X}"]\dar["\varphi \eta_X"'] &&& \widetilde{R}\widetilde{L}\varphi X\dar["\widetilde{R}\sigma", "\simeq "']\\
            \varphi RLX \rar["\widetilde{\eta}_{\varphi RLX}"]& \widetilde{R}\widetilde{L}\varphi RLX \rar["\simeq"',"\widetilde{R}\sigma_{RLX}"] & \widetilde{R}\psi LRLX \rar["\widetilde{R}\psi\varepsilon_{LX}"]& \widetilde{R}\psi LX
        \end{tikzcd}
    \end{equation}
    commutes. To this end, just observe that the bottom left composite participates in the following commuting diagram (where the bottom composite is  an equivalence by adjointability)
    \begin{center}
        \begin{tikzcd}
            \varphi X\dar["\varphi \eta_X"'] \rar["\widetilde{\eta}_{\varphi X}"]&\widetilde{R}\widetilde{L}\varphi X\dar["\widetilde{R}\widetilde{L}\varphi \eta_X"']\rar["\widetilde{R}\sigma_X", "\simeq"']& \widetilde{R}\psi LX\dar["\widetilde{R}\psi L\eta_X"']\ar[dr,equal]\\
            \varphi RLX \rar["\widetilde{\eta}_{\varphi RLX}"]& \widetilde{R}\widetilde{L}\varphi RLX \rar["\simeq"',"\widetilde{R}\sigma_{RLX}"] & \widetilde{R}\psi LRLX \rar["\widetilde{R}\psi\varepsilon_{LX}"]& \widetilde{R}\psi LX
        \end{tikzcd}
    \end{center}
    yielding the desired commutation \cref{eqn:unitIdentificationCommutation}. The case of counits is similar, using instead the commuting diagram
    \begin{center}
        \begin{tikzcd}
            & \widetilde{L}\varphi RA \rar["\simeq"', "\sigma_{RA}"] & \psi LRA \rar["\psi\varepsilon_A"] &\psi A\\
            \widetilde{L}\varphi RA\rar["\widetilde{L}\widetilde{\eta}_{\varphi RA}"]\ar[ur,equal, shift left = 2] & \widetilde{L}\widetilde{R}\widetilde{L}\varphi RA \uar["\widetilde{\varepsilon}_{\widetilde{L}\varphi RA}"']\rar["\simeq"', "\widetilde{L}\widetilde{R}\sigma_{RA}"] & \widetilde{L}\widetilde{R}\psi LRA \uar["\widetilde{\varepsilon}_{\psi LRA}"']\rar["\widetilde{L}\widetilde{R}\psi\varepsilon_A"] & \widetilde{L}\widetilde{R}\psi A  \uar["\widetilde{\varepsilon}_{\psi A}"']
        \end{tikzcd}
    \end{center}
    where the bottom composite is the Beck--Chevalley equivalence.
\end{proof}

\begin{prop}\label{prop:limitOfBeckChevalleyAdjunctions}
    Fix a category $J$ and suppose we have two objects $\underline{\sC},\underline{\D}\in\func(J\op,\cat)$ together with a morphism $R \colon \underline{\D}\rightarrow\underline{\sC}$. Suppose moreover that the map $R$ is fibrewise a right adjoint and that for each morphism $f\colon i\rightarrow j$ in $J$, the Beck--Chevalley transformation
    \begin{center}
        \begin{tikzcd}
            \sC_i \rar["f^*"] & \sC_j\\
            \D_i \rar["f^*"']\uar["L_i"] \ar[ur,phantom,"\Leftarrow"]& \D_j\uar["L_j"']
        \end{tikzcd}
    \end{center}
    is an equivalence, where $L_i\dashv R_i, L_j\dashv R_j$. Then the left adjoints assemble to a morphism $L\colon \underline{\sC}\rightarrow\underline{\D}$ in $\func(J\op,\cat)$ which  induces an adjunction $\lim_{J\op}L \colon \lim_{J\op}\underline{\sC} \rightleftharpoons \lim_{J\op}\underline{\D} : \lim_{J\op}R$ in $\cat$.    
\end{prop}
\begin{proof}
    For the first part, we will use Lurie's theory of relative adjunctions. Write $\int\underline{\sC},\int\underline{\D}\rightarrow J\op$ for the respective cocartesian unstraightening, so that we have a map of $J\op$--cocartesian fibrations $R\colon \int\underline{\D}\rightarrow\int\underline{\sC}$. There are two conditions in \cite[Prop. 7.3.2.11]{lurieHA} to check in order to obtain a left adjoint $L\colon \int\underline{\sC}\rightarrow\int\underline{\D}$ to $R$ relative to $J\op$. Condition (1) there  is immediate from our fibrewise adjunction hypothesis. A straightforward unwinding of condition (2) there states that we need to check the following: for every morphism $f\colon i\rightarrow j$, the map $f^*\xrightarrow{f^*\eta}f^*R_iL_i\simeq R_jf^*L_i$ adjoints to an equivalence $L_jf^*\xrightarrow{\simeq} f^*L_i$. But this adjointed map is precisely the Beck--Chevalley transformation, and so by hypothesis, is an equivalence. It is then an easy check to see that the relative left adjoint $L$ is automatically a map of $J\op$--cocartesian fibrations (see for example \cite[Prop. 2.2.5 (2)]{kaifPresentable}), whence a map $L\colon \underline{\sC}\rightarrow \underline{\D}$ in $\func(J\op,\cat)\simeq \mathrm{coCart}(J\op)$ as wanted. Finally, by the formula for limits in $\cat$ in terms of cocartesian sections of the cocartesian unstraightening (cf. the dual of \cite[Cor. 3.3.3.2]{lurieHTT} recorded  in \cite[Prop. I.36]{fabianWagner}), we obtain the adjunction 
    $\lim_{J\op}L\colon \lim_{J\op}\underline{\sC} \simeq \Gamma_{\mathrm{cocart}}(\int\underline{\sC}) \rightleftharpoons \Gamma_{\mathrm{cocart}}(\int\underline{\D}) \simeq \lim_{J\op}\underline{\D} : \lim_{J\op}R$ as claimed.
\end{proof}

\begin{fact}[Adjoints of equivariant functors]\label{fact:adjointOfEquivariantAdjunctions}
    It is a standard categorical fact that the $G$--equivariant structure of a functor admitting an adjoint induces a $G$--equivariant structure on the adjoint and on the adjunction. Since we have not been able to find it anywhere in the literature, we will record a proof of this here which we learnt from Maxime Ramzi. Let $L\colon \sC \rightarrow \D$ be a $G$--equivariant functor whose underlying functor admits a right adjoint. In particular, $L$ can be encoded as a morphism in $\func(BG,\cat)$ and so upon unstraightening, we have a map $L\colon \int\sC \rightarrow \int\D$ of cocartesian fibrations over $BG$. Now since $BG$ was an $\infty$--groupoid, we get by \cite[Lem. 3.1.6]{markusInfinityBook} that all cocartesian morphisms in $\int\sC$ and $\int\D$ are equivalences. In particular, any map between them over $BG$ is a map of cocartesian fibrations. Hence, by the dual of \cite[Prop. 7.3.2.11]{lurieHA}, the fibrewise right adjoint coming from the underlying adjunction induces a relative right adjoint $R\colon\int\D\rightarrow \int\sC$ which is automatically a map of cocartesian fibrations over $BG$ by the previous sentence. All in all, we have obtained a $G$--equivariant structure on the right adjoint as well as on the unit and counit maps, as wanted.  
\end{fact}

The next pair of results provide abstract colimit decomposition results that will be crucial to our cubical theory in \cref{section:theoryOfParametrisedCubes}.

\begin{lem}\label{lem:limitOfAdjoints}
    Suppose we are given a map 
    \begin{center}
        \begin{tikzcd}
            \underline{\sC}\rar["\gamma^*"]&\underline{\E}&\lar["\beta^*"']\underline{\D} \\
            \underline{\A} \rar[equal] \uar["p^*"] & \underline{\A}\uar["r^*"]& \underline{\A}\lar[equal]\uar["q^*"]
        \end{tikzcd}
    \end{center}
    of diagrams in $\func(\Lambda^2_0,\cat_{\baseCat})$,  where $\gamma^*,\beta^*,p^*,q^*,r^*$ admit right adjoints $\gamma_*,\beta_*,p_*,q_*,r_*$ respectively. If $\underline{\A}$ has pullbacks, then the functor $p^*\times_{r^*}q^*\colon \underline{\A}\rightarrow\underline{\sC}\times_{\underline{\E}}\underline{\D}$ has a right adjoint given by the functor $\varphi$  sending $(c,d)\in \underline{\sC}\times_{\underline{\E}}\underline{\D}$ to the left pullback square in $\underline{\A}$ in the following
    \begin{center}
        \begin{tikzcd}
            \varphi(c,d)\rar\dar \ar[dr, "\lrcorner", phantom, very near start]& q_*d\dar["q_*\eta^{\beta}_d"]\\
            p_*c\rar["p_*\eta^{\gamma}_c"] & r_*\gamma^*c \simeq r_*\beta^*d
        \end{tikzcd}
        \hspace{6mm}
        \begin{tikzcd}
            r_!\gamma^*c \simeq r_!\beta^*d\rar["q_!\varepsilon^{\beta}_d"]\dar["p_!\varepsilon^{\gamma}_c"'] \ar[dr, "\ulcorner", phantom, very near end]& q_!d\dar\\
            p_!c\rar & \psi(c,d)
        \end{tikzcd}
    \end{center}
    Similarly, if $\gamma^*,\beta^*,p^*,q^*,r^*$ admit instead left adjoints $\gamma_!,\beta_!,p_!,q_!,r_!$ respectively and $\underline{\A}$ admits pushouts, then the functor $p^*\times_{r^*}q^*$ admits a left adjoint given by the functor $\psi$ sending $(c,d)\in\underline{\sC}\times_{\underline{\E}}\underline{\D}$ to the right  pushout square in $\underline{\A}$ above.
\end{lem}
\begin{proof}
    The putative right adjoint $\varphi$ is clearly a $\baseCat$--functor, and is fibrewise a right adjoint by \cite[Thm. 5.5]{horevYanovski}. Hence, by \cref{criteriaForTAdjunctions}, we get that $\varphi$ is indeed the $\baseCat$--right adjoint of $p^*\times_{r^*}q^*$. The statement for the left adjoint follows by passing to the opposite categories.
\end{proof}

\begin{cor}\label{cor:(co)limitDecompositionAlongPushouts}
    Suppose we have the left pushout diagram in $\cat_{\baseCat}$ in the following
    \begin{center}
        \begin{tikzcd}
            \underline{\B} \rar\dar \ar[dr, phantom ,very near end, "\ulcorner"]& \underline{\D}\dar\\
            \underline{\sC} \rar & \underline{\proper}
        \end{tikzcd}
        \hspace{7mm}
        \begin{tikzcd}
            \underline{\colim}_{\underline{\B}}\partial \rar\dar \ar[dr, phantom ,very near end, "\ulcorner"]& \underline{\colim}_{\underline{\D}}\partial\dar\\
            \underline{\colim}_{\underline{\sC}}\partial\rar&\underline{\colim}_{\underline{\proper}}\partial.
        \end{tikzcd}
    \end{center}
    Let $\underline{\A}\in\cat_{\baseCat}$ have all $\baseCat$--colimits. Then for any $\partial\in \underline{\func}(\underline{\proper},\underline{\A})$, we have the pushout diagram in $\underline{\A}$ as in the right-hand square above,
    where we have suppressed the restriction functors. A similar statement holds for limits, with all colimits in sight replaced with limits.
\end{cor}
\begin{proof}
    This is an immediate consequence of \cref{lem:limitOfAdjoints}, using that $\underline{\func}(\underline{\proper},\underline{\A})\simeq \underline{\func}(\underline{\sC},\underline{\A})\times_{\underline{\func}(\underline{\B},\underline{\A})}\underline{\func}(\underline{\D},\underline{\A})$ by definition of $\underline{\proper}$ as a pushout.
\end{proof}

Next, recall the notion of  $\baseCat$--faithful functors from \cite[Def. 3.4.4]{kaifPresentable}, i.e. a $\baseCat$--functor which induces fibrewise inclusions of components on mapping spaces.

\begin{lem}\label{lem:faithful(co)limitCreation}
    Let $\underline{\sC}\subset \underline{\D}$ be a $\baseCat$--faithful inclusion where $\underline{\D}$ is admits limits of shape $\underline{J}$. Then $\underline{\sC}$ inherits $\underline{J}$--shaped limits from $\underline{\D}$ if the following conditions hold:
    \begin{enumerate}
        \item for any object  $\partial\in\underline{\func}(\underline{J},\underline{\sC})\subset\underline{\func}(\underline{J},\underline{\D})$, the limit of $\partial$ has the property of lying in $\underline{\sC}$ and the adjunction counit $\underline{\constant}_{\underline{J}}\underline{\lim}_{\underline{J}}\partial \rightarrow \partial$ lies in $\underline{\func}(\underline{J},\underline{\sC})$,
        \item for any $C\in\underline{\sC}$ equipped with a morphism $\underline{\constant}_{\underline{J}}C\rightarrow \partial$ in $\underline{\func}(\underline{J},\underline{\sC})$, the induced morphism $C\rightarrow \underline{\lim}_{\underline{J}}\partial$ lies in $\underline{\sC}\subset\underline{\D}$.
    \end{enumerate}
    A similar statement holds also for colimits by passing to the opposite categories.
\end{lem}
\begin{proof}
    Suppose we are given such an object $\partial\in\underline{\func}(\underline{J},\underline{\sC})$ satisfying (1) and (2). We need to argue that the natural map $\myuline{\map}_{\underline{\sC}}(-,\underline{\lim}_{\underline{J}}\partial)\rightarrow\myuline{\map}_{\underline{\func}(\underline{J},\underline{\sC})}(\underline{\constant}_{\underline{J}}-,\partial)$ induced by the morphism $\underline{\constant}_{\underline{J}}\underline{\lim}_{\underline{J}}\partial \rightarrow \partial$, which is in $\underline{\func}(\underline{J},\underline{\sC})$ by condition (1), is an equivalence. To this end, consider the commuting diagram
    \begin{center}
        \begin{tikzcd}
            \myuline{\map}_{\underline{\sC}}(-,\underline{\lim}_{\underline{J}}\partial)\dar[hook]\rar&\myuline{\map}_{\underline{\func}(\underline{J},\underline{\sC})}(\underline{\constant}_{\underline{J}}-,\partial)\dar[hook]\\
            \myuline{\map}_{\underline{\D}}(-,\underline{\lim}_{\underline{J}}\partial)\rar["\simeq"]&\myuline{\map}_{\underline{\func}(\underline{J},\underline{\D})}(\underline{\constant}_{\underline{J}}-,\partial)
        \end{tikzcd}
    \end{center}
     where the bottom horizontal is an equivalence by definition of $\underline{\lim}_{\underline{J}}\partial$. Since the vertical maps are inclusion of subcomponents, so is the top horizontal map. Hence, we are left to showing that the top horizontal map is $\pi_0$--surjective, which is precisely supplied by condition (2).
\end{proof}

We include the proof of the following standard observation for the reader's convenience as well as to establish what we mean by ``reflecting (co)limits''.

\begin{lem}\label{lem:conservativePreservationImpliesReflection}
    Let $\underline{\sC},\underline{\D}, \underline{I},\in\cat_{\baseCat}$ such that $\underline{\sC},\underline{\D}$ admits $\underline{I}$--shaped (co)limits. Let $F\colon \underline{\sC}\rightarrow \underline{\D}$ be a conservative functor preserving $\underline{I}$--shaped (co)limits. Then $F$ reflects $\underline{I}$--shaped (co)limits. 
\end{lem}
\begin{proof}
    Without loss of generality we deal with the case of limits. Suppose we have a coned $\underline{I}$--shaped diagram $\partial\colon \underline{I}\tcone \rightarrow \underline{\sC}$ such that $F\partial\colon \underline{I}\tcone\rightarrow\underline{\D}$ is a limit diagram. This means that if we write $\infty$ for the cone point, then we have a map $\partial(\infty)\rightarrow \underline{\lim}_{\underline{I}}\partial$ in $\underline{\sC}$ such that $F\partial(\infty)\rightarrow F\underline{\lim}_{\underline{I}}\partial\simeq \underline{\lim}_{\underline{I}}F\partial$ is an equivalence in $\underline{\D}$. Since $F$ was conservative, we get that $\partial(\infty)\rightarrow \underline{\lim}_{\underline{I}}\partial$ was already an equivalence, as was to be shown.
\end{proof}

We will deduce the parametrised analogue of \cite[{Lem 1.7.ii}]{makkaiPitts} from the unparametrised version proven in \cite{nineAuthorsIV}. We will need some terminology for this.

\begin{terminology}
Let $\underline{\sC}$ be a ${\baseCat}$-cocomplete category and $S$ be a set of objects in $\underline{\sC}$. We say that it is \textit{jointly conservative} if $S$ induces a jointly conservative set of objects in each fibre of $\underline{\sC}$, i.e. for every $V \in {\baseCat}$ and writing $S_V$ for the set of objects of $\sC_V$ in the set $S$, the functor $\prod_{x\in S_V}\map_{\sC _V}(x,-) : \sC_V \rightarrow \prod_{x\in S_V}\spc$ is conservative. We say that it is a \textit{set of parametrised generators of $\underline{\sC}$} if the smallest ${\baseCat}$-cocomplete subcategory of $\underline{\sC}$ containing $S$ is $\underline{\sC}$ itself. That is, every parametrised object in $\underline{\sC}$ can be written as a parametrised colimit of objects in $S$. 
\end{terminology}

\begin{prop}[Parametrised Makkai-Pitts] \label{parametrisedMakkaiPitts}
Let $\kappa$ be a regular cardinal and $\underline{\sC}$ a ${\baseCat}$-cocomplete category. Let $S \subseteq \underline{\sC}$ be a jointly conservative set of parametrised-$\kappa$-compact objects. Then $S$ is a set of parametrised-$\kappa$-compact generators. In particular, $\underline{\sC}$ is parametrised-$\kappa$-compactly generated.
\end{prop}
\begin{proof}
We want to show that for every $V \in {\baseCat}$, any $\baseCat_{/V}$--object in $\underline{\sC}_{\underline{V}}$ is a $\baseCat_{/V}$--colimit of objects in $S$. By hypothesis, $\prod_{x \in S_V}\map_{\sC _V}(x, -) \colon \sC_V \rightarrow \prod_{x\in S_V}\spc$ is jointly conservative. Hence, by \cite[{Prop 1.1.2}]{nineAuthorsIV}, every object in $\sC_V$ is a $\kappa$-small colimit of objects in $S_V$.
\end{proof}

Next, we supply the expected anti--equivalence of presentable categories.

\begin{prop}\label{prop:antiequivalencePrLPrR}
    There is a canonical equivalence of $\baseCat$--categories $\underline{\presentable}_{L}\simeq \underline{\presentable}_R\vop$.
\end{prop}
\begin{proof}
    This is just the proof of \cite[Prop. 5.5.3.3]{lurieHTT} written in our setting, except that we do not need to invoke the adjoint functor \cref{parametrisedAdjointFunctorTheorem} since $\underline{\presentable}_L$ and $\underline{\presentable}_R$ were defined with morphisms being left and right adjoints respectively. Write $\widehat{\underline{\cat}}_{\baseCat}$ and $\underline{\mathrm{CAT}}_{\baseCat}$ for the $\baseCat$--categories of large and huge $\baseCat$--categories, respectively. In particular, we have that $\underline{\presentable}_L,\underline{\presentable}_R\subset\widehat{\underline{\cat}}_{\baseCat}\in \underline{\mathrm{CAT}}_{\baseCat}$. Now, fix a $\underline{S}\in\underline{\mathrm{CAT}}_{\baseCat}$. Under the cocartesian unstraightening equivalence of \cref{thm:parametrisedStraighteningUnstraightening}, it is easy to see that elements in $\pi_0\map_{{\mathrm{CAT}}_{\baseCat}}(\underline{S},\underline{\widehat{\cat}}_{\baseCat})$ that lie in $\pi_0\map_{{\mathrm{CAT}}_{\baseCat}}(\underline{S},\underline{\presentable}_{L})\subset \pi_0\map_{{\mathrm{CAT}}_{\baseCat}}(\underline{S},\underline{\widehat{\cat}}_{\baseCat})$ are precisely those $\baseCat$--cocartesian fibrations $\underline{\proper}\rightarrow \underline{S}$ which are also $\baseCat$--cartesian and whose parametrised fibres are  $\baseCat$--presentable. Similarly, under the cartesian unstraightening from \cref{cons:cartesianStraightening}, the subset $\pi_0\map_{{\mathrm{CAT}}_{\baseCat}}(\underline{S}\vop,\underline{\presentable}_{R})\subset \pi_0\map_{{\mathrm{CAT}}_{\baseCat}}(\underline{S}\vop,\underline{\widehat{\cat}}_{\baseCat})$ is precisely described as those $\baseCat$--cartesian fibrations $\underline{\E}\rightarrow\underline{S}$ which are also $\baseCat$--cocartesian and whose parametrised fibres are $\baseCat$--presentable. Therefore, we obtain bijections 
    $\pi_0\map_{{\mathrm{CAT}}_{\baseCat}}(\underline{S},\underline{\presentable}_{L})\cong \pi_0\map_{{\mathrm{CAT}}_{\baseCat}}(\underline{S}\vop,\underline{\presentable}_{R})\cong \pi_0\map_{{\mathrm{CAT}}_{\baseCat}}(\underline{S},\underline{\presentable}_{R}\vop)$ natural in $\underline{S}\in\underline{\mathrm{CAT}}_{\baseCat}$. Since $\underline{\mathrm{CAT}}_{\baseCat}$ admits pushouts, we may replace $\underline{S}$ with $\Sigma^n\underline{S}$ in the natural bijection of sets to upgrade it to an equivalence of mapping spaces, whence the equivalence $\underline{\presentable}_L\simeq \underline{\presentable}_R\vop$ as wanted.
\end{proof}

Recall \cref{LFunRFunNotations}. In \cite[$\S3.4$]{nardinThesis}, Nardin constructed a $\baseCat$--symmetric monoidal structure on $\underline{\presentable}_L$ generalising Lurie's tensor product for presentable categories with the tensor unit given by the $\baseCat$--category $\underline{\spc}_{\baseCat}$ of spaces. The following  was then stated as \cite[Ex. 3.26]{nardinThesis} without proof, and we have supplied a proof in \cite[Prop. 6.7.5]{kaifPresentable}. 

\begin{prop}[Formula for presentable $\baseCat$--tensors]\label{TTensorsOfPresentables} Let $\baseCat$ be an atomic orbital category, and let $\underline{\sC}, \underline{\D}$ be $\baseCat$--presentable categories. Then 
$\underline{\sC}\otimes \underline{\D} \simeq \underline{\func}_{\baseCat}^{R}(\underline{\sC}\vop,\underline{\D})$. 
\end{prop}

    To state the next hypothesis, recall Nardin's definition \cite[Def. 3.9]{nardinThesis} of an idempotent object in $\underline{\presentable}_L$, i.e. an object $\underline{\A}\in\underline{\presentable}_L$ equipped with a map from the tensor unit $\underline{\spc}\rightarrow\underline{\A}$ is said to be $\baseCat$--\textit{idempotent} if for all maps $f\colon U \rightarrow V$ in $\finite_{/V}$, the map \[V^*\underline{\A}\simeq V^*\underline{\A}\otimes f_{\otimes}U^*\underline{\spc}\longrightarrow V^*\underline{\A}\otimes f_{\otimes}U^*\underline{\A}\] is an equivalence. In particular, such objects are idempotent objects in the usual sense of Lurie \cite[Def. 4.8.2.1]{lurieHA}. The relevance of this notion for us comes via \cite[Lem. 3.11]{nardinThesis}, which guarantees that the smashing localisation associated to such $\underline{\A}$ refines to a $\baseCat$--symmetric monoidal localisation.

\begin{hypothesis}\label{hypothesis:normed_spectra}
    As the next proposition will show, $\myuline{\spectra}$ is always an idempotent algebra in the nonparametrised sense, and in particular, it admits an algebra unit  $\baseCat$--colimit--preserving map $\underline{\spc}\rightarrow\myuline{\spectra}$. We say that $\baseCat$ satisfies the \textit{spectral idempotence hypothesis} if the map $\underline{\spc}\rightarrow\myuline{\spectra}$ is even a $\baseCat$--idempotent in the sense recalled above.

    The reason we need to make this hypothesis is that we have found a gap in the proof of \cite[Prop. 3.27]{nardinThesis} which is used to guarantee that $\myuline{\spectra}$ is a $\baseCat$--idempotent. As pointed out in the introduction, together with Branko Juran and Natalie Stewart in forthcoming work \cite{brankoKaifNatalie}, we have however been able to fix the proof in the case where $\baseCat$ is the orbit category $\orbit_G$ of a finite group $G$, and so this does not affect our main applications later in the paper. 
\end{hypothesis}

\begin{prop}\label{prop:presentableTensorFormula}
For $\underline{\sC}$ a $\baseCat$--presentable category, we have that $\myuline{\spectra}_{\baseCat}(\sC) \simeq \underline{\sC} \otimes\myuline{\spectra}_{\baseCat}$. In particular, $\myuline{\spectra}_{\baseCat}$ is always an idempotent in the nonparametrised sense.
\end{prop}
\begin{proof}
Consider the sequence of equivalences 
\begin{equation*}
    \begin{split}
        \underline{\sC} \otimes \myuline{\spectra} &\simeq \underline{\func}^R(\underline{\sC}\vop , \myuline{\spectra})\simeq \underline{\func}^R(\underline{\sC}\vop , \myuline{\spectra}^{\fib}\underline{\cmonoid}(\underline{\spc}))\simeq \myuline{\spectra}^{\fib}\underline{\func}^R(\underline{\sC}\vop , \underline{\cmonoid}(\underline{\spc}))\\
        &\simeq \myuline{\spectra}^{\fib}\underline{\cmonoid}(\underline{\func}^R(\underline{\sC}\vop , \underline{\spc}))\simeq \myuline{\spectra}^{\fib}\underline{\cmonoid}(\underline{\sC}\otimes \underline{\spc})\simeq \myuline{\spectra}^{\fib}\underline{\cmonoid}(\underline{\sC}) \simeq \myuline{\spectra}(\underline{\sC})
    \end{split}
\end{equation*}
where the first equivalence is by \cref{TTensorsOfPresentables}. The third equivalence is because, by \cite[Def. 1.4.2.8, Prop. 1.4.2.21]{lurieHA}, $\myuline{\spectra}^{\fib}$ is given concretely by applying $\func^{\mathrm{exc}}_*(\spc_*^{\mathrm{fin}},-)$ fibrewise where $\func^{\mathrm{exc}}_*\subseteq \func$ is the full subcategory of reduced functors which sends pushout squares to pullback squares. Thus, by \cref{parametrisedCotensors}, both $\underline{\func}^R(\underline{\sC}\vop , \myuline{\spectra}^{\fib}\underline{\cmonoid}(\underline{\spc}))$ and $ \myuline{\spectra}^{\fib}\underline{\func}^R(\underline{\sC}\vop , \underline{\cmonoid}(\underline{\spc}))$  are full $\baseCat$--subcategories of $\underline{\func}(\underline{\sC}\vop \times \underline{\constant}(\spc_{*}^{\mathrm{fin}}), \underline{\cmonoid}(\underline{\spc}))$ satisfying the same requisite properties on the $\underline{\sC}\vop$ and $\underline{\constant}(\spc_{*}^{\mathrm{fin}})$ variables, and hence are equivalent. Lastly, the fourth equivalence is since $\underline{\func}^R(\underline{\sC}\vop,\underline{\func}^{\times}(\myuline{\effBurn}(\underline{\finite}),\underline{\spc}))\simeq \underline{\func}^{\times}(\myuline{\effBurn}(\underline{\finite}),\underline{\func}^{R}(\underline{\sC}\vop,\underline{\spc}))$.
\end{proof}

\begin{prop}[Parametrised stabilisation is smashing, ``{\cite[{Thm. 4.6}]{GGN}}'']\label{parametrisedSemiadditivisationSmashing}
The association $\underline{\sC}\mapsto \myuline{\spectra}_{\baseCat}(\underline{\sC})$ refines to a smashing localisation 
$\myuline{\spectra}_{\baseCat}\otimes - : \underline{\presentable}_{\baseCat, L} \longrightarrow \underline{\presentable}_{\baseCat, L}$ with essential image the $\baseCat$--full subcategory of $\baseCat$--presentable-stable categories $\underline{\presentable}_{\baseCat, \mathrm{st}, L}$. If $\baseCat$ satisfies \cref{hypothesis:normed_spectra}, then this functor even enhances to a $\baseCat$--symmetric monoidal functor.
\end{prop}
\begin{proof}
That $\myuline{\spectra}_{\baseCat}(-)\simeq \myuline{\spectra}_{\baseCat}\otimes(-)$ is the proposition above, which also gives the required essential image. That the functor is a $\baseCat$--symmetric monoidal localisation under \cref{hypothesis:normed_spectra} is  then immediate by \cite[Lem. 3.11]{nardinThesis}.
\end{proof}

\begin{cor} \label{presheafCommutesWithTensor}
For $f : U \rightarrow W$ a map in $\finite_{\baseCat}$ and $\underline{\sC} \in \cat_{\baseCat_{/U}}$, there is a natural equivalence $f_{\otimes}\underline{\presheaf}_{U}(\underline{\sC}) \simeq \underline{\presheaf}_{W}(f_*\underline{\sC})$. If $\baseCat$ satisfies \cref{hypothesis:normed_spectra}, then we also have $f_{\otimes}\underline{\presheaf}^{\underline{\mathrm{st}}}_{U}(\underline{\sC}) \simeq \underline{\presheaf}^{\underline{\mathrm{st}}}_{W}(f_*\underline{\sC})$.
\end{cor}
\begin{proof}
Since $\myuline{\spectra}\otimes-$ is a smashing localisation from \cref{parametrisedSemiadditivisationSmashing}, it suffices just to prove the case of presheaves in spaces. Let $\underline{\D}\in\underline{\presentable}_{\baseCat_{/W}}$. By \cite[Prop. 3.19]{nardinThesis}, the restriction map $\underline{\func}^L_W\big(f_{\otimes}\underline{\presheaf}_U(\underline{\sC}), \underline{\D}\big) \longrightarrow \underline{\func}_W\big(f_*\underline{\sC}, \underline{\D}\big)$ is an equivalence. But then the target is naturally equivalent to $\underline{\func}^L_W\big(\underline{\presheaf}_W(f_*\underline{\sC}), \underline{\D}\big)$ by \cite[Thm. 11.5]{shahThesis} and so we are done.
\end{proof}

\begin{obs}[$\baseCat$--exactness on $\baseCat$--stables]\label{TRightExactsOnTStables}
Write $\underline{\func}\Tlexact,\underline{\func}\Trexact,\underline{\func}\Texact\subseteq \underline{\func}$ for the full subcategories of functors which preserve finite $\baseCat$--limits, finite $\baseCat$--colimits, and finite $\baseCat$--(co)limits, respectively.
If $\underline{\sC}, \underline{\D}$ are $\baseCat$--stable, then note that the two $\baseCat$--full subcategories $\underline{\func}_{\baseCat}\Tlexact(\underline{\sC}, \underline{\D}) \subseteq \underline{\func}_{\baseCat}(\underline{\sC}, \underline{\D})\supseteq  \underline{\func}_{\baseCat}\Trexact(\underline{\sC}, \underline{\D})$ agree. To wit, both imply that they are fibrewise right and left exact (since these are fibrewise stable after all); moreover, preserving finite $\baseCat$--coproducts and preserving finite $\baseCat$--products are equivalent since $\underline{\sC}, \underline{\D}$ were $\baseCat$--semiadditive. Hence in this case we have $\underline{\func}_{\baseCat}\Tlexact(\underline{\sC}, \underline{\D}) =\underline{\func}_{\baseCat}\Texact(\underline{\sC}, \underline{\D}) =\underline{\func}_{\baseCat}\Trexact(\underline{\sC}, \underline{\D})$.
\end{obs}

\begin{lem}[Internal hom object of $\baseCat$--perfects]\label{internalHomObjectOfTPerfects}
Let $\underline{\sC}, \underline{\D} \in \underline{\cat}\Tperfect_{\baseCat}$. Then the $\baseCat$--full subcategory $\underline{\func}_{\baseCat}\Texact(\underline{\sC}, \underline{\D}) \subseteq \underline{\func}_{\baseCat}(\underline{\sC}, \underline{\D})$ on the $\baseCat$--exact functors is also small $\baseCat$--idempotent-complete-stable, that is, $\underline{\func}_{\baseCat}\Texact(\underline{\sC}, \underline{\D})$ is again an object of $\underline{\cat}_{\baseCat}\Tperfect$.
\end{lem}
\begin{proof}
Smallness is clear. For $\baseCat$--stability, just note that $\underline{\func}_{\baseCat}\Texact(\underline{\sC}, \underline{\D}) \simeq \underline{\func}_{\baseCat}\Tlexact(\underline{\sC}, \underline{\D}) \simeq \underline{\func}_{\baseCat}\Tlexact(\underline{\sC},  \myuline{\spectra}^{\fib}\underline{\cmonoid}(\underline{\D}))\simeq \myuline{\spectra}^{\fib}\underline{\cmonoid}( \underline{\func}_{\baseCat}\Tlexact(\underline{\sC}, \underline{\D}))\simeq \myuline{\spectra}^{\fib}\underline{\cmonoid}( \underline{\func}_{\baseCat}\Texact(\underline{\sC}, \underline{\D}))$
where the first and last equivalences are by \cref{TRightExactsOnTStables}, the second is since $\underline{\D}$ is $\baseCat$--stable, and the third by the same proof as in \cref{prop:presentableTensorFormula}. Hence, $\underline{\func}\exact_{\baseCat}(\underline{\sC},\underline{\D})$ is $\baseCat$--stable. For $\baseCat$--idempotent-completeness, note that $\baseCat$--colimits of $\underline{\func}\Texact_{\baseCat}(\underline{\sC},  \underline{\D})\simeq \underline{\func}\Trexact_{\baseCat}(\underline{\sC}, \underline{\D})$ are computed in $\underline{\D}$, and since being $\baseCat$--idempotent-complete is just the condition of admitting certain fibrewise $\baseCat$--colimits, this point is clear too.
\end{proof}

\begin{prop}\label{colimitGenerationSpectra}
Let $\underline{\sC} \in \underline{\cat}^{\myuline{\mathrm{perf}}(\kappa)}_{\baseCat}$. Then $\underline{\func}_{\baseCat}^{\mathrm{ex}(\kappa)}(\myuline{\spectra}_{\baseCat}^{\kappa}, \underline{\sC}) \simeq \underline{\sC}$, where $\underline{\func}_{\baseCat}^{\mathrm{ex}(\kappa)}$ denotes the functors which preserves $\kappa$--finite (co)limits.
\end{prop}
\begin{proof}
Recall the equivalence $\underline{\cat}_{\baseCat}^{\myuline{\mathrm{perf}}(\kappa)} \simeq \underline{\presentable}_{\baseCat, \mathrm{st}, L,\kappa}$ from \cref{TPresentableIdempotentCorrespondence} so that $(\underline{\ind}_{\kappa}\underline{\sC})^{\kappa}\simeq \underline{\sC}$. Writing $\underline{\func}\tkappa \subseteq \underline{\func}$ for the $\baseCat$--full subcategory of  functors preserving parametrised $\kappa$--compact objects, consider $\underline{\func}_{\baseCat}^{\mathrm{ex}(\kappa)}(\myuline{\spectra}_{\baseCat}^{\kappa}, \underline{\sC})  \simeq \underline{\func}_{\baseCat}^{\mathrm{rex}(\kappa)}(\myuline{\spectra}_{\baseCat}^{\kappa}, (\underline{\ind}_{\kappa}\underline{\sC})^{\kappa}) \simeq \underline{\func}_{\baseCat}^{L, \kappa}(\myuline{\spectra}_{\baseCat}, \underline{\ind}_{\kappa}\underline{\sC})\simeq \underline{\func}_{\baseCat}^{L, \kappa}(\underline{\spc}_{\baseCat}, \underline{\ind}_{\kappa}\underline{\sC})\simeq (\underline{\ind}_{\kappa}\underline{\sC})^{\kappa} \simeq \underline{\sC}$
where the second equivalence is by \cite[Prop. 3.5.4]{kaifPresentable}; the third  is by \cref{parametrisedSemiadditivisationSmashing}; and the fourth  by the universal property of $\baseCat$--presheaves of \cite[Thm. 11.5]{shahThesis}.
\end{proof}

\subsection{Symmetric monoidality and localisations}\label{subsection:symmetricMonoidalityLocalisations}

\noindent \textbf{Convention:} From now on for the rest of the article, our base category $\baseCat$ will always be assumed to be atomic orbital.

\vspace{1mm}

The main aim of this subsection is to give a multiplicative enhancement of the presentable Dwyer--Kan localisations  from \cite[$\S6.3$]{kaifPresentable}.

\begin{terminology}
    Let $S$ be a collection of morphisms in a $\baseCat$--category $\underline{\sC}$. For any $U = U_1\sqcup\cdots\sqcup U_n \in \finite_{\baseCat}$ where $U_i\in\baseCat$, we will write $S_U$ for the collection of morphisms inside $S$ in $\sC_U\simeq \sC_{U_1}\times\cdots\times \sC_{U_n}$. We will say that $S$ is a $\baseCat$--\textit{collection} if for any morphism $f\colon U \rightarrow V$ in $\finite_{\baseCat}$ and any morphism $\varphi\colon A \rightarrow B$ in $S_V$ (in general, this is a tuple of morphisms as is clear from our definition of $S_V$ in the preceding sentence), the morphism $f^*\varphi\colon f^*A\rightarrow f^*B$ is a morphism in $S_U$.
\end{terminology} 

\begin{terminology}
    We say that a $\baseCat$--collection of morphisms $S$ in a $\baseCat$--symmetric monoidal category $\underline{\sC}\totimes$ is $\underline{\otimes}$\textit{--multiplicatively closed} if for any $V\in\baseCat$, any morphism $p\colon U\rightarrow V$ in $\finite_{\baseCat}$, and any morphism $\varphi\colon A \rightarrow B$ in $S_U$, the morphism $p_{\otimes}\varphi\colon p_{\otimes}A\rightarrow p_{\otimes}B$ in $\sC_V$ lies in $S_V$.
\end{terminology}

\begin{nota}\label{nota:dwyerKanAndBousfieldLocalisations}
Let $\underline{\sC}$ be a ${\baseCat}$-category and $S$ a $\baseCat$--collection of morphisms in $\underline{\sC}$. If there exists a ${\baseCat}$-category $S^{-1}\underline{\sC}$   equipped with a map $\mathrm{DK} : \underline{\sC} \rightarrow S^{-1}\underline{\sC}$ inducing the equivalence 
${\mathrm{DK}}^* : \underline{\func}_{\baseCat}(S^{-1}\underline{\sC},  \underline{\D}) \xrightarrow{\simeq} \underline{\func}_{\baseCat}^{S^{-1}}(\underline{\sC},  \underline{\D})$ for all ${\baseCat}$-categories $\underline{\D}$, where $\underline{\func}_{\baseCat}^{S^{-1}}(\underline{\sC}, \underline{\D}) \subseteq \underline{\func}_{\baseCat}(\underline{\sC},  \underline{\D})$ is the ${\baseCat}$-full subcategory of parametrised functors sending morphisms in $S$ to equivalences, then it must necessarily be unique, and this is then defined to be the \textit{${\baseCat}$--Dwyer--Kan localisation of} $\underline{\sC}$ \textit{with respect to} $S$. 

By ${\baseCat}$\textit{--Bousfield localisations}, we mean a ${\baseCat}$-adjunction $L : \underline{\sC} \rightleftarrows \underline{\D} : i$  where the ${\baseCat}$-right adjoint $i$ is ${\baseCat}$-fully faithful. Writing $Z$ for the morphisms in $\underline{\sC}$ that get sent to equivalences under $L$, we may then view $\underline{\D}$ as precisely the $\baseCat$--full subcategory of \textit{$Z$--local objects}, i.e. those $X\in\underline{\sC}$ such that for any morphism $\varphi\colon A \rightarrow B$ in $Z$, the induced map $\varphi^*\colon \myuline{\map}(B,X)\rightarrow\myuline{\map}(A,X)$ is an equivalence. In \cite[{Prop. 6.3.2}]{kaifPresentable} we showed that a ${\baseCat}$-Bousfield localisation $L\colon \underline{\sC}\rightarrow \underline{\D}$ is in particular a  ${\baseCat}$-Dwyer-Kan localisation with respect to this collection $Z$.
\end{nota}

We now record the following proposition, which has appeared also as \cite[Lem. 5.27]{quigleyShahParametrisedTate}. While in all likelihood our proof is perhaps only cosmetically distinct from theirs, we think that it is slightly simpler to verify without having to ``open the blackbox'' of parametrised operads. 

\begin{prop}[Dwyer-Kan symmetric monoidality]\label{TDwyerKanSymmetricMonoidality}
Let $\underline{\sC}\totimes$ be a $\baseCat$--symmetric monoidal category and $S$ a $\baseCat$--collection of morphisms in $\underline{\sC}$ which is $\underline{\otimes}$--multiplicatively closed. Then:
\begin{enumerate}
    \item the Dwyer--Kan localisation $S^{-1}\underline{\sC}$ attains a unique $\baseCat$--symmetric monoidal structure such that the canonical map $\mathrm{DK}  \colon \underline{\sC}\rightarrow S^{-1}\underline{\sC}$ enhances to a $\baseCat$--symmetric monoidal functor,
    \item For any $\baseCat$--symmetric monoidal category $\underline{\D}\totimes$, the induced functor $\mathrm{DK}^* : \myuline{\func}_{\baseCat}^{\underline{\otimes}}(S^{-1}\underline{\sC}^ {\underline{\otimes}}, \underline{\D}^{\underline{\otimes}}) \rightarrow \myuline{\func}_{\baseCat}^{\underline{\otimes}, S^{-1}}(\underline{\sC}^{\underline{\otimes}}, \underline{\D}^{\underline{\otimes}})$ is an equivalence.
\end{enumerate}
\end{prop}
\begin{proof}
We proceed by bootstrapping from the construction and proof of \cite[{Prop. 4.1.7.4}]{lurieHA}. Recall from \cite[{Cons. 4.1.7.1}]{lurieHA} that we have a category $\mathrm{WCat}$ whose objects are pairs $(\sC,  W)$ where $\sC$ is a category and $W$ is a collection of morphisms in $\sC$ stable under composition and contains all equivalences in $\sC$, and morphisms $f : (\sC,  W) \rightarrow (\sC', W')$ are functors $f : \sC \rightarrow \sC'$ such that $f(W) \subseteq W'$. By \cite[{Prop. 4.1.7.2}]{lurieHA} we have a Bousfield localisation 
\begin{equation}\label{WCatAdjunction}
    \begin{tikzcd}
        \mathrm{WCat} \rar[shift left = 1.5, "I"] & \mathrm{Cat} \lar[shift left = 1.5, hook]
    \end{tikzcd}
\end{equation}
where both functors preserve finite products and the functor $I$ sends $(\sC,  W)$ to the Dwyer-Kan localisation $W^{-1}\sC$. Applying \cref{cofreeParametrisations} to this adjunction we get the $\baseCat$--Bousfield localisation $\underline{I}_{\baseCat} : \cofree_{\baseCat}(\mathrm{WCat}) \rightleftarrows \cofree_{\baseCat}(\mathrm{Cat}) : \underline{\mathrm{incl}}_{\baseCat}$. Since both functors in \cref{WCatAdjunction} preserve finite products, the functor $\underline{I}_{\baseCat}$  preserves indexed products. Hence, applying the forgetful functor from \cref{cons:forgetfulFunctor}, we even have a commuting square of $\baseCat$--Bousfield localisations
\begin{equation}\label{WCatParametrisedAdjunction}
    \begin{tikzcd}
         \underline{\cmonoid}_{\baseCat}(\mathrm{WCat}) \rar[shift left = 1.5, "\underline{I}_T"] \dar["\forget"']&  \underline{\cmonoid}_{\baseCat}(\mathrm{Cat})\lar[shift left = 1.5, hook, "\underline{\mathrm{incl}}"]\dar["\forget"]\\
         \cofree_{\baseCat}(\mathrm{WCat})\rar[shift left = 1.5, "\underline{I}_T"]  & \cofree_{\baseCat}(\cat)\lar[shift left = 1.5, hook, "\underline{\mathrm{incl}}"]
    \end{tikzcd}
\end{equation}
It is straightforward to see that our hypotheses on the pair $(\underline{\sC},S)\in \cofree_{\baseCat}(\mathrm{WCat})$ ensures that it lifts to an object in $\underline{\cmonoid}_{\baseCat}(\mathrm{WCat})$. Furthermore, recall that $\baseCat$--symmetric monoidal categories are equivalently $\baseCat$--commutative monoids in $\underline{\mathrm{Cat}}_{\baseCat}$ by \cref{thm:nardinShahEnvelopes} (1).

Given these, part (1) may now be obtained exactly by the argument in \cite[Lem. 3.6]{GGN} (which also saw an immediate parametrised adaptation in \cite[Lem. 4.2.3]{kaifPresentable}). For part (2), we would like to argue that for any $\baseCat$--category $\underline{\E}$, the map 
$\mathrm{DK}^* : \myuline{\map}_{\cat_{\baseCat}}\big(\underline{\E},\myuline{\func}_{\baseCat}^{\underline{\otimes}}(S^{-1}\underline{\sC}^ {\underline{\otimes}}, \underline{\D}^{\underline{\otimes}})\big) \rightarrow \myuline{\map}_{\cat_{\baseCat}}\big(\underline{\E},\myuline{\func}_{\baseCat}^{\underline{\otimes}, S^{-1}}(\underline{\sC}^{\underline{\otimes}}, \underline{\D}^{\underline{\otimes}})\big)$
is an equivalence. But by the universal property of the pointwise $\baseCat$--symmetric monoidal structure from \cref{recollect:universalPropertyPointwiseMonoidal}, we may rewrite the domain and codomain respectively as $\myuline{\map}_{\baseCat}^{\underline{\otimes}}(S^{-1}\underline{\sC}^ {\underline{\otimes}}, \underline{\func}(\underline{\E},\underline{\D}^{\underline{\otimes}}))$ and $\myuline{\map}_{\baseCat}^{\underline{\otimes}, S^{-1}}(\underline{\sC}^{\underline{\otimes}}, \underline{\func}(\underline{\E},\underline{\D}^{\underline{\otimes}}))$. Under these identifications, the map displayed above is then an equivalence by the top $\baseCat$--adjunction in \cref{WCatParametrisedAdjunction}.
\end{proof}

\begin{rmk}\label{rmk:replaceCoproductClosureWithColimitClosure}
    For our multiplicative considerations shortly, the following notations and observations will be important. For a $\baseCat$--collection of morphisms $S$ in a $\baseCat$--presentable category $\underline{\sC}$, we write $S_{\underline{\amalg}}\supseteq S$ for the closure of $S$ under finite indexed coproducts of morphisms in $S$. Similarly, we denote by $S_{\underline{\colim}}\supseteq S_{\underline{\amalg}}\supseteq S$ the closure of $S$ under \textit{all } parametrised colimits valued in $S$. These are easily seen to be, again, $\baseCat$--collections of morphisms. We write $S_{\underline{\colim}}^{-1}\underline{\sC}, S_{\underline{\amalg}}^{-1}\underline{\sC}\subseteq \underline{\sC}$ for the $\baseCat$--full subcategories of $S_{\underline{\colim}}$-- and $S_{\underline{\amalg}}$--local objects in $\underline{\sC}$ respectively (this is consistent with \cref{nota:dwyerKanAndBousfieldLocalisations} by virtue of \cite[Thm. 6.3.7]{kaifPresentable}). Since $S_{\underline{\colim}}\supseteq S_{\underline{\amalg}}$, we have the inclusion $S_{\underline{\colim}}^{-1}\underline{\sC}\subseteq S_{\underline{\amalg}}^{-1}\underline{\sC}$. To see that this inclusion is an equivalence, let $X\in S^{-1}_{\underline{\amalg}}\underline{\sC}$ and let $\partial\colon \underline{J}\rightarrow \underline{\sC}^{\Delta^1}$ be a diagram taking values in $S$. We need to show that $X$ is local against the morphism $\underline{\colim}_{\underline{J}} \colon \underline{\colim}_{\underline{J}}\partial_0\rightarrow \underline{\colim}_{\underline{J}}\partial_1$ in $S_{\underline{\colim}}$. But this is clearly implied by the commutation $\myuline{\map}(\underline{\colim}_{\underline{J}}\partial_1,X)\simeq \underline{\lim}_{\underline{J}\vop}\myuline{\map}(\partial_1,X)$ and the analogous ones for $\underline{\colim}_{\underline{J}}\partial_0$ and $\underline{\colim}_{\underline{J}}\partial$. 
\end{rmk}

\begin{lem}\label{lem:coproductClosureNormClosed}
    Let $\underline{\sC}$ be a $\baseCat$--presentable category equipped with a $\baseCat$--distributive--symmetric--monoidal structure $\underline{\sC}\totimes$. Let $S$ be a $\baseCat$--collection of morphisms in $\underline{\sC}$ and $S_{\underline{\colim}}\supseteq S$ its closure from \cref{rmk:replaceCoproductClosureWithColimitClosure}. If $S$ is $\underline{\otimes}$--multiplicatively closed, then so is $S_{\underline{\colim}}$.
\end{lem}
\begin{proof}
    Let $p\colon W\rightarrow U$ be a morphism in $\finite_{\baseCat}$ and suppose we have a diagram of morphisms $\partial\colon \underline{J}\rightarrow p^*\underline{\sC}^{\Delta^1}$ in $S$. Our goal is to show that the morphism $p_{\otimes}\underline{\colim}_{\underline{J}}\partial$ is still contained in $S_{\underline{\colim}}$. By $\baseCat$--distributivity, we have the equivalence
    \[p_{\otimes}\underline{\colim}_{\underline{J}}\partial\simeq \underline{\colim}\big(p_*\underline{J}\xlongrightarrow{p_*\partial} p_*p^*\underline{\sC}^{\Delta^1}\xlongrightarrow{p_{\otimes}} \underline{\sC}^{\Delta^1}\big)\]
    where we have endowed $\underline{\sC}^{\Delta^1}$ with the pointwise $\baseCat$--symmetric monoidal structure from \cite[$\S3.3$]{nardinShah}, which is again $\baseCat$--distributive. Since $S$ is $\underline{\otimes}$--multiplicatively closed, the right hand side is in $S_{\underline{\colim}}$, and so $p_{\otimes}\underline{\colim}_{\underline{J}}\partial$ is too, as required.
\end{proof}

\begin{prop}\label{parametrisedPresentableDwyerKanLocalisation}
    Let $\underline{\sC}$ be a $\baseCat$--presentable category equipped with a $\baseCat$--distributive--symmetric--monoidal structure $\underline{\sC}\totimes$. Let $S$ be a $\baseCat$--collection of morphisms in $\underline{\sC}$.  
    \begin{enumerate}
        \item There is a $\baseCat$--presentable category $L_S\underline{\sC}$ participating in  a $\baseCat$--Bousfield localisation
        \begin{center}
            \begin{tikzcd}
                L\colon \underline{\sC} \rar[shift left = 1] & L_S\underline{\sC} \lar[shift left = 1, hook] : i
            \end{tikzcd}
        \end{center}
        satisfying the following universal property: for any other $\baseCat$--presentable category $\underline{\D}$, the map $L$ induces the equivalence
        $L^*\colon \underline{\func}^L(L_S\underline{\sC},\underline{\D}) \xlongrightarrow{\simeq} \underline{\func}^{L,S^{-1}}(\underline{\sC},\underline{\D})$.

        \item If $S$ was furthermore $\underline{\otimes}$--multiplicatively closed, then there is a canonical enhancement $L_S\underline{\sC}\totimes$ of $L_S\underline{\sC}$ to the structure of a $\baseCat$--symmetric monoidal category. This is uniquely characterised by the following universal property: for any $\baseCat$--cocomplete $\baseCat$--symmetric monoidal category $\underline{\D}\totimes$, this functor induces the equivalence  $L^*\colon \underline{\func}^{\underline{\otimes},L}(L_S\underline{\sC},\underline{\D}) \xlongrightarrow{\simeq} \underline{\func}^{\underline{\otimes},L,S^{-1}}(\underline{\sC},\underline{\D})$. 
    \end{enumerate}
\end{prop}
\begin{proof}
    Part (1) is an immediate consequence of \cite[Thm. 6.3.7]{kaifPresentable}, using $L_S\underline{\sC}\coloneqq S^{-1}_{\underline{\amalg}}\underline{\sC}$. For part (2), we recall that $S_{\underline{\colim}}^{-1}\underline{\sC}\simeq S^{-1}_{\underline{\amalg}}\underline{\sC}$ from \cref{rmk:replaceCoproductClosureWithColimitClosure}. Now, \cref{lem:coproductClosureNormClosed} ensures that $S_{\underline{\colim}}$ is $\underline{\otimes}$--multiplicatively closed, and so by \cref{TDwyerKanSymmetricMonoidality} (1) we obtain the first sentence of (2). The final sentence is now an immediate combination of the universal property of part (1) and \cref{TDwyerKanSymmetricMonoidality} (2).
\end{proof}

\subsection{Borel equivariant theory}\label{subsection:BorelEquivariantTheory}
In this subsection, we specialise for the moment to the case of $\baseCat=\orbit_G$. We write $b\colon BG\hookrightarrow \orbit_G$ for the fully faithful inclusion of the free transitive $G$--set $G/e$. This subsection pertains to working out some far--reaching consequences of constructions associated to this one map just from adjunction considerations. The philosophy here is to seriously reckon with the fact that the process of forgetting ``genuine'' $G$--structures to ``Borel'' $G$--structures is one that penetrates and appears on multiple categorical levels. In this sense, the ``categorification--decategorification'' approach we present here is very much in line with the so--called ``metacosm--macrocosm--microcosm'' trichotomy of \cite{StratifiedNoncommGeometry}. Our end point will be the full version of \cref{alphaThm:monoidalBorelificationPrinciple} in the form of \cref{monoidalBorelificationPrinciple} and our starting point is the  Bousfield localisation
\begin{equation}\label{eqn:LargeCatBorelificationLocalisation}
            \begin{tikzcd}
                b^* \colon \widehat{\cat}_G=\func(\orbit_G\op,\widehat{\cat}) \rar[shift left = 1] & \func(BG,\widehat{\cat}) : b_*\lar[shift left =1 , hook] 
            \end{tikzcd}
\end{equation}

The $G$--category of small $G$--categories $\underline{\cat}_G$ is then an object of $\widehat{\cat}_G$. Due to its special role throughout, we denote $b_*b^*\underline{\cat}_G\in\widehat{\cat}_G$ by $\underline{\borel}(\cat)$ for \textit{Borel $G$--categories}. By the right Kan extension formula, we see that under the embedding $b_*$ from \cref{eqn:LargeCatBorelificationLocalisation}, $\underline{\borel}(\cat)$ is the $G$--category whose value at $G/H$ is given by $\cat^{hH}\simeq \func(BH,\cat)$. Importantly, the adjunction unit $\underline{\cat}_G\rightarrow \underline{\borel}(\cat)$ can be checked easily to be given precisely \textit{again} by the map $b^*$. By right Kan extension, the $G$--functor $b^*\colon \underline{\cat}_G\rightarrow \underline{\borel}(\cat)$ admits fibrewise fully faithful right adjoints $b_*$. Over every fibre, objects in $\underline{\borel}(\cat)$ may then be viewed via $b_*$ precisely as those objects in $\underline{\cat}$ which are \textit{Borel local}, i.e. those $\underline{\sC}\in\underline{\cat}$ with the property that if $\varphi \colon \underline{I}\rightarrow\underline{J}$ is a map in $\underline{\cat}$ such that $b^*\varphi$ is an equivalence (such maps are also called Borel equivalences), then $\varphi^*\colon {\map}_{\underline{\cat}}(\underline{J}, \underline{\sC})\rightarrow {\map}_{\underline{\cat}}(\underline{I}, \underline{\sC})$ is an equivalence. Similarly, $b^*$ also admits fibrewise fully faithful left adjoints $b_!$ by left Kan extension. Under this inclusion, $\underline{\borel}(\cat)$ may be viewed as the \textit{Borel colocal} objects, i.e. those $\underline{\sC}\in\underline{\cat}$ such that $\varphi_*\colon {\map}_{\underline{\cat}}(\underline{\sC},\underline{I})\rightarrow {\map}_{\underline{\cat}}(\underline{\sC},\underline{J})$ is an equivalence for any Borel equivalence $\varphi\colon \underline{I}\rightarrow\underline{J}$. The first observation to be made is that these can be assembled to a $G$--Bousfield (co)localisation by the following:

\begin{prop}\label{prop:masterBorelificationLocalisation}
    The $G$--functor $b^*$ participates in the $G$--Bousfield (co)localisation \begin{equation*}
            \begin{tikzcd}
                \underline{\cat}_G \ar[rr,"b^*" description] & &\underline{\borel}(\cat). \ar[ll, shift left =2 , hook, "b_*"] \ar[ll,shift right = 2, hook, "b_!"' ] 
            \end{tikzcd}
    \end{equation*}
\end{prop}
\begin{proof}
    We would like to use \cref{criteriaForTAdjunctions} and its dual to say that the fibrewise fully faithful right/left adjoints assemble to a $G$--functor. For the case of $b_*$, without loss of generality, we show that for every $H\leq G$, the adjointed square
    \begin{center}
        \begin{tikzcd}
            \cat_G \dar["\res^G_H"']\ar[dr,phantom, "\Rightarrow"]& \func(BG,\cat)\dar["\res^G_H"]\lar["b_*"', hook]\\
            \cat_H & \func(BH,\cat)\lar["b_*"', hook]
        \end{tikzcd}
    \end{center}
    commutes. To this end, we show the following: if $\underline{\sC}\in\cat_G$ is Borel local, then $\res^G_H\underline{\sC}\in\cat_H$ is also Borel local. So let $\varphi \colon \underline{I}\rightarrow\underline{J}$ be a Borel equivalence in $\cat_H$. Since the functor $\forget\colon \func(BH,\cat)\rightarrow\cat$ is conservative, this is the same as requiring that $\res^H_e\varphi$ is an equivalence. Because of this, and because $\res^G_e\coprod_{G/H}\simeq \coprod_{g\in G/H}\res^H_e$, we see that $\coprod_{G/H}\varphi\colon \coprod_{G/H}\underline{I}\rightarrow\coprod_{G/H}\underline{J}$ is still a Borel equivalence. Hence, by the computation
    \[\map_{\cat_H}(\underline{J},\res^G_H\underline{\sC})\simeq \map_{\cat_G}(\coprod_{G/H}\underline{J},\underline{\sC})\xrightarrow[\simeq]{\varphi^*}\map_{\cat_G}(\coprod_{G/H}\underline{I},\underline{\sC})\simeq\map_{\cat_H}(\underline{I},\res^G_H\underline{\sC})  \]
    we see that $\res^G_H\underline{\sC}$ is still Borel local, as claimed. The case of $b_!$ is similar, using now instead that $\res^G_e\prod_{G/H}\simeq \prod_{g\in G/H}\res^H_e$, so that $\prod_{G/H}\varphi$ is still a Borel equivalence.
\end{proof}

From now on,  the word \textit{Borelification} we will mean either the functor $b^*$ or $b_*b^*$ in all its various incarnations.

\begin{obs}\label{obs:borelificationAsUniqueLift}
    A point that we will be using several times to prove properties relating to $b^*$ is that it factorises the functor $\res^G_e\colon \cat_G\rightarrow \cat$ uniquely as $\cat_G\xrightarrow{b^*}\func(BG,\cat)\xrightarrow{\forget}\cat$ (with trivial $G$--equivariance everywhere).
\end{obs}

\begin{prop}\label{prop:omnibusBorelBasics}
    Let $\underline{J}\in\cat_G$, $\underline{\sC}\in\underline{\cat}_G$ and $\D\in\underline{\borel}(\cat)$. 
    \begin{enumerate}
        \item There is a natural equivalence $b^*\underline{\func}(\underline{J},\underline{\sC})\simeq \func(b^*\underline{J},b^*\underline{\sC})$ of objects in $\underline{\borel}(\cat)$,
        \item the adjunction unit $\underline{\func}(\underline{J},b_*\D)\rightarrow b_*b^*\underline{\func}(\underline{J},b_*\D)$ is an equivalence,

        \item if $\D$ has finite (co)products, then $b_*\D\in\cat_G$ is a $G$--category with finite indexed (co)products. In the case of admitting products, we furthermore have a natural equivalence $\underline{\cmonoid}_G(b_*\D)\simeq b_*\cmonoid(\D)$,
        \item if $\D\in\cat$ is moreover semiadditive, then $b_*\D\in{\cat}_G$ is a $G$--semiadditive $G$--category.
    \end{enumerate}
\end{prop}
\begin{proof}
For part (1), note that we have both commuting squares
\begin{center}
        \begin{tikzcd}
            \cat_G\rar["\res^G_e"] \dar["{\underline{\func}(\underline{J},-)}"']& \cat\dar["{\func(\res^G_e\underline{J},-)}"] &\func(BG,\cat)\ar[rr,"{\func(b^*\underline{J},-)}"] \dar["\forget"]& &\func(BG,\cat)\dar["\forget"]\\
            \cat_G\rar["\res^G_e"] & \cat  & \cat\ar[rr,"{\func(\res^G_e\underline{J},-)}"] && \cat
        \end{tikzcd}
    \end{center}
    Here, all categories and functors in the left square are endowed with the trivial $G$--equivariant structure. Hence, by \cref{obs:borelificationAsUniqueLift}, the bottom left composite of the left square lifts to the functor $b^*\underline{\func}(\underline{J},-)\colon \cat_G\rightarrow\func(BG,\cat)$. But the same observation together with the right commuting square yields that the unique lift to $\func(BG,\cat)$ of the top right composite of the left square is given by $\func(b^*\underline{J},b^*-)\colon \cat_G\rightarrow \func(BG,\cat)$. This gives part (1).

    For part (2), we just need to show that $\underline{\func}(\underline{J},b_*\D)$ is Borel local, and this  is an immediate consequence of the definition and unwinding adjunctions, using also the fact that $b^*$ preserves products (for example because it has a left adjoint $b_!$). 
    
    For point (3), we deal with the case of products since that of coproducts will be dual. Let $H\leq G$ and $w\colon G/H\rightarrow G/G$ be the unique equivariant map. This then induces a $G$--equivariant map $w^*\colon \D \simeq \func(G/G,\D)\rightarrow \func(G/H,\D)\simeq \prod_{G/H}\res^G_H\D$. Since $\D$ has finite products and since $w^*$ was $G$--equivariant, this map admits a right adjoint $w_*$ which furthermore can be endowed with a canonical $G$--equivariant structure with $G$--equivariant (co)units by \cref{fact:adjointOfEquivariantAdjunctions}. In particular, this induces, upon applying $(-)^{hG}$, the adjunction
    \begin{center}
        \begin{tikzcd}
            (b_*\D)^G = \D^{hG} \rar[shift left = 1, "w^*"] &(\prod_{G/H}\res^G_H\D)^{hG}\simeq \D^{hH} = (w_*w^*b_*\D)^G \lar[shift left = 1, "w_*"]
        \end{tikzcd}
    \end{center}
    To see that this can be assembled to a $G$--adjunction $w^*\colon b_*\D \rightleftharpoons w_*w^*b_*\D : w_*$, we need to show by \cref{criteriaForTAdjunctions} that $w_*$ commutes with restrictions under the appropriate Beck--Chevalley transformation. The key to this is the double coset decomposition of finite $G$--sets. More precisely, let $K\leq G$ and $f\colon G/K\rightarrow G/G$. Then by the double coset decomposition we have the left pullback diagram of finite $G$--sets
    \begin{center}\small
        \begin{tikzcd}
            \coprod_{g\in K\backslash G/H}G/K^g\cap H \rar\dar \ar[dr, phantom,very near start, "\lrcorner"]& G/H\dar["w"] &\prod_{g\in K\backslash G/H}\prod_{G/K^g\cap H}\res^G_{G/K^g\cap H}\D \dar["w_*"']\ar[dr, phantom, "\Leftarrow", yshift = 2, xshift = -2] & \prod_{G/H}\res^G_H\D\dar["w_*"]\lar["f^*"']\\
            G/K\rar["f"] & G/G & \prod_{G/K}\res^G_K\D  & \D\lar["f^*"']
        \end{tikzcd}
    \end{center}\normalsize
    Applying $\func(-,\D)$, we obtain the adjointed square on the right of $G$--equivariant maps and transformation, again by \cref{fact:adjointOfEquivariantAdjunctions}. This commutes by \cite[Lem. 6.1.6.3]{lurieHA}. Finally, applying $(-)^{hG}$ to the right square gives us the commutation of the Beck--Chevalley square in the aforementioned criterion, as required.
    
    For the second part of (3), since $\res^G_e\underline{\cmonoid}_G(b_*\D)\simeq \cmonoid(\D)\in\cat$, we see by the same argument as for part (1)  that $b^*\underline{\cmonoid}_G(b_*\D)\simeq \cmonoid(\D)\in\func(BG,\cat)$. Hence, we need only to argue that $\underline{\cmonoid}_G(b_*\D)$ is Borel local. So let $\varphi \colon \underline{I}\rightarrow\underline{J}$ be a Borel equivalence. Recall from \cref{nota:TCommutativeMonoidsDefinition} that $\underline{\cmonoid}_G(b_*\D)\coloneqq \underline{\func}^{\times}(\myuline{\effBurn}(\underline{\finite}),b_*\D)$. Now, consider
    \[\underline{\func}(\underline{J},\underline{\func}^{\times}(\myuline{\effBurn}(\underline{\finite}),b_*\D))\simeq\underline{\func}^{\times}(\myuline{\effBurn}(\underline{\finite}),\underline{\func}(\underline{J},b_*\D))\simeq\underline{\func}^{\times}(\myuline{\effBurn}(\underline{\finite}),b_*\func(b^*\underline{J},\D))\] to conclude, where the second equivalence is by points (1) and (2).

    Lastly, point (4) is simply because, as explained in the proof of part (3), the indexed (co)product for $b_*\D$ are induced by the left and right adjoint of the $G$--equivariant restriction functor $\D \rightarrow \prod_{G/H}\res^G_H\D$. Hence, forgetting the $G$--equivariant structure on the $G$--semiadditivity norm map \cref{eqn:semiadditivityNormMap} gives the usual semiadditivity norm map associated to the set of size $|G/H|$, which is an equivalence by hypothesis.
\end{proof}

\begin{prop}\label{monoidalBorelificationInTheLarge}
Applying the transformation $\forget\colon \underline{\cmonoid}_G(-)\rightarrow \id(-)$ from \cref{cons:forgetfulFunctor} to the $b^*\dashv b_*$ $G$--Bousfield localisation of \cref{prop:masterBorelificationLocalisation} yields the commuting square of $G$--Bousfield localisations (i.e. both the $(b^*,\forget)$ and $(b_*,\forget)$ squares commute)
\begin{equation}\label{eqn:borelificationForgetSquare}
    \begin{tikzcd}
        \underline{\cmonoid}_G(\underline{\cat}) \rar[shift left = 1, "b^*"] \dar["\forget"']& \underline{\borel}(\cmonoid(\cat))\lar[shift left =1, "b_*",hook]\dar["\forget"]\\
        \underline{{\cat}}_G \rar[shift left =1, "b^*"] &\underline{\borel}(\cat).\lar[shift left =1, "b_*",hook]
    \end{tikzcd}
\end{equation}
\end{prop}
\begin{proof}
By $b_!\dashv b^*\dashv b_*$ from \cref{prop:masterBorelificationLocalisation}, both $b^*$ and $b_*$ preserve finite indexed products. Since $\underline{\cmonoid}_G(-)$ is functorial on finite indexed product--preserving functors and since the $\forget$ transformation is implemented by precomposition from \cref{cons:forgetfulFunctor}, we immediately obtain a commuting square of $G$--Bousfield localisations. All that is left to argue is that $\underline{\cmonoid}_G(\underline{\borel}(\cat))\simeq \underline{\borel}(\cmonoid(\cat))$: this is precisely supplied by \cref{prop:omnibusBorelBasics} (3).
\end{proof}

\begin{rmk}
    Recall that by definition, $\mackey_G(\cat)\coloneqq \func^{\times}(\effBurn(G),\cat)$. Hence, we have $\mackey_G(\cat)\simeq \mackey_G(\cmonoid(\cat))$ since $\effBurn(G)$ is semiadditive by \cite[Prop. 4.3, Ex. B]{barwick1}. Evaluating the top horizontal $G$--adjunction from \cref{eqn:borelificationForgetSquare} at $G/G$ and using that $\cmonoid_G(\underline{\cat})\simeq \mackey_G(\cat)$ from \cref{thm:nardinShahEnvelopes} yields the Bousfield localisation
        \begin{equation*}
            \begin{tikzcd}
                b^* \colon \mackey_G(\cat) \simeq \mackey_G(\cmonoid(\cat)) \rar[shift left = 1] & \func(BG,\cmonoid(\cat)) : b_*\lar[shift left =1 , hook] 
            \end{tikzcd}
        \end{equation*}
     This then recovers \cite[Cons. 8.1]{barwick2}, and one may indeed view \cref{monoidalBorelificationInTheLarge} as a parametrised enhancement of the cited result.
\end{rmk}

\begin{obs}\label{obs:concreteDescriptionOfBorelNorms}
    Suppose we are given an object $\D^{\otimes}\in\func(BG,\cmonoid(\cat))$. Then by \cref{thm:nardinShahEnvelopes} (1), we know that $b_*\D^{\otimes} \in \cmonoid_G(\underline{\cat})\simeq \mackey_G(\cat)$ may now be viewed as a $G$--symmetric monoidal category. We explain now a concrete description of the multiplicative norm functors on $b_*\D^{\otimes}$. First of all, because $\cmonoid(\cat)$ is semiadditive, we get that $\underline{\borel}(\cmonoid(\cat))$ is $G$--semiadditive by \cref{prop:omnibusBorelBasics} (4). In particular,  for any $G$--object $\D^{\otimes}\in\underline{\borel}(\cmonoid(\cat))$ and for any $H\leq G$, we have the counit $\prod_{G/H}\res^G_H\D^{\otimes}\xlongrightarrow{\varepsilon_{\D}} \D^{\otimes}$
    from the adjunction $\prod_{G/H}\simeq \coprod_{G/H} \colon \func(BH,\cmonoid(\cat))\rightleftharpoons \func(BG,\cmonoid(\cat)) : \res^G_H$. As indicated in the proof of \cref{prop:omnibusBorelBasics} (4), upon forgetting the $G$--actions, the adjunction counit is given by the multiplication map $\otimes\colon \prod_{|G/H|}\D \rightarrow\D$. Now note that an object in $\prod_{G/H}\res^G_H\D$ admits an explicit description given by a tuple $(X_g)_{g\in G/H}$ and so an object in $(\prod_{G/H}\res^G_H\D)^{hG}\simeq \D^{hH}$ can be described also as the tuple $(gX)_{g\in G/H}$ where $X\in\D^{hH}$. Since $\varepsilon_{\D}$ is $G$--equivariant, we then get the map
    \[(\prod_{G/H}\res^G_H\D)^{hG}\simeq \D^{hH}\xlongrightarrow{\varepsilon_{\D}^{hG}} \D^{hG}\quad::\quad X\mapsto \bigotimes_{g\in G/H}gX\] 
    
    Now, by virtue of the commuting square of $G$--Bousfield localisations from \cref{monoidalBorelificationInTheLarge}, we get from \cref{prop:adjointabilityPreservesAdjunction} an identification of $b_*\varepsilon_{\D}$ with the counit $\prod_{G/H}\res^G_Hb_*\D^{\otimes}\xlongrightarrow{\varepsilon_{b_*\D}} b_*\D^{\otimes}$ in $\underline{\cmonoid}_G(\underline{\cat})$.    
    Since this adjunction counit \textit{is} the map encoding the multiplicative norm of an object in ${\cmonoid}_G(\underline{\cat})\simeq \cat_G^{\otimes}$, all in all, we see that the $G$--symmetric monoidal category $b_*\D^{\otimes}$ has $G/H$--norms given by the formula $X\mapsto \bigotimes_{g\in G/H}gX$.
\end{obs}

Because pointwise symmetric monoidal structures are merely cotensors instead of being internal hom objects, the next two lemmas will require a proof separate from, but very much in the spirit of, that of \cref{prop:omnibusBorelBasics}.
\begin{lem}\label{lem:borelificationOfPointwiseMonoidal}
    Let $\underline{J}\in\cat_G$ and $\underline{\D}\totimes\in \underline{\cmonoid}_G(\underline{\cat})$. Then we have an equivalence $b^*\underline{\func}(\underline{J},\underline{\D}\totimes)\simeq \func(b^*\underline{J},b^*\underline{\D}\totimes)$ in $\func(BG,\cmonoid(\cat))$. 
\end{lem}
\begin{proof}
    By \cite[Prop. 9.7]{shahThesis}, we see that the pairing construction of the pointwise $G$--symmetric monoidal structure from \cite[$\S3.3$]{nardinShah}  is compatible with restrictions (for an example argument, see the proof of \cite[Prop. 3.2.2]{nardinShah}). Hence, we have the left commuting square 
    \begin{center}
        \begin{tikzcd}
            {\cmonoid}_G(\underline{\cat})\rar["\res^G_e"] \dar["{\underline{\func}(\underline{J},-)}"']& \cmonoid(\cat)\dar["{\func(\res^G_e\underline{J},-)}"] &\cmonoid(\cat)^{BG}\ar[rr,"{\func(b^*\underline{J},-)}"] \dar["\forget"]& &\cmonoid(\cat)^{BG}\dar["\forget"]\\
            {\cmonoid}_G(\underline{\cat})\rar["\res^G_e"] & \cmonoid(\cat)  & \cmonoid(\cat)\ar[rr,"{\func(\res^G_e\underline{J},-)}"] && \cmonoid(\cat)
        \end{tikzcd}
    \end{center}
    Since the category $\func(BG,\cmonoid(\cat))$ is also cotensored over $\func(BG,\cat)$ under the pointwise symmetric monoidal structure, we also get the right commuting square above.
    Now all maps in sight in the left commuting square are moreover $G$--equivariant with the trivial $G$--equivariant structures on all the categories. Hence, the bottom composite of the left square lifts uniquely to the map $b^*\underline{\func}(\underline{J},-)\colon {\cmonoid}_G(\underline{\cat})\rightarrow\func(BG,\cmonoid(\cat))$ whereas the top composite of the left square lifts uniquely to a map $\func(b^*\underline{J},b^*-)\colon {\cmonoid}_G(\underline{\cat})\rightarrow \func(BG,\cmonoid(\cat))$ (using also the existence of the right commuting square). Hence, all in all, we have obtained an equivalence of functors $b^*\underline{\func}(\underline{J},-)\simeq \func(b^*\underline{J},b^*-)$ as required.
\end{proof}

Recall our notation of $s\colon\ast \hookrightarrow \orbit_G$ being the inclusion of  $G/G$ from \cref{obs:leftKanExtensionAsA2Adjunction}.

\begin{lem}\label{lem:borelificationCommutesWithPointwiseMonoidal}
    Let $\D^{\otimes}\in\func(BG,\cmonoid(\cat))$ and $I\in\cat$. Then there is a natural map $\underline{\func}(s_!I,b_*\D^{\otimes})\rightarrow b_*\func(I,\D^{\otimes})$ in ${\cmonoid}_G(\underline{\cat})$ which is an equivalence.
\end{lem}
\begin{proof}
    Since $b^*s_!\simeq\trivial_G\colon \cat\rightarrow\func(BG,\cat)$ and $b^*b_*\simeq \id$, by \cref{lem:borelificationOfPointwiseMonoidal} we have an equivalence $b^*\underline{\func}(s_!I,b_*\D^{\otimes})\simeq \func(I,\D^{\otimes})$. Adjointing this over we obtain the claimed map $\underline{\func}(s_!I,b_*\D^{\otimes})\rightarrow b_*\func(I,\D^{\otimes})$ in ${\cmonoid}_G(\underline{\cat})$. Since the composite ${\cmonoid}_G(\underline{\cat})\rightarrow\cat_G\xrightarrow{\prod_{H\leq G}(-)^H}\prod_{H\leq G}\cat$ is conservative, it is enough to show that the said map is an equivalence upon passing to fixed points for all subgroups of $G$. In order to keep notations to a minimum, we may just show it for $G$--fixed points since the other fixed points can be dealt with similarly after restriction. In this case, we need to show that the map 
    $s^*\underline{\func}(s_!I,b_*\D)\longrightarrow s^*b_*\func(I,\D)$ in $\cat$
    is an equivalence. The target is given by $s^*b_*\func(I,\D)\simeq \func(I,\D)^{hG}$, whereas the source is, via \cref{obs:leftKanExtensionAsA2Adjunction}, given by $s^*\underline{\func}(s_!I,b_*\D)\simeq \func(I,s^*b_*\D)\simeq \func(I,\D^{hG})$, and it can be checked easily that the map is also the canonical one implementing the limit exchange equivalence $\func(I,\D^{hG})\xrightarrow{\simeq} \func(I,\D)^{hG}$. This concludes the proof.
\end{proof}

\begin{prop}\label{prop:symmetricMonoidalFunctorsToBorelifications}
    Let $\underline{\sC}\totimes\in{\cmonoid}_G(\underline{\cat})$ and $\D^{\otimes}\in\func(BG,\cmonoid(\cat))$. There is a natural equivalence of categories $\func\totimes_G(\underline{\sC}\totimes, b_*\D^{\otimes}) \simeq \func^{\otimes}(b^*\underline{\sC}\totimes, \D^{\otimes})^{hG}$.
\end{prop}
\begin{proof}
    Fix $I \in\cat$. Consider the following sequence of natural equivalences
    \begin{equation*}
        \begin{split}
            \map_{\cat}\big(I,s^*\underline{\func}\totimes(\underline{\sC}\totimes,b_*\D^{\otimes})\big) & \simeq \map_{\cat_G}\big(s_!I,\underline{\func}\totimes(\underline{\sC}\totimes,b_*\D^{\otimes})\big)\\
            &\simeq \map_{{\cmonoid}_G(\underline{\cat})}\big(\underline{\sC}\totimes, \underline{\func}(s_!I, b_*\D^{\otimes})\big)\\
            &\simeq \map_{{\cmonoid}_G(\underline{\cat})}\big(\underline{\sC}\totimes, b_*\func(I,\D^{\otimes})\big)\\
            &\simeq \map_{\cmonoid(\cat)^{BG}}\big(b^*\underline{\sC}\totimes, \func(I,\D^{\otimes})\big)\\
            &\simeq \map_{\cmonoid(\cat)}\big(b^*\underline{\sC}\totimes, \func(I,\D^{\otimes})\big)^{hG}\\
            &\simeq \map_{\cat}\big(I,\func^{\otimes}(b^*\underline{\sC}\totimes, \D^{\otimes})^{hG}\big)
        \end{split}
    \end{equation*}
    where the second and sixth equivalences are by the universal property from \cref{recollect:universalPropertyPointwiseMonoidal}, the third by \cref{lem:borelificationCommutesWithPointwiseMonoidal}, and the fifth by the mapping space formula in $\cmonoid(\cat)^{BG}$. Since by definition, $s^*\underline{\func}\totimes(\underline{\sC}\totimes,b_*\D^{\otimes})= \func\totimes_G(\underline{\sC}\totimes, b_*\D^{\otimes})$, we obtain the desired equivalence.
\end{proof}

We now distil all that we have done in this subsection into the following principle which establishes an abstract but very important link between $G$--categories and their underlying category with $G$--action. We will use the notation $\underline{\borel}$ for the functor $b_*\colon \underline{\borel}(\cat) \hookrightarrow \underline{\cat}_G$ for intuitive appeal.  We thank Asaf Horev for discussions leading to it, especially in teaching us the trick of using symmetric monoidal envelopes.

\begin{thm}[Monoidal Borelification principle]\label{monoidalBorelificationPrinciple}\label{asafCalgMonad}
Let $\underline{\sC}\totimes \in {\cmonoid}_G(\underline{\cat})\simeq \mackey_G(\cat)$ be a $G$--symmetric monoidal category and $\D^{\otimes}\in \func(BG,\cmonoid({\cat}))$ be a symmetric monoidal category with a $G$--action. Then:
\begin{enumerate}
    \item The $G$--category $\underline{\borel}(\D)\in\cat_G$ canonically refines to a $G$--symmetric monoidal category $\underline{\borel}(\D^{\otimes})\in{\cmonoid}_G(\underline{\cat})$. The multiplicative norm map $\norm^G_H\colon \D^{hH}\rightarrow\D^{hG}$ can be concretely described as follows: for $X \in \D^{hH}$ a $H$--object in $\underline{\borel}(\D)$, the $G$--object $\norm^G_HX \in \D^{hG}$ is given by $\bigotimes_{g\in G/H}gX$,
    \item Writing $\sC_e\in\func(BG,\cat)$ for the value of $\underline{\sC}\in\cat_G$ at $G/e$, the adjunction unit $\underline{\sC}\rightarrow \underline{\borel}(\sC_e)$ of \cref{prop:masterBorelificationLocalisation} canonically refines to a $G$--symmetric monoidal functor $\underline{\sC}\totimes \rightarrow \underline{\borel}(\sC_e^{\otimes})$.
    \item There is a natural equivalence $\calg_G\big(\underline{\borel}(\D^{\otimes})\big)\simeq \calg(\D^{\otimes})^{hG}$.
\end{enumerate}
\end{thm}
\begin{proof}
    Part (1) is by the fact that the $(b_*,\forget)$--square in \cref{monoidalBorelificationInTheLarge} commutes, and the description of the norm is by \cref{obs:concreteDescriptionOfBorelNorms}. Part (2) is by the fact that we have a commuting square of $G$--Bousfield localisations from \cref{monoidalBorelificationInTheLarge}. Finally, note that
\begin{equation*}
    \begin{split}
        \calg_G(\underline{\borel}(\D^{\otimes})) & \simeq \func\totimes_G(\underline{\mathrm{Env}}(\underline{\finite}_*), \underline{\borel}(\D^{\otimes}))\simeq \func^{\otimes}(\mathrm{Env}(\finite_*), \D^{\otimes})^{hG}\simeq \calg(\D^{\otimes})^{hG}
    \end{split}
\end{equation*}
which gives part (3), where the first and last equivalences are by \cref{thm:nardinShahEnvelopes} (2) and the second equivalence is by \cref{prop:symmetricMonoidalFunctorsToBorelifications}.
\end{proof}


\begin{rmk}
    The structure of $G$--commutative algebra objects and morphisms thereof are often tricky to construct. Part (3) of the theorem guarantees us, however, that at least when the $G$--category involved is Borel, such algebras and their morphisms are nothing but algebras and morphisms in the underlying category equipped with a $G$--action. By further applying suitable $G$--lax symmetric monoidal functors on these Borel $G$--commutative algebras, we may construct out of them many interesting non--Borel examples of $G$--commutative algebras.  We refer the reader to \cref{borelificationPrincipleForCatPerf} for an illustration of this strategy.
\end{rmk}

\begin{rmk}
    While we have not pursued it here so as not to obfuscate the general exposition and since we will not be needing it for our purposes, we believe that the notion of Borel objects and its attendant monoidality theory above can be  developed more generally for any atomic orbital base category $\baseCat$ equipped with a full subcategory $b\colon \B\hookrightarrow \baseCat$ which is a \textit{sieve}, i.e. if we are given a morphism $C\rightarrow X$ in $\baseCat$ where $X\in\B$, then $C\in\B$ too. The main point is that this will allow the proof of \cref{prop:masterBorelificationLocalisation} to go through, from which much else should follow via careful analyses of all the notions involved.
\end{rmk}

\begin{rmk}
    The theorem above is a slight expansion and strengthening of \cite[Thm. 3.3.4, Prop. 3.3.6]{kaifThesis} from the author's thesis. Since then, an article \cite{lucyYang} of Lucy Yang's has appeared that gave a concrete description of $C_p$--$\mathbb{E}_{\infty}$--algebras which in particular also yields \cref{monoidalBorelificationPrinciple} (3) in the special of $G= C_p$. Another approach to proving the theorem has also since appeared in \cite{phil}. 
\end{rmk}

\subsection{Perfect--stable categories and Mackey functors}\label{subsec4.1:MackeyInclusion}\label{section:generationSplitKaroubiSequences}
\label{subsec:SemiadditivityStability}\label{sec4:perfectStables}

In this subsection, we will work out some basic categorical properties of the $\baseCat$--category $\underline{\cat}\Tperfect_{\baseCat}$ of $\baseCat${--perfect--stable categories}. We show that $\underline{\cat}\Tperfect_{\baseCat}$ is $\baseCat$--semiadditive--presentable, cf. \cref{cor:catPerfPresentability}. Moreover, we will show in \cref{mainMackeyInclusion} how this internally defined $\baseCat$--category relates to $\baseCat$--Mackey functors valued in $\cat\perfect$. This  will  allow us to transport results about split Karoubi sequences on $\cat\perfect$ to our setting as well as relate our notion of $\baseCat$--parametrised algebraic K--theory with, for instance, the one considered in \cite{barwick2,CMMN2}.

To begin, let us first record some formalities on $\baseCat$--semiadditivity. 

\begin{lem}\label{cor:indexedProductAdjunctionUnitIdentification}
    Suppose $\underline{\sC}$ is $\baseCat$--pointed with finite indexed (co)products. Let $f\colon U \rightarrow V$ be in $\underline{\finite}_{/V}$ and $Y\in \underline{\sC}$. We then have an identification $f_*f^*\eta_Y\simeq \eta_{f_*f^*Y}$ of maps $f_*f^*\eta_Y\simeq \eta_{f_*f^*Y}\colon f_*f^*Y \longrightarrow f_*f^*f_*f^*Y$.
\end{lem}
\begin{proof}
    This is an immediate consequence of \cref{prop:adjointabilityPreservesAdjunction}, using the adjointable square
    \begin{center}
        \begin{tikzcd}
            \underline{\sC} \rar["f_*f^*"]\dar["f^*"'] & \underline{\sC}\dar["f^*"']\\
            f_*f^*\underline{\sC} \rar["f_*f^*"] \ar[u, bend right = 50, dashed ,"f_*"']& f_*f^*\underline{\sC}\ar[u, bend right = 50, dashed ,"f_*"']
        \end{tikzcd}
    \end{center}
    coming from the fact that both the restriction functor $f^*\colon \underline{\sC}\rightarrow f_*f^*\underline{\sC}$ and indexed product functor $f_*\colon f_*f^*\underline{\sC}\rightarrow \underline{\sC}$ preserve indexed products.
\end{proof}

\begin{obs}\label{obs:retractionInAtomicOrbitals}
    Suppose $\underline{\sC}$ is $\baseCat$--pointed with finite indexed products. Let $f\colon U \rightarrow V$ be in $\underline{\finite}_{/V}$ and let $X\in f^*\underline{\sC}$. Observe that $X$ is a retract of $f^*f_*X$ since, by the pullback decomposition from \cref{eqn:complementationSquare}, we have $f^*f_*X\simeq X\times \overline{c}_*c^*X$. Hence,  using the map $\ast\rightarrow \overline{c}_*c^*X\rightarrow\ast$, we can get a retraction $X\simeq X\times \ast\longrightarrow f^*f_*X\simeq X\times \overline{c}_*c^*X\xlongrightarrow{\pi}  X\simeq X\times\ast$.
\end{obs}

\begin{cor}\label{cor:normMapBeingAnEquivalence}
    Let $\underline{\sC}$ be a $\baseCat$--pointed category with finite indexed products. Let $f\colon U \rightarrow V$ be in $\underline{\finite}_{/V}$ and let $X\in f^*\underline{\sC}$. The composite $f_*X \xlongrightarrow{\eta_{f_*X}}f_*f^*f_*X \xlongrightarrow{f_*\pi} f_*X$ is an equivalence.
\end{cor}
\begin{proof}
    Since retractions of equivalences are equivalences, by \cref{obs:retractionInAtomicOrbitals},  it suffices to show it when $X=f^*Y$ for some $Y\in\underline{\sC}$. In this case, this composite is identified with $f_*f^*Y \xlongrightarrow{f_*f^*\eta_{Y}} f_*f^*f_*f^*Y \xlongrightarrow{f_*\pi}f_*f^*Y$    by \cref{cor:indexedProductAdjunctionUnitIdentification}. By  unwinding  adjunctions and via the decomposition from \cref{eqn:complementationSquare},  $f^*\eta_Y\colon f^*Y\rightarrow f^*f_*f^*Y\simeq f^*Y\times \overline{c}_*c^*f^*Y$ is the identity map on the $f^*Y$ component in $f^*f_*f^*Y$. Therefore,  the composite is indeed an equivalence.
\end{proof}


With these generalities explained, we now begin our categorical study of $\underline{\cat}\Tperfect$ in earnest.

\begin{prop} \label{SemiadditivityOfPresentableStables}
The $\baseCat$--categories $\underline{\presentable}_{\baseCat, \mathrm{st}, L,\kappa}$ and $\underline{\presentable}_{\baseCat, L,\kappa}$ are $\baseCat$--semiadditive, where the $\baseCat$--products are created in $\underline{\widehat{\cat}}_{\baseCat}$. In particular, we have that $\underline{\cat}\Tperfect$ is $\baseCat$--semiadditive and the faithful inclusion $\underline{\cat}\Tperfect\subset \underline{\cat}$ is closed under finite $\baseCat$--products.
\end{prop}
\begin{proof}
We only show that $\underline{\presentable}_{\baseCat, L,\kappa}$ is $\baseCat$--semiadditive. This would then imply that the $\baseCat$--full subcategory $\underline{\presentable}_{\baseCat, \mathrm{st}, L,\kappa}$ is too, since $\baseCat$--presentable-stables are closed under $\baseCat$--products.

First, we show that $\underline{\presentable}_{\baseCat,L,\kappa}$ is ordinary semiadditive. By \cref{simpsonTheorem},  $\presentable_{\baseCat,L,\kappa}$ is the non--full subcategory of $\func(\baseCat\op,\presentable_{L,\kappa})$ consisting of those objects for which all the restriction functors have left adjoints and satisfy the left Beck--Chevalley condition, and the morphisms are those which satsify the left Beck--Chevalley condition. Since $\presentable_{L,\kappa}$ is semiadditive (cf. for instance \cite[Ex. 4.3.11]{hopkinsLurie}), it suffices now to argue that $\presentable_{\baseCat,L,\kappa}\subset \func(\baseCat\op,\presentable_{L,\kappa})$ creates finite (co)products by applying \cref{lem:faithful(co)limitCreation}. Let $\underline{\sC}_1,\underline{\sC}_2\in \presentable_{\baseCat,L,\kappa}$. The product and coproduct $\underline{\sC}_1\times\underline{\sC}_2$ clearly still satisfies Beck--Chevalley, and the projection maps $\underline{\sC}_1\times\underline{\sC}_2\rightarrow\underline{\sC}_i$ and inclusion maps $\underline{\sC}_i\rightarrow \underline{\sC}_1\times\underline{\sC}_2$ (defined using that $\underline{\sC}_i$ had initial objects) also clearly satisfies Beck--Chevalley. These give condition (1) of the lemma, and to see condition (2), suppose we are given $h_i\colon \underline{\D}\rightarrow \underline{\sC}_i$ and $f_i\colon \underline{\sC}_i\rightarrow \underline{\E}$
all satisfying Beck--Chevalley. Then it is similarly easy to see that the maps $h_1\times h_2\colon \underline{\D}\rightarrow\underline{\sC}_1\times\underline{\sC}_2$ and $f_1\sqcup f_2\colon\underline{\sC}_1\times\underline{\sC}_2\rightarrow\underline{\E}$ also satisfy Beck--Chevalley, whence condition (2) as wanted. 

Next, we show that there is a canonical adjunction datum witnessing that $f_*\dashv f^*$ for every $f\colon W\rightarrow V$ in $\finite_{\baseCat}$. For this, simply observe the natural equivalences
\small\[\myuline{\map}_{\underline{\presentable}_{L,\kappa}}(f_*\underline{\D}, \underline{\sC}) \simeq \myuline{\map}_{\underline{\presentable}_{R,\kappa\text{-filt}}}(\underline{\sC},  f_*\underline{\D}) \simeq \myuline{\map}_{f_*f^*\underline{\presentable}_{R,\kappa\text{-filt}}}(f^*\underline{\sC}, \underline{\D}) \simeq \myuline{\map}_{f_*f^*\underline{\presentable}_{L,\kappa}}(\underline{\D}, f^*\underline{\sC})\] \normalsize where the first and last equivalences is by \cref{prop:antiequivalencePrLPrR}, and the middle equivalence is by \cite[Prop. 6.6.2]{kaifPresentable}. Now write $\eta,\varepsilon$ as the adjunction (co)unit for $f^*\dashv f_*$ in $\underline{\presentable}_{R,\kappa\text{-filt}}$ and $\widetilde{\eta},\widetilde{\varepsilon}$ for the $f_*\dashv f^*$ (co)unit in $\underline{\presentable}_{L,\kappa}$. Tracing through the identifications above, we obtain that $\widetilde{\eta}\dashv \varepsilon$ and $\widetilde{{\varepsilon}}\dashv \eta$. Under these notations, the Beck--Chevalley equivalences $f^*f_*\simeq \id \sqcup \overline{c}_*c^*$ and $f^*f_*\simeq \id \times \overline{c}_*c^*$ are then implemented by
\[\widetilde{\mathrm{BC}}\colon \id\sqcup \overline{c}_*c^* \xlongrightarrow{(\id\sqcup \overline{c}_*c^*)\widetilde{\eta}^{f}}(\id\sqcup \overline{c}_*c^*)f^*f_*\simeq (\id\sqcup\overline{c}_*\overline{c}^*)f^*f_* \xlongrightarrow{\widetilde{\varepsilon}^{\overline{c}}_{f^*f_*}} f^*f_*\]
\[\mathrm{BC}\colon f^*f_*\xlongrightarrow{\eta^{\overline{c}}_{f^*f_*}}(\id\times\overline{c}_*\overline{c}^*)f^*f_*\simeq (\id\times\overline{c}_*{c}^*)f^*f_*\xlongrightarrow{(\id\times \overline{c}_*c^*)\varepsilon^{f}}\id\times \overline{c}_*c^*.\]
Now since $\widetilde{\mathrm{BC}}$ was an equivalence, an inverse is given by the right adjoint, which may in turn be seen easily to be given by $\mathrm{BC}$. Therefore, the composite $\mathrm{BC}\circ\widetilde{\mathrm{BC}}\colon \id\sqcup \overline{c}_*c^* \rightarrow \id\times \overline{c}_*c^*$ \textit{is} the semiadditive equivalence $\sqcup\simeq \times$. In particular, we have the equivalence of maps $(\id\sqcup \emptyset)\simeq (\id\times\emptyset)\circ\mathrm{BC}\circ\widetilde{\mathrm{BC}} \colon \id\sqcup \overline{c}_*c^* \longrightarrow \id$.

Finally, to see that $\underline{\presentable}_{\baseCat, L,\kappa}$ is $\baseCat$--semiadditive via $f_*\dashv f^*$, we need to show that the composite
$f_* \xlongrightarrow{\eta_{f_*}} f_*f^*f_* \xleftarrow[\simeq]{f_*\widetilde{\mathrm{BC}}} f_*(\id\sqcup\overline{c}_*c^*) \xlongrightarrow{f_*(\id\sqcup \emptyset)}f_*$
is an equivalence. By the equivalence $(\id\sqcup \emptyset)\simeq (\id\times\emptyset)\circ\mathrm{BC}\circ\widetilde{\mathrm{BC}} $, this composite is equivalent to the one of the form in \cref{cor:normMapBeingAnEquivalence}, which is an equivalence. This completes the proof that $\underline{\presentable}_{L,\kappa}$ is $\baseCat$--semiadditive. The statement about $\underline{\cat}\Tperfect$ follows  from \cref{nota:CatTPerfectStable}.
\end{proof}

\begin{lem}\label{lem:limitClosureOfPr_R}
    The inclusion $\underline{\presentable}_{\baseCat,R, \mathrm{st},\kappa\text{-filt}}\subset \underline{\widehat{\cat}}$ is closed under arbitrary parametrised limits. In particular, $\underline{\presentable}_{\baseCat,L, \mathrm{st},\kappa}$ is $\baseCat$--cocomplete.
\end{lem}
\begin{proof}
    By the equivalence $\underline{\presentable}_{L, \mathrm{st},\kappa}\simeq \underline{\presentable}_{R, \mathrm{st},\kappa\text{-filt}}\vop$ from \cref{prop:antiequivalencePrLPrR} and the $\baseCat$--semiadditivity from \cref{SemiadditivityOfPresentableStables}, we know that $\underline{\presentable}_{R, \mathrm{st},\kappa\text{-filt}}\subset \underline{\widehat{\cat}}$ is closed under finite indexed products. Hence, since arbitrary parametrised limits can be decomposed into arbitrary fibrewise limit and finite indexed products, we are left to argue in the case of arbitrary fibrewise limits. This can in turn be split up into showing the case of arbitrary products and pullbacks. We will only treat the case of pullbacks since that of arbitrary products is simpler. 

    We would like to apply \cref{lem:faithful(co)limitCreation}. Since fibrewise limits in ${\widehat{\cat}_{\baseCat}}=\func(\baseCat\op,\widehat{\cat})$ are computed pointwise and since $\presentable_{L, \mathrm{st},\kappa},\presentable_{R, \mathrm{st},\kappa\text{-filt}}\subset \widehat{\cat}$ are closed under limits,  $\widehat{\cat}_{\baseCat}$--pullbacks $\underline{\sC}_1\times_{\underline{\sC}_3}\underline{\sC}_2$ of objects in ${\presentable}_{R, \mathrm{st},\kappa\text{-filt}}$ are still in $\func(\baseCat\op,\presentable_{L, \mathrm{st},\kappa})$. Using \cref{prop:limitOfBeckChevalleyAdjunctions},  we can easily check that $\underline{\sC}_1\times_{\underline{\sC}_3}\underline{\sC}_2$ still satisfies the left Beck--Chevalley conditions, and so by \cref{simpsonTheorem}, $\underline{\sC}_1\times_{\underline{\sC}_3}\underline{\sC}_2$ is an object in $\presentable_{\baseCat,R, \mathrm{st},\kappa\text{-filt}}$. Moreover, the projection maps $\underline{\sC}_1\times_{\underline{\sC}_3}\underline{\sC}_2\rightarrow\underline{\sC}_i$ for $i\in\{1,2,3\}$ are also easily seen to preserve parametrised limits and $\kappa$--filtered colimits,
    whence condition (1) of \cref{lem:faithful(co)limitCreation}. Since these projection maps preserves said (co)limits, this means that such (co)limits are computed pointwise in $\underline{\sC}_1\times_{\underline{\sC}_3}\underline{\sC}_2$ and hence  condition (2) of the lemma is easily seen to be satisfied also. The last statement is then an immediate consequence of the first statement and the equivalence $\underline{\presentable}_{L,\kappa}\simeq \underline{\presentable}_{R,\kappa\text{-filt}}\vop$.
\end{proof}

\begin{prop}\label{prop:limitClosureOfCatPerfInCategories}
    The faithful inclusion $\underline{\cat}\Tperfect\subset\underline{\cat}$ is closed under arbitrary parametrised limits.
\end{prop}
\begin{proof}
    Since parametrised limits can be decomposed into fibrewise limits and arbitrary indexed products by \cite[$\S12$]{shahThesis}, by \cref{SemiadditivityOfPresentableStables} we are left to show that the inclusion is closed under fibrewise limits. Since fibrewise limits are computed fibrewise by the dual of \cite[Cor. 5.9]{shahThesis},  we need to show that the  inclusion $\cat\Tperfect_{\baseCat}\subset \cat_{\baseCat}$ is closed under arbitrary limits.

    To do so, we first show that limits of $\baseCat$--perfect stable categories along $\baseCat$--exact functors are again $\baseCat$--perfect stable. Let $\underline{\sC}\colon I\rightarrow \cat\Tperfect_{\baseCat}$ be a diagram. Because limits in $\cat_{\baseCat}=\func(\baseCat\op,\cat)$ are computed pointwise, $\lim_{a\in I}\underline{\sC}_a$ is fibrewise perfect stable. We next show that $\lim_{a\in I}\underline{\sC}_a$ is $\baseCat$--semiadditive. Without loss of generality, assume $\baseCat$ has a final object $T$ and let $f\colon V\rightarrow T$ be a map in $\baseCat$. Since each $\underline{\sC}_a$ is $\baseCat$--semiadditive, the semiadditivity norm transformation $f_!\Rightarrow f_*$ of functors $f_*f^*\underline{\sC}_a \rightarrow \underline{\sC}_a $
    is an equivalence for all $a\in I$. But since fibrewise limits along $\baseCat$--exact functors preserve the indexed (co)product adjunctions ($\baseCat$--exact functors satisfy the Beck--Chevalley condition in \cref{prop:limitOfBeckChevalleyAdjunctions}), zero objects, and ordinary finite biproducts (the zero object and ordinary finite biproducts in $\lim_{a\in I}\underline{\sC}_a$ are given by the ones in each $\underline{\sC}_a$), applying $\lim_{a\in I}$  yields the semiadditivity norm map equivalence $f_!\Rightarrow f_*$ of functors $f_*f^*\lim_{a\in I}\underline{\sC}_a \rightarrow \lim_{a\in I}\underline{\sC}_a$
    whence the $\baseCat$--semiadditivity of $\lim_{a\in I}\underline{\sC}_a$ as claimed. 

    To see that this has the correct universal property in $\cat\Tperfect_{\baseCat}$, we  check the two conditions in \cref{lem:faithful(co)limitCreation}. For condition (2), we need to argue that if we have a transformation of $I$--shaped diagrams $\varphi\colon \constant_I\underline{\D}\rightarrow \underline{\sC}$ in $\cat\Tperfect_{\baseCat}$, then the map $\varphi\colon \underline{\D}\rightarrow\lim_{a\in I}\underline{\sC}_a$ induced by the universal property in $\cat_{\baseCat}$ is already $\baseCat$--exact. Since this statement holds in the unparametrised setting, we are left with showing that the induced functor preserves finite $\baseCat$--biproducts. To see this, letting $f\colon V\rightarrow T$ as in the preceding paragraph, we need to show that the left square in 
    \begin{center}
        \begin{tikzcd}
            f_*f^*\underline{\D}\rar["f_*"]\dar["f_*f^*\varphi"'] & \underline{\D}\dar["\varphi"] && f_*f^*\underline{\D}\rar["f_*"]\dar["f_*f^*\varphi"'] & \underline{\D}\dar["\varphi"] \\
            f_*f^*\lim_{a\in I}\underline{\sC}_a \rar["f_*"] & \lim_{a\in I}\underline{\sC}_a && f_*f^*\underline{\sC}_a \rar["f_*"] & \underline{\sC}_a
        \end{tikzcd}
    \end{center}
    commutes. This is  because for each $a\in I$, the right square commutes by assumption that everything in sight is $\baseCat$--exact, from which we can conclude that the left square also commutes by applying $\lim_{a\in I}$ to the bottom horizontal map in the right square. Finally, for condition (1), by the previous paragraph, we are left to argue that the adjunction counit is $\baseCat$--exact. Since arbitrary limits can be decomposed in terms of arbitrary products and pullbacks and the case of products is simple to see, we will only argue in the case of pullbacks. For this, we need to argue that the projection maps $\underline{\sC}_1\times_{\underline{\sC}_3}\underline{\sC}_2\rightarrow \underline{\sC}_i$ preserve finite indexed products for $i\in\{1,2,3\}$. But this is clear since $\baseCat$--biproducts in $\underline{\sC}_1\times_{\underline{\sC}_3}\underline{\sC}_2$ are created pointwise.
\end{proof}

Next, we mimic the techniques of \cite[{$\S1.1$}]{nineAuthorsIV} to prove:

\begin{prop}\label{compactGenerationOfCatPerf}
The set $\big\{\myuline{\spectra}\tomega, \:\underline{\func}(\Delta^1, \myuline{\spectra}\tomega)\big\}$ consists of $\omega$--compact objects and is jointly  conservative on $\underline{\cat}_{\baseCat}\Tperfect$. Thus, $\underline{\cat}_{\baseCat}\Tperfect$ is $\kappa$-compactly generated for all regular cardinals $\kappa$.
\end{prop}
\begin{proof}
Since compactness and joint conservativity are checked fibrewise, we show that $\big\{\myuline{\spectra}\tomega, \:\underline{\func}(\Delta^1, \myuline{\spectra}\tomega)\big\}$ are $\omega$--compact and jointly conservative on $\cat_{\baseCat_{/V}}\Tperfect$ for an arbitrary $V \in {\baseCat}$. We claim that $\myuline{\spectra}\tomega$ and $\underline{\func}(\Delta^1, \myuline{\spectra}\tomega)$ corepresent the functors  $\underline{\sC} \mapsto \underline{\sC}\Tcore$ and $\underline{\sC} \mapsto \underline{\func}(\Delta^1,\underline{\sC})\Tcore$, respectively. We only show this for the second one since the first is easier:
\begin{equation*}
    \begin{split}
        \myuline{\map}_{\cat_{\baseCat_{/V}}\Tperfect}\big(\underline{\func}(\Delta^1, \myuline{\spectra}\tomega), \underline{\sC}\big) &\simeq \myuline{\map}_{\presentable_{L,\mathrm{st}, \omega}}\big(\underline{\func}(\Delta^1, \myuline{\spectra}), \underline{\ind}_{\omega}\underline{\sC}\big)\\
        &\simeq \underline{\func}^{L, \underline{\omega}}(\underline{\func}(\Delta^1, \myuline{\spectra}), \underline{\ind}_{\omega}\underline{\sC}\big)^{\simeq}\\
        &\simeq \underline{\func}^{R, \omega\operatorname{-filt}}\big(\underline{\ind}_{\omega}\underline{\sC}, \underline{\func}(\Delta^1, \myuline{\spectra}) \big)^{\simeq}\\
        &\simeq \underline{\func}\Big(\Delta^1,\underline{\func}^{R, \omega\operatorname{-filt}}\big(\underline{\ind}_{\omega}\underline{\sC},  \myuline{\spectra} \big)\Big)^{\simeq}\\
        &\simeq \underline{\func}\Big(\Delta^1,\underline{\func}^{L, \underline{\omega}}\big(\myuline{\spectra}, \underline{\ind}_{\omega}\underline{\sC} \big)\Big)^{\simeq}\\
        &\simeq \underline{\func}(\Delta^1,\underline{\sC})^{\simeq}
    \end{split}
\end{equation*}
where the first equivalence is by \cref{TPresentableIdempotentCorrespondence}; the third and fifth are by \cref{HTT5.2.6.2} and \cref{parametrisedAdjointFunctorTheorem}; the fourth by \cref{parametrisedCotensors}; and the last is by \cref{colimitGenerationSpectra}. Since the two corepresented functors preserve $\omega$--filtered colimits, $\myuline{\spectra}\tomega$ and $\underline{\func}(\Delta^1, \myuline{\spectra}\tomega)$ are $\omega$--compact objects. For joint conservativity, suppose $\varphi : \underline{\sC} \rightarrow \underline{\D}$ is a functor such that $\varphi : \underline{\sC}\Tcore \xrightarrow{\simeq} \underline{\D}\Tcore$ and  $\varphi : \underline{\func}(\Delta^1,\underline{\sC})^{\simeq} \xrightarrow{\simeq} \underline{\func}(\Delta^1, \underline{\D})\Tcore$ are equivalences. In particular, the first equivalence implies that $\varphi$ is $\baseCat_{/V}$--essentially surjective. Now, the fibre over $[W \rightarrow V]$ of $\underline{\func}(\Delta^1,\underline{\sC})$ is $\func(\Delta^1, \sC_W)$ and so the second equivalence together with the the formula for unparametrised mapping spaces as pullbacks $\func(\Delta^1, \sC_W)\times_{\sC _W^{\times 2}}\{\ast\}$ gives us that $\varphi : \underline{\sC}\rightarrow \underline{\D}$ is $\baseCat_{/V}$--fully faithful. Therefore, $\varphi $ is an equivalence, and so by \cref{parametrisedMakkaiPitts}, we are done.
\end{proof}

We may now summarise the preceding results in the following package.

\begin{cor}\label{cor:catPerfPresentability}
    The $\baseCat$--category $\underline{\cat}\Tperfect$ is $\baseCat$--semiadditive--presentable. The parametrised limits in it are computed in $\underline{\cat}$ and the parametrised colimits may be computed by the formula $\underline{\colim}_{\underline{J}}\partial \simeq (\underline{\lim}_{\underline{J}\vop}\widetilde{\underline{\ind}\partial})\tomega$, where $\widetilde{\underline{\ind}\partial}\colon \underline{J}\vop\rightarrow \underline{\presentable}_{R,\kappa\text{-filt}}\subset\underline{\widehat{\cat}}_{\baseCat}$ is the diagram obtained by passing to right adjoints.
\end{cor}
\begin{proof}
    $\baseCat$--semiadditivity is by \cref{SemiadditivityOfPresentableStables}. To see $\baseCat$--presentability, by \cref{lem:limitClosureOfPr_R}, we know that $\underline{\cat}\Tperfect\simeq \underline{\presentable}_{L,\mathrm{st},\omega}$ is $\baseCat$--cocomplete. This, together with \cref{compactGenerationOfCatPerf} and \cite[Thm. 6.1.2 (6)]{kaifPresentable}, then implies that $\underline{\cat}\Tperfect$ is $\baseCat$--presentable. That parametrised limits in $\underline{\cat}\Tperfect$ are created in $\underline{\cat}$ was shown in \cref{prop:limitClosureOfCatPerfInCategories}. Finally, the formula for parametrised colimits is an immediate consequence of \cref{lem:limitClosureOfPr_R} and the equivalences $\ind \colon \underline{\cat}\Tperfect \simeq \underline{\presentable}_{L,\mathrm{st},\omega} : (-)\tomega$ and $\underline{\presentable}_{L,\mathrm{st},\omega} \simeq \underline{\presentable}_{R,\mathrm{st},\omega\text{-filt}}\vop$. 
\end{proof}

Our next goal is to articulate the relationship between $\cat\Tperfect_{\baseCat}$ and $\baseCat$--Mackey functors valued in $\cat\perfect$. The basic ingredient will be the following evaluation functors.

\begin{cons}[The evaluation functor]\label{cons:evaluationFunctor}
    Let $V\in\baseCat$. Writing $s\colon \ast\hookrightarrow \baseCat_{/V}$ for the inclusion of the final object, consider the solid part of the diagram
    \begin{equation}\label{eqn:evaluationLiftingSquare}
        \begin{tikzcd}
            \cat\perfect \ar[r, hook] & \cat\\
            \cat\Tperfect_{\baseCat_{/V}} \ar[u,"s^*",dashed]\rar[hook]& \cat_{\baseCat_{/V}}\uar["s^*"]
        \end{tikzcd}
    \end{equation}
    where the hooked arrows are faithful functors. The functor $s^*$ here implements the evaluation at $V\in \baseCat$. Since objects in $\cat\Tperfect_{\baseCat}$ are in particular fibrewise idempotent--complete and stable, and morphisms are in particular fibrewise exact, we obtain the dashed lift as shown. Observe that the top horizontal inclusion and the solid $s^*$ functors preserve all limits. Observe also that, on objects, $s_*\colon\cat\rightarrow \cat_{\baseCat}=\func(\baseCat\op,\cat)$ is concretely given by the functor which sends $\sC\in\cat\perfect\subset\cat$ to the object  $\underline{\sC}\coloneqq s_*\sC\in \func(\baseCat\op,\cat)$ given by $\sC$ at the final object in $\baseCat$ and $\ast$ everywhere else.

    Be warned that while the functor $s_* \colon \cat\rightarrow \cat_{\baseCat}$ also descends to a functor $s_* \colon \cat\perfect\rightarrow \cat\Tperfect_{\baseCat_{/V}}$, it is \textit{not} true that the adjunction $s^*\colon \cat_{\baseCat}\rightleftharpoons \cat : s_*$ descends to an adjunction $s^* \colon \cat\Tperfect_{\baseCat_{/V}}\rightleftharpoons \cat\perfect : s_*$ because the adjunction unit is not a morphism in $\cat\Tperfect_{\baseCat_{/V}}$. However, the functor $s_* \colon \cat\perfect\rightarrow\cat\Tperfect_{\baseCat}$ \textit{does} admit an interesting left adjoint, and we refer the reader to \cite[Thm. A]{PD1} for more details.
\end{cons}

\begin{lem}\label{lem:catPerfEvaluationFunctor}
    For every $V\in\baseCat$, the composite functor $\cat\Tperfect_{\baseCat}\xrightarrow{\res}\cat\Tperfect_{\baseCat_{/V}}\xrightarrow{s^*} \cat\perfect$, where $s^*$ is as constructed above, preserves  arbitrary limits and colimits.
\end{lem}
\begin{proof}
    Since the restriction functor is given by the global section of the $\baseCat$--functor $\underline{\cat}\Tperfect_{\baseCat}\xrightarrow{p^*}\underline{\func}(\underline{V},\underline{\cat}\Tperfect_{\baseCat})$ where $p\colon \underline{V}\rightarrow\underline{\ast}$ is the unique $\baseCat$--functor to the final $\baseCat$--category, it clearly preserves all (co)limits. Hence, we are left to studying $s^*$ for a fixed $V\in\baseCat$. Thus without loss of generality, instead of writing $\baseCat_{/V}$ everywhere, we just assume that $\baseCat$ has a final object.
    
    For limits, by \cref{prop:limitClosureOfCatPerfInCategories} the faithful inclusion $\cat\Tperfect_{\baseCat}\subset \cat_{\baseCat}$ is closed under limits.  Thus, all the solid arrows in \cref{eqn:evaluationLiftingSquare} preserves limits, and so the dashed  $s^*$ also preserves limits.  For colimits,  since $s^*\colon \cat\Tperfect_{\baseCat}\rightarrow \cat\perfect$ preserves arbitrary products and since both sides are semiadditive, $s^*$ also preserves arbitrary coproducts. For pushouts, we use the identifications from \cref{nota:CatTPerfectStable} and just argue that the map $s^*\colon \cat\Tperfect_{\baseCat}\simeq \presentable_{\baseCat,L,\mathrm{st},\omega}\rightarrow \presentable_{L,\mathrm{st},\omega}\simeq \cat\perfect$ preserves pushouts. By  \cref{prop:antiequivalencePrLPrR}, this is equivalent to  $s^*\colon \presentable_{\baseCat,R,\mathrm{st},\omega\mathrm{-filt}}\rightarrow \presentable_{R,\mathrm{st},\omega-\mathrm{filt}}$ preserving pullbacks. And this is true  by \cref{lem:limitClosureOfPr_R}.
\end{proof}

Given these, we are now ready to prove the embedding result:

\begin{thm} \label{mainMackeyInclusion}\label{MackeyInclusionCreatesFibreCofibres}
We have a conservative $\baseCat$--faithful inclusion $\underline{\cat}\Tperfect_{\baseCat} \subset\underline{\cmonoid}_{\baseCat}(\cat\perfect)$. Moreover, this inclusion preserves and reflects parametrised (co)limits.
\end{thm}
\begin{proof}
We first construct the said $\baseCat$--faithful functor. Note first that we have non-full  inclusions $\underline{\cat}\Tperfect_{\baseCat}\subset \underline{{\cat}}_{\baseCat} = \cofree_{\baseCat}(\cat)$ and $\cofree_{\baseCat}(\cat\perfect)\subset \cofree_{\baseCat}(\cat)$ which  preserve finite $\baseCat$--products: the former by \cref{SemiadditivityOfPresentableStables} and the latter since $\cat\perfect \subset {\cat}$ preserves limits. By \cref{nota:CatTPerfectStable}, the first inclusion in fact factors as $\underline{\cat}\Tperfect_{\baseCat}\subset \cofree_{\baseCat}(\cat\perfect)\subset \underline{{\cat}}_{\baseCat}$. Hence, by the preceding points, the inclusion $\underline{\cat}\Tperfect_{\baseCat}\subset \cofree_{\baseCat}(\cat\perfect)$  must also  preserve finite $\baseCat$--products. Now by definition  $\underline{\cmonoid}_{\baseCat}(-) \coloneqq \underline{\func}^{\times}_{\baseCat}(\myuline{\effBurn}(\underline{\finite}), -) \subseteq \underline{\func}_{\baseCat}(\myuline{\effBurn}(\underline{\finite}), -)$ and so applying $\underline{\cmonoid}_{\baseCat}(-)$ and invoking \cite[Cor. 3.4.6]{kaifPresentable} we get a $\baseCat$--faithful inclusion 
$\underline{\cat}\Tperfect_{\baseCat} \subset  \underline{\cmonoid}_{\baseCat}(\cat\perfect)$
where we can dispense with the $\baseCat$--semiadditivisation of the source by virtue of \cref{semiadditivisation} and \cref{SemiadditivityOfPresentableStables}.\

Now to see that it is conservative,  note that we have the commuting triangle of categories
\begin{equation}\label{eqn:evaluationTriangle}
    \begin{tikzcd}
        \cat\Tperfect_{\baseCat} \dar[hook]\ar[drr,"\prod_{V\in\baseCat}\eval_{V}"]\\
        \cmonoid_{\baseCat}(\cat\perfect)\simeq \mackey_{\baseCat}(\cat\perfect) \ar[rr, "\prod_{V\in\baseCat}\eval_{V}"'] && \prod_{V\in\baseCat} \cat\perfect
    \end{tikzcd}
\end{equation}
where the diagonal map is conservative, and so the vertical map must be conservative as well.

For the final statement, first note that by construction, the inclusion $\underline{\cat}\Tperfect_{\baseCat} \subset\underline{\cmonoid}_{\baseCat}(\cat\perfect)$ preserves finite indexed products, and so also finite indexed (co)products by $\baseCat$--semiadditivity of both source and target. Hence, by \cref{lem:conservativePreservationImpliesReflection} the inclusion also reflects these. To deal with the fibrewise (co)limits, since all the restriction functors in sight preserve (co)limits, it suffices by \cite[Cor. 5.9]{shahThesis} to argue that $\cat\Tperfect_{\baseCat}\rightarrow \cmonoid_{\baseCat}(\cat\perfect)\simeq \mackey_{\baseCat}(\cat\perfect)$ preserves and reflects arbitrary limits and colimits. Since (co)limits in $\mackey_{\baseCat}(\cat\perfect)$ are computed pointwise by \cite[Cor. 6.7.1]{barwick1}, the bottom horizontal evaluation map in \cref{eqn:evaluationTriangle} is conservative and preserves (co)limits. On the other hand, the diagonal functor in \cref{eqn:evaluationTriangle} preserves  arbitrary (co)limits  by \cref{lem:catPerfEvaluationFunctor}. Combining these, we get that the vertical functor in \cref{eqn:evaluationTriangle} preserves arbitrary (co)limits in addition to being conservative from the previous paragraph, and hence by \cref{lem:conservativePreservationImpliesReflection} it also reflects (co)limits, as wanted.
\end{proof}

\begin{rmk}
We now give an intuitive description of objects in    $\mackey_{\baseCat}(\cat\perfect)\simeq \cmonoid_{\baseCat}(\cat\perfect)$ which lie in $\cat_{\baseCat}\Tperfect$. Let $\underline{\sC}\in\mackey_{\baseCat}(\cat\perfect)$ and let $f\colon U \rightarrow V$ be a map in $\baseCat$. Write $f_{\#}\colon \sC_U \rightarrow\sC_V$ for the associated transfer map. From the pullback \cref{eqn:complementationSquare} we obtain by the datum of a Mackey functor a decomposition $f^*f_{\#}\simeq \id\oplus \overline{c}_{\#}c^*$. Hence, by inclusion and projection, we obtain the  transformations $u\colon \id \Rightarrow f^*f_{\#}$ and $c\colon f^*f_{\#}\Rightarrow \id$.
By unwinding the definitions in the proof of the theorem,  $\underline{\sC}$ lies in the non--full subcategory $\cat\Tperfect_{\baseCat}$ if $u$ exhibits $f_{\#}\dashv f^*$ and $c$ exhibits $f^*\dashv f_{\#}$. These conditions \textit{can} however be enforced provide we work with the $(\infty,2)$\textit{--category} of spans. As such, we expect that an $(\infty,2)$-categorical version of spans and of Mackey functors should precisely yield ${\cat}\Tperfect_{\baseCat}$.
\end{rmk}

For the purposes of our motivic analyses later, we will also record here the following:

\begin{prop}\label{prop:leftKanExtensionFunctorialityPresentables}
    Fix $\underline{\sC}\in \presentable_{\baseCat,L,\kappa}$. The functor $\underline{\func}(-,\underline{\sC})\colon \underline{\cat}_{\baseCat}\rightarrow\underline{\presentable}_{\baseCat,L,\kappa}$ taken along left Kan extensions preserve finite parametrised colimits.
\end{prop}
\begin{proof}
    Since every finite parametrised colimit can be decomposed as the finite indexed coproducts and finite fibrewise colimits, we will split up the proof into these two cases. For the indexed coproducts, we just note that for a fixed $f\colon U\rightarrow V$ in $\underline{\finite}_{/V}$, we have $\underline{\func}(f_!\underline{I},\underline{\sC})\simeq f_*f^*\underline{\func}(\underline{I},\underline{\sC})$, and so since \cref{SemiadditivityOfPresentableStables} gives that $f_*f^*\underline{\func}(\underline{I},\underline{\sC})$ is also the indexed coproduct in $\underline{\presentable}_{\baseCat,L,\kappa}$, we are done in this case. Since we have argued for an arbitrary $U\in\underline{\finite}_{\baseCat}$, this also covers the case of ordinary finite coproducts by setting $U=\coprod_{i=1}^n V$.

    Hence, we are left with showing the case of pushouts. Suppose we have a pushout diagram
    \begin{center}
        \begin{tikzcd}
            \underline{I}\rar["i"]\dar["k"']\ar[dr,phantom,very near end, "\ulcorner"] & \underline{J}\dar["p"]\\
            \underline{K}\rar["q"'] & \underline{P}.
        \end{tikzcd}
    \end{center}
    By \cref{lem:limitClosureOfPr_R}, we have an equivalence $\underline{\func}(\underline{P},\underline{\sC})\simeq \underline{\func}(\underline{J},\underline{\sC})\times_{\underline{\func}(\underline{I},\underline{\sC})}\underline{\func}(\underline{K},\underline{\sC})$ in $\underline{\presentable}_{\baseCat, R,\kappa\text{-filt}}$ using the restriction maps $p^*,q^*,i^*,k^*$, and so upon passing to left adjoints under $\underline{\presentable}_{\baseCat,L,\kappa}\simeq \underline{\presentable}_{\baseCat, R,\kappa\text{-filt}}\vop$, we obtain the desired result.
\end{proof}

\section{The theory of parametrised cubes}\label{section:theoryOfParametrisedCubes}
In this section, we lay down the theory of parametrised cubes associated to atomic orbital base categories $\baseCat$. As we shall see in \cref{cons:singletonInclusion}, this hypothesis on $\baseCat$ will be exploited in an essential way to identify the ``singletons'' in a parametrised cube, which will in turn play an important role in our treatment of $C_2$--pushouts in the setting of equivariant K--theory in \cref{subsection:PointwiseEqualsNormed2-Groups}. The key insight is that the inclusions of the initial and final objects in $\Delta^1$ will allow us to encode the idea of a ``subset'' in a purely abstract and categorical manner (where $1$ means being in a ``subset'' and $0$ means the converse). The advantage of this point of view is that it not only lets one speak of cubes in a very general setting but also allows many proofs to be carried out using concise adjunction manipulations. As far as we are aware, this use case of atomic orbitality of base categories is new and might represent a third expression -- alongside that of $\baseCat$--semiadditivity and $\baseCat$--symmetric monoidality for which it was first designed by \cite{parametrisedIntroduction} -- of the algebraic richness of the atomic orbitality hypothesis. 

We now outline the contents of this section: in \cref{subsection:basicNotionsCubes}, we  introduce the basic definitions and constructions of cubes; we  then provide a (co)limit decomposition result  \cref{prop:pullbackFormulaForConeDiagrams} and record the interaction of parametrised tensor powers with cofibres as \cref{normsOfCofibres} in \cref{subsection:cubicalDecompositions}; lastly, we specialise the general theory to the equivariant setting where we look at $G/H$--cubes when $|G/H|=2$ in preparation for our K--theoretic applications in \cref{subsection:PointwiseEqualsNormed2-Groups}.

\subsection{Basic notions}\label{subsection:basicNotionsCubes}

\begin{defn}
    Let $w\colon W\rightarrow T$ be a map in $\underline{\finite}_{\baseCat}$ and let $\underline{\sC}$ be a $\baseCat$--category. We write $\underline{\Delta}^1\in \cat_{T}=\func(\baseCat_{/T}\op,\cat)$ for the constant $\baseCat_{/T}$--category with value $\Delta^1$. By \textit{the parametrised $w$--cube} we will mean the $\baseCat_{/T}$--category $w_*w^*\underline{\Delta}^1$ and by  \textit{a parametrised $w$--cube in $\underline{\sC}$} we will mean a $\baseCat$--functor $Q\colon w_*w^*\underline{\Delta}^1\rightarrow \underline{\sC}$. 
\end{defn}

\begin{prop}\label{prop:parametrisedCubesAreFibrewiseCubes}
    The parametrised cubes $w_*w^*\underline{\Delta}^1$ are all parametrised posets, i.e. they belong to the full subcategory $\func(\baseCat\op,\mathrm{Poset})\subseteq \func(\baseCat\op,\cat)$. In fact, these are fibrewise given by cubes of various dimensions.
\end{prop}
\begin{proof}
    This is just because the inclusion $\mathrm{Poset}\subseteq \cat_{\infty}$ preserves all limits (with left adjoint $\tau_{\leq-1}$) and hence, in particular, products. Now, we know that for any $\sC\in\cat_{\infty}$ admitting finite products, $\cofree(\sC)$ also have all indexed products coming from the products on $\sC$. Therefore, $\cofree(\mathrm{Poset})\subseteq \cofree(\cat_{\infty})$ admits (and is closed under) indexed products. Since $\underline{\Delta}^1$ is a parametrised poset, so is $w_*w^*\underline{\Delta}^1$ as required. The last statement is simply because these indexed products are computed as various products in $\mathrm{Poset}$ of $\Delta^1$, which are cubes.
\end{proof}

\begin{cons}\label{cons:parametrisedConing}
    Upon applying $\func(\baseCat\op,-)$, the join--slice adjunction (cf. for instance \cite[Cor. 1.4.17]{markusInfinityBook}) induces the parametrised join--slice adjunction \[(-)\tcone\colon \cat_{\baseCat} \rightleftharpoons \cat_{\baseCat,\underline{\ast}/} : (-)_{p/}\]
where $(-)_{p/}$ is the slice construction on a category equipped with a choice of object, i.e. $(\underline{\D},d)_{/p}\coloneqq \underline{\ast}\times_{\underline{\D}}\underline{\D}^{\Delta^1}$ where the map $\underline{\D}^{{\Delta}^1}\rightarrow \underline{\D}$ is the source map. In particular, if $\underline{\D}$ has an initial object $\ast$, then any functor $ \underline{\sC}\rightarrow\underline{\D}$ will induce an extension $\underline{\sC}\tcone\rightarrow\underline{\D}$.
\end{cons}

We thank Sil Linskens for pointing out the need for the strict initial object assumption in the following result.

\begin{prop}\label{prop:universalConing}
    Let $\underline{\D}$ be a $\baseCat$--category with a strict initial object (i.e. any morphism to the initial object in $\underline{\D}$ must be an equivalence) and $F\colon \underline{\sC}\rightarrow \underline{\D}$ be an arbitrary fully faithful $\baseCat$--functor whose essential image does not contain the initial object of $\underline{\D}$. Then the extension $F\tcone\colon \underline{\sC}\tcone\rightarrow\underline{\D}$ is also fully faithful.
\end{prop}
\begin{proof}
    Since fully faithfulness is a fibrewise notion, we can just prove this fibrewise and hence reduce to proving it in the unparametrised case. Since we already know that the composite $\sC\rightarrow \sC^{\triangleleft}\rightarrow \D$ is fully faithful and since, on the space of objects, it is standard that $(\sC^{\triangleleft})^{\simeq}\simeq \sC^{\simeq}\sqcup\{\emptyset_{\sC}\}$ (see for example  \cite[Prop. A.4]{CMMN2}), we are left to argue that the induced map of mapping spaces is an equivalence when one of the objects is $\emptyset_{\sC}$ the initial object in $\sC^{\triangleleft}$. To this end, first note that $\emptyset_{\sC}$ does not admit any map from any $X\in\sC$: this is because we have a functor $\sC^{\triangleleft}\rightarrow \ast^{\triangleleft}\simeq \Delta^1$, and so any such map would give a map from 1 to 0 in $\Delta^1$, which does not exist. Thus the object $\emptyset_{\sC}\in \underline{\sC}\tcone$ is also a strict initial object. Hence, we are left to argue that $\map_{\sC^{\triangleleft}}(\emptyset_{\sC},X) \rightarrow \map_{\D}(F^{\triangleleft}(\emptyset_{\sC}),F^{\triangleleft}(X))$ is an equivalence. But this is true because by the universal property, we must have that $F^{\triangleleft}(\emptyset_{\sC})\simeq \emptyset_{\D}$, and so both sides are contractible.
\end{proof}

The following is the key construction in the theory of parametrised cubes.

\begin{cons}[Singleton inclusion]\label{cons:singletonInclusion}
    Let $w\colon W\rightarrow T$ be a map in $\underline{\finite}_{\baseCat}$. We would like to construct a map
    \[\psi_w\colon w_!w^*\terminalTCat \longrightarrow w_*w^*\underline{\Delta}^1\] which generalises the inclusion of the singletons in Goodwillie's definition of strong cocartesianness. First note by atomic orbitality that we have the pullback in $\finite_{\baseCat}$
    \begin{equation}\label{eqn:singletonInclusionPullback}
        \begin{tikzcd}
            W\coprod C \ar[dr, phantom , very near start , "\lrcorner"]\rar["\id\coprod c"]\dar["\id\coprod \overline{c}"'] & W\dar["w"]\\
            W\rar["w"'] & T
        \end{tikzcd}
    \end{equation}
    where $C$ is some object in $\finite_\baseCat$. In particular, we have the decomposition
    \[w^*w_*w^*\underline{\Delta}^1\simeq \id_*\id^*w^*\underline{\Delta}^1\times \overline{c}_*c^*w^*\underline{\Delta}^1\simeq w^*\underline{\Delta}^1\times \overline{c}_*c^*w^*\underline{\Delta}^1 \]
    Now by adjunction, to construct $\psi_w$, it would suffice to construct its adjoint $\overline{\psi}_w\colon w^*\terminalTCat \rightarrow w^*w_*w^*\underline{\Delta}^1$. By the decomposition above, this is equivalent to constructing maps 
    \[w^*\terminalTCat \rightarrow w^*\underline{\Delta}^1\quad\quad\quad \big(w^*\terminalTCat\rightarrow \overline{{c}}_*c^*w^*\underline{\Delta}^1\big) \Leftrightarrow\big(\overline{c}^*w^*\terminalTCat\simeq c^*w^*\terminalTCat \rightarrow c^*w^*\underline{\Delta}^1\big)\] To this end, we declare the first map to be the inclusion of the target and the second map to be the inclusion of the source. This yields the map $\psi_w$ which one should think of the inclusion of the singletons in a cube. We will see that this is always fully faithful in \cref{cor:singletonInclusionFullyFaithful}.
    Therefore, by \cref{cons:parametrisedConing}, the map $\psi_w$ constructed above induces a map    \[\varphi_w \colon (w_!w^*\terminalTCat)\tcone\longrightarrow w_*w^*\underline{\Delta}^1\] This map $\varphi_w$ should be thought of generalising the subsets of size at most 1 in a cube, as the following important example will illustrate.
\end{cons}

\begin{example}\label{example:singletonInclusionForOrdinaryCubes}
    In the special case when $T\in\baseCat$ is the final object, $W = \coprod^n_{j=1}T$, and $w\colon W \rightarrow T$ is the fold map, we will see how the above construction yields the usual inclusion of singletons in the $n$--cube $\prod^n_{j=1}\Delta^1=\poset([n])$. To wit, since $T\times T\simeq T$ by finality, we have a decomposition 
    \[(\coprod^n_{j=1}T)\times (\coprod^n_{j=1}T) \simeq \coprod_{\substack{(a,b)\in \mathbb{Z}/n\times\mathbb{Z}/n,\\ a-b\equiv_n0} }T \sqcup \coprod_{j=1}^{n-1}\coprod_{\substack{(a,b)\in \mathbb{Z}/n\times\mathbb{Z}/n,\\ a-b\equiv_nj} }T\]
    where  the term $\coprod_{\substack{(a,b)\in \mathbb{Z}/n\times\mathbb{Z}/n,\\ a-b\equiv_n0} }T$ should be thought of as the diagonal tuples in $(\coprod^nT)\times(\coprod^nT)$.
    Hence, the pullback \cref{eqn:singletonInclusionPullback} in this case looks like
    \begin{center}
        \begin{tikzcd}
            \displaystyle{\coprod_{\substack{(a,b)\in \mathbb{Z}/n\times\mathbb{Z}/n,\\ a-b\equiv_n0} }T \sqcup \coprod_{j=1}^{n-1}\coprod_{\substack{(a,b)\in \mathbb{Z}/n\times\mathbb{Z}/n,\\ a-b\equiv_nj} }T} \rar["\id\sqcup c"] \dar["\id\sqcup \overline{c}"']& \coprod^n_{j=1}T\dar["w"]\\
            \coprod^n_{j=1}T\rar["w"'] &T 
        \end{tikzcd}
    \end{center}
    In this case, the map $\psi\colon \coprod_{j=1}^n\underline{\ast}\simeq w_!\underline{\ast}\rightarrow \prod_{j=1}^n\underline{\Delta}^1\simeq w_*\underline{\Delta}^1$ constructed in \cref{cons:singletonInclusion} comes from specifying the map $w^*\underline{\ast}\rightarrow w^*w_*\underline{\Delta}^1\simeq \underline{\Delta}^1\times\overline{c}_*c^*\underline{\Delta}^1$ describable as \[(\ast,\ldots,\ast) \longrightarrow (\Delta^1\times\prod_{1\leq j\leq n, j\neq 1}\Delta^1, \Delta^1\times\prod_{1\leq j\leq n, j\neq 2}\Delta^1,\ldots, \Delta^1\times\prod_{1\leq j\leq n, j\neq n}\Delta^1)\]
    choosing $1$ in the first copy of $\Delta^1$ and $0$'s in $\prod_{1\leq j\leq n, j\neq k}\Delta^1$ for all $k$. Equivalently, the $k$--th summand in $w_!\underline{\ast}\simeq \coprod_{j=1}^n\ast$ picks out $1$ in the $k$--th copy of $\Delta^1$ in $\prod_{1\leq j\leq n}\Delta^1$ and $0$'s in the other copies of $\Delta^1$. Hence, the $k$--th summand in $w_!\underline{\ast}$ sits  as the $k$--th singleton in $\prod_{j=1}^n\Delta^1\simeq \poset([n])$.
    \end{example}

We now record the following fundamental observation which will serve both as a basic principle for our proofs as well as as an indication that our notion of singletons is ``correct'' from the parametrised point of view, in that it is a notion that is stable under basechange along arbitrary morphisms in $\baseCat$.

\begin{rmk}[Stability of singleton inclusions under basechange]\label{rmk:singletonInclusionBasechange}
    Let $b\colon B \rightarrow T$ be a map in $\baseCat$. Recall first that we have an adjunction $b_! \colon \finite_B \rightleftharpoons \finite_T : b^*$ from \cref{recollect:finiteTSets}   
    where $b_!$ is given by postcomposing with $b\colon B \rightarrow T$ and $b^*$ is pullback along $b$. In particular, applying $b^*$ to the pullback \cref{eqn:singletonInclusionPullback} gives us a pullback
    \begin{equation*}
        \begin{tikzcd}
            b^*W\coprod b^*C \ar[dr, phantom , very near start , "\lrcorner"]\rar["\id\coprod {a}"]\dar["\id\coprod \overline{a}"'] & b^*W\dar["\overline{w}"]\\
            b^*W\rar["\overline{w}"'] & B
        \end{tikzcd}
    \end{equation*}
    We  show now that $b^*\psi_w \simeq \psi_{\overline{w}}$, so that singleton inclusions are stable under basechange.
    
    By pasting pullback squares, we also see that we have the pullback
    \begin{equation}\label{eqn:CWPullback}
        \begin{tikzcd}
            b^*C \ar[dr, phantom , very near start , "\lrcorner"]\rar["\overline{z}"]\dar[" \overline{a}"'] & C\dar["c"]\\
            b^*W\rar["z"'] & W
        \end{tikzcd}
    \end{equation}
    Now recall that the adjoint of $\psi_w$ was a map defined as the composite
    \[w^*\terminalTCat \hookrightarrow w^*\terminalTCat \amalg \overline{c}_!c^*w^*\terminalTCat\xlongrightarrow{w^*\psi_w\simeq\big((1,0)\:(0,0)\big)} w^*\underline{\Delta}^1 \times \overline{c}_*c^*w^*\underline{\Delta}^1\]
    Here, the map $\big((1,0) \: (0,0)\big) $ is matrix notation for the map induced by the four maps
    \[w^*\terminalTCat\xrightarrow{1}w^*\underline{\Delta}^1\quad\quad w^*\terminalTCat\xrightarrow{0}\overline{c}_*c^*w^*\underline{\Delta}^1\quad \quad\overline{c}_!c^*w^*\terminalTCat\xrightarrow{0}w^*\underline{\Delta}^1\quad \quad\overline{c}_!c^*w^*\terminalTCat\xrightarrow{0}\overline{c}_*c^*w^*\underline{\Delta}^1\]
    We claim now that $b^*\psi_w\colon \overline{w}_!\overline{w}\terminalTCat \rightarrow \overline{w}_*\overline{w}^*\underline{\Delta}^1$ is defined similarly, so that $b^*\psi_w\simeq \psi_{\overline{w}}$ as wanted. Adjoining over $b^*\psi_w$ gives the map
    \begin{equation}\label{eqn:adjointingOver}
        \overline{w}^*\terminalTCat \hookrightarrow \overline{w}^*\terminalTCat\amalg \overline{a}_!a^* \terminalTCat \xlongrightarrow{\overline{w}^*b^*\psi_w} \overline{w}^*\underline{\Delta}^1\times\overline{a}_*a^*\underline{\Delta}^1
    \end{equation}
    But then $\overline{w}^*b^* \simeq z^*w^*$, and so together with the pullback \cref{eqn:CWPullback}, the map $\overline{w}^*b^*\psi_w$ from \cref{eqn:adjointingOver} can be analysed as
    $z^*w^*\terminalTCat \amalg\overline{a}_!a^*\terminalTCat\xlongrightarrow{z^*w^*\psi_w\simeq((1,0)\:(0,0))} z^*w^*\underline{\Delta}^1 \times \overline{a}_*a^*\overline{w}^*\underline{\Delta}^1$
    which is the desired form of map.
\end{rmk}

\begin{cor}\label{cor:singletonInclusionFullyFaithful}
    The map $\varphi_w\colon (w_!w^*\underline{\ast})\tcone \rightarrow w_*w^*\underline{\Delta}^1$ is fully faithful.
\end{cor}
\begin{proof}
    By \cref{prop:universalConing}, it suffices to show that the map $\psi_w\colon w_!w^*\underline{\ast}\rightarrow w_*w^*\underline{\Delta}^1$ is fully faithful. Now, since fully faithfulness is a fibrewise notion, it suffices to check the statement in each fibre, and since these maps are stable under basechange by \cref{rmk:singletonInclusionBasechange}, we may without loss of generality argue in the fibre over $T\in\baseCat$. In case $w\colon W \rightarrow T$ is not an equivalence, then we know by \cref{obs:topFixedPointsOfIndexedCoproducts} that $\eval_{T}(w_!\underline{\ast})\simeq \emptyset$, and so the map $\psi_w$ is vacuously fully faithful in the fibre over $T\in\baseCat$. On the other hand, if $w\colon W\rightarrow T$ were an equivalence, then the map $\psi_w \colon w_!\underline{\ast}\rightarrow w_*\underline{\Delta}^1$ is just the inclusion of $0$ in $\Delta^1$, which is also fully faithful. Hence, $\psi$ is fully faithful in either case, as was to be shown. 
\end{proof}

\begin{example}[$G$--cubes and $C_2$--pushouts]\label{example:C_2Pushouts}
    Let $G$ be a finite group, $H\leq G$ a subgroup, and consider the case of $\baseCat=\orbit_G$. The $G$--category $w_*\underline{\Delta}^1= \prod_{G/H}\underline{\Delta}^1$ should then be thought of as the $|G/H|$--dimensional cube equipped with the $G$--action dictated by the one on $G/H$. For example, one would expect such a diagram to have $G$--fixed points only the initial and the final object, and the underlying diagram to be a $|G/H|$--cube. And indeed, we do have that $\eval_{G/G}(w_*w^*\underline{\Delta}^1)\simeq \Delta^1$ and $\eval_{G/e}(w_*w^*\underline{\Delta}^1)\simeq \prod_{|G/H|}\Delta^1$, as expected. In the rest of the article, we will be especially interested in the case when $|G/H|=2$. In this case, a $G$--diagram indexed by $(w_!w^*\underline{\ast})\tcone\subseteq w_*w^*\underline{\Delta}^1$ may be schematically represented as the datum
    \begin{center}
        \begin{tikzcd}
            A \rar["f"] \dar["f"'] & B \\
            B\ar[ur, dash, dashed]
        \end{tikzcd}
    \end{center}
    where $A$ is a $G$--object, $B$ is a $H$--object, the maps $f$ are maps of $H$--objects $\res^G_HA\xrightarrow{f} B$, and the $G$--action on $G/H$ swaps the two copies of $f$'s. We will call the $G$--colimits of such diagrams \textit{$C_2$--pushouts}. In the next section, we will give a formula to compute such $G$--colimits in terms of ordinary pushouts and indexed coproducts. 
\end{example}

\subsection{Cubical decompositions}\label{subsection:cubicalDecompositions}

We would like now to apply the general nonsense \cref{cor:(co)limitDecompositionAlongPushouts} to the cubical setting, the highlights of which are the decomposition  \cref{prop:pullbackFormulaForConeDiagrams,normsOfCofibres}.

\begin{obs}\label{obs:joinsAsAPushout}
Let $\underline{I}\in\cat_{\baseCat}$. Then there is a pushout in $\cat_{\baseCat} = \func(\baseCat\op,\cat)$
\begin{center}
    \begin{tikzcd}
        \underline{I}\rar["c"]\dar["t"']\ar[dr, phantom , very near end, "\ulcorner"] & \terminalTCat\dar["\ell"]\\
        \underline{I}\times\underline{\Delta}^1\rar["\pi"] & \underline{I}\tcocone
    \end{tikzcd}
\end{center}
To see this, since all constructions and operations in sight in the square above are done fibrewise, we may reduce to showing the statement in the unparametrised case in $\cat$, which is standard.
\end{obs}

Using the notations from the pushout in the observation above, we may now extract the following important decomposition result for parametrised (co)limits.

\begin{prop}\label{prop:pullbackFormulaForConeDiagrams}
    Let $i\colon \underline{I}\hookrightarrow \underline{I}\tcocone$ be the inclusion and let the morphism $\canonical \colon i^*\rightarrow c^*\ell^*$ in $\underline{\func}(\underline{I},\underline{\sC})$ be the canonical structure map. The (co)limit functors $\underline{\lim}_{\underline{I}\tcocone}\colon \underline{\func}(\underline{I}\tcocone, \underline{\sC})\rightarrow \underline{\sC}$ and $\underline{\colim}_{\underline{I}\tcone}\colon \underline{\func}(\underline{I}\tcocone, \underline{\sC})\rightarrow \underline{\sC}$ can be computed as follows: for any $\partial\in\underline{\func}(\underline{I}\tcocone,\underline{\sC})$ and $F\in \underline{\func}(\underline{I}\tcone,\underline{\sC})$, we have a pullback and pushout in $\underline{\sC}$
    \begin{center}
    \begin{tikzcd}
        \underline{\lim}_{\underline{I}\tcocone} \partial\rar\dar\ar[dr, phantom , very near start, "\lrcorner"] & \ell^*\partial\dar["\eta"]\\
        c_*i^*\partial\rar["c_*\canonical"] & c_*c^*\ell^*\partial
    \end{tikzcd}
    \hspace{4mm}
    \begin{tikzcd}
        c_!c^*\ell^*F\rar["\varepsilon"]\dar["c_!\canonical"']\ar[dr, phantom , very near end, "\ulcorner"] & \ell^*F\dar\\
        c_!i^*F\rar & \underline{\colim}_{\underline{I}\tcone} F.
    \end{tikzcd}
\end{center}
\end{prop}
\begin{proof}
    We  first  collect a few elementary formal observations: since the inclusion of the source $s\colon \underline{I}\hookrightarrow\underline{I}\times\underline{\Delta}^1$  has a right adjoint $f$ (because $0\in\Delta^1$ is an initial object), we get that $f^*\dashv s^*$ and so $s^*\simeq f_*$.  Hence, if we write $v\colon \underline{I}\times\underline{\Delta}^1\rightarrow\underline{\ast}$ for the unique map, then since $v\simeq cf$, we get $v_*\simeq c_*f_*\simeq c_*s^*$. Moreover, since $ft\simeq \id$, we have $s^*t_*\simeq f_*t_*\simeq \id$.  Next, since $\pi s\simeq i$, we get $s^*\pi^*\simeq i^*$.  Under these identifications, the canonical map in the statement of the result can be obtained by applying $s^*$ to $\eta^t_{\pi^*}\colon \pi^*\rightarrow t_*t^*\pi^*$ since $s^*\pi^*\simeq i^*$ and $s^*t_*t^*\pi^*\simeq \id\circ c^*\ell^*\simeq c^*\ell^*$. It is  straightforward to check  that this is the canonical structure map.
    
    By \cref{obs:joinsAsAPushout}, we have a pullback description of $\underline{\func}(\underline{I}\tcocone,\underline{\sC})$ as in the left square in
    \[
        \begin{tikzcd}
            \underline{\func}(\underline{I}\tcocone, \underline{\sC})\rar["\ell^*"]\dar["\pi^*"']\ar[dr, phantom , very near start, "\lrcorner"] & \underline{\sC}\dar["c^*"]\\
            \underline{\func}(\underline{I}, \underline{\sC})^{\underline{\Delta}^1}\rar["t^*"] & \underline{\func}(\underline{I},\underline{\sC})
        \end{tikzcd}
        \hspace{3mm}
        \begin{tikzcd}
            \underline{\lim}_{\underline{I}\tcocone} \partial\rar\dar\ar[dr, phantom , very near start, "\lrcorner"] & \ell^*\partial\dar["\eta^c_{\ell^*\partial}"]\\
            v_*\pi^*\partial\rar["v_*\eta^t_{\pi^*\partial}"] & c_*t^*\pi^*\partial\simeq c_*c^*\ell^*\partial.
        \end{tikzcd}
    \]
    Hence, by \cref{lem:limitOfAdjoints}, we then get the pullback diagram on the right.     The identifications above now allow us to conclude that this pullback is of the form in the statement of the result. The  case of colimits can  be deduced by passing to opposite categories.
\end{proof}

\begin{cons}\label{cons:subcategoriesOfHornPowers}
    We view the inclusion $\Lambda^2_0\subseteq (\Delta^1)^{\times2}$ as 
    \begin{center}
        \begin{tikzcd}
            00 \rar\dar & 10 \ar[drr, phantom, "\hookrightarrow"]&& 00 \rar \dar& 10\dar \\
            01 & && 01 \rar & 11
        \end{tikzcd}
    \end{center}
    In this way, for any $n\geq 1$, we may describe the objects in the full subcategory $\prod^n_{i=1}\Lambda^2_0\subseteq \prod^n_{i=1}(\Delta^1)^{\times2}$ as those tuples $(a_i)_{1\leq i\leq n}$ where $a_i\in\{00,10,01\}$. We want to decompose $\prod^n_{i=1}\Lambda^2_0$ in terms of subcategories that we will now specify:
    \begin{itemize}
        \item Let $M_n\subseteq \prod^n_{i=1}\Lambda^2_0$ be the full subcategory $(\prod^n_{i=1}\Delta^1)\times\{0\}$. Concretely in terms of the tuples description above, this is the full subcategory consisting of those tuples such that $a_i\in\{00,10\}$.
        \item Let $B_n\subseteq \prod^n_{i=1}\Lambda^2_0$ be the full subcategory $\prod^n_{i=1}\Lambda^2_0\backslash \{(10,\ldots,10)\}$, i.e. of those tuples where some of the $a_i$'s must either be $00$ or $01$.
        \item Let $F_n\subseteq B_n$ be the full subcategory of those tuples where $a_i\in\{01,10\}$, that is, where none of the $a_i$'s are $00$ and at least one of them is $01$.
        \item Let $J_n\coloneqq M_n\cap B_n$.
    \end{itemize}
    Similarly, since for any map $f\colon U\rightarrow V$ in $\baseCat$,  $f_*\underline{\Lambda}^2_0$ is fibrewise just products of multiple copies of $\Lambda^2_0$, we may similarly define in a fibrewise fashion the $\baseCat$--full subcategories $\underline{J}_f, \underline{F}_f, \underline{B}_f,\underline{M}_f\subseteq f_*\underline{\Lambda}^2_0$. Here, $\underline{J}_f,\underline{B}_f, \underline{M}_f$ are clearly $\baseCat$--subcategories, and to see that $\underline{F}_f$ is too, we need to argue that it is closed under restrictions. More precisely, we need to show that if $A\xrightarrow{a} W\xrightarrow{w} V$ are maps in $\baseCat$, then the map $(\underline{F}_f)_W\subseteq (f_*\underline{\Lambda}^2_0)_W\xrightarrow{a^*} (f_*\underline{\Lambda}^2_0)_A$ factors through $(\underline{F}_f)_A$. For this, let us set up some notations and consider the pullback diagram in $\finite_{\baseCat}$
    \begin{center}
        \begin{tikzcd}
            \coprod_{i,j}Y_{ij} \ar[dr, phantom ,very near start, "\lrcorner"]\dar["\sqcup p_{ij}"']\rar["\sqcup q_{ij}"] & \coprod_i Z_i\ar[dr, phantom ,very near start, "\lrcorner"]\dar["\sqcup u_i"'] \rar["\sqcup v_i"] & U\dar["f"]\\
            A \rar["a"] & W \rar["w"] & V
        \end{tikzcd}
    \end{center}
    We then get $(f_*\underline{\Lambda}^2_0)_W\simeq (w^*f_*\underline{\Lambda}^2_0)_W\simeq (\prod_iu_{i*}\underline{\Lambda}^2_0)_W\simeq \prod_i\Lambda^2_0$ and  $(a^*w^*f_*\underline{\Lambda}^2_0)_A\simeq (\prod_{i}\prod_jp_{ij*}\underline{\Lambda}^2_0)_A\simeq \prod_i\prod_j\Lambda^2_0$. Under these identifications, the map $(f_*\underline{\Lambda}^2_0)_W\xrightarrow{a^*} (f_*\underline{\Lambda}^2_0)_A$ is then simply the product  $\prod_i\Delta\colon \prod_i\Lambda^2_0\rightarrow\prod_i\prod_j\Lambda^2_0$ of the diagonal maps. From the tuples definition of $F_n$ above, it is then clear that $(\underline{F}_f)_W\subseteq \prod_i\Lambda^2_0$ is sent to $(\underline{F}_f)_A\subseteq \prod_i\prod_j\Lambda^2_0$ as was to be argued.
\end{cons}

\begin{obs}\label{obs:pushoutDecompositionOfHornPowers}
    The reason we will be interested in the subcategories above is that we have a union of posets $\prod^n_{i=1}\Lambda^2_0 = M_n\cup_{J_n}B_n$ and so we have the left  pushout in 
    \begin{center}
        \begin{tikzcd}
            J_n \rar[hook]\dar[hook]\ar[dr,phantom,very near end, "\ulcorner"] & M_n\dar[hook]\\
            B_n\rar[hook] & \prod_{i=1}^n\Lambda^2_0
        \end{tikzcd}
        \hspace{10mm}
        \begin{tikzcd}
            \underline{J}_f \rar[hook]\dar[hook]\ar[dr,phantom,very near end, "\ulcorner"] & \underline{M}_f\dar[hook]\\
            \underline{B}_f\rar[hook] & f_*\underline{\Lambda}^2_0
        \end{tikzcd}
    \end{center}
    in $\cat$.
    One way to see this is that fully faithful inclusions of posets induce monomorphisms of their associated simplicial sets under the nerve functor, and so such inclusions are in particular cofibrations in the Joyal model structure (cf. \cite[Thm. 2.2.5.1]{lurieHTT}), whence the strict pushout being a homotopy pushout in $\cat$.
    Hence, since pushouts in $\cat_{\baseCat}=\func(\baseCat\op,\cat)$ are computed pointwise, we also have the right pushout above in $\cat_{\baseCat}$.
\end{obs}
\begin{lem}\label{lem:F_nB_nAdjunction}
    The inclusion $F_n\subseteq B_n$ is cofinal.
\end{lem}
\begin{proof}
    To apply Quillen's Theorem A (e.g. \cite[Thm. 4.4.20]{markusInfinityBook}), we need to show that for every tuple $\underline{a}\in B_n$,  the category $(F_n)_{\underline{a}/}$ is weakly contractible. We will show this by showing that $(F_n)_{\underline{a}/}$ has an initial object.
    To this end, consider the tuple $\widetilde{\underline{a}}$ defined as follows:
    \begin{equation*}
        \widetilde{a}_i = \begin{cases}
            01 & \text{ if } a_i\in\{00,01\}\\
            10 & \text{ if } a_i=10
        \end{cases}
    \end{equation*}
    The map $00\rightarrow01$ then supplies us with a map $\eta\colon \underline{a}\rightarrow \widetilde{\underline{a}}$. Since $\underline{a}\in B_n$, some of the $a_i$'s are either $00$ or $01$, and so $\widetilde{\underline{a}}$ indeed lies in $F_n$. Hence, we have  an object $(\eta\colon \underline{a}\rightarrow\widetilde{\underline{a}})$ in $(F_n)_{\underline{a}/}$.

    Now since we are dealing with a poset, to show that $\widetilde{\underline{a}}$ is initial in $(F_n)_{\underline{a}/}$ and since $B_n\subseteq \prod_n\Lambda^2_0$ is fully faithful, it suffices to show that  $\map_{\prod_n\Lambda^2_0}(\widetilde{\underline{a}}, \underline{f})\simeq\ast$ for all $(\underline{a}\rightarrow \underline{f})\in (F_n)_{\underline{a}/}$. For this, consider the map $\map_{\prod_n\Lambda^2_0}\big(\widetilde{\underline{a}},\underline{f}\big)\xlongrightarrow{\eta^*}\map_{\prod_n\Lambda^2_0}\big(\underline{a}, \underline{f}\big)\simeq \ast$
    where the last term is contractible by definition of $(\underline{a}\rightarrow\underline{f})\in(F_n)_{\underline{a}/}$.  Now noting that $\map_{\prod_n\Lambda^2_0}(\underline{x},\underline{z})\simeq \prod_i\map_{\Lambda^2_0}(x_i,z_i)$, the map  is easily seen to be an equivalence since $00\rightarrow 01$ induces an equivalence $\map_{\Lambda^2_0}(01,01)\rightarrow \map_{\Lambda^2_0}(00,01)$.
\end{proof}

\begin{cor}\label{cor:F_fB_fAdjunction}
    The inclusion of $\baseCat$--categories $\underline{F}_f\subseteq \underline{B}_f$ is $\baseCat$--cofinal, and so restriction along this inclusion induces an equivalence $\underline{\colim}_{\underline{B}_f}\simeq \underline{\colim}_{\underline{F}_f}$.
\end{cor}
\begin{proof}
    We may deduce this statement from the unparametrised \cref{lem:F_nB_nAdjunction} immediately since $\baseCat$--cofinality may be checked fibrewise by \cite[Thm. 6.7]{shahThesis}.
\end{proof}

\begin{nota}
    Let $\underline{\sC}\totimes$ be a $\baseCat$--symmetric monoidal category, $f\colon U\rightarrow V$ a map in $\baseCat$, and $(A\rightarrow B) = \varphi \colon \underline{\Delta}^1\rightarrow f^*\underline{\sC}$ a morphism in $f^*\underline{\sC}$. It will be convenient to use the notation $\underline{\colim}_{\underline{J}_f}f_{\otimes}(A\rightarrow B)$ for the $\underline{J}_f$--shaped colimit of the functor $\underline{J}_f\subseteq f_*\underline{\Delta}^1 \xrightarrow{f_*\varphi} f_*f^*\underline{\sC}\xrightarrow{f_{\otimes}} \underline{\sC}$.
\end{nota}

We now come to the main proposition of this subsection:

\begin{prop}\label{normsOfCofibres}
    Let $\underline{\sC}$ be a $\baseCat$--pointed category with all $\baseCat$--colimits equipped with a $\baseCat$--distributive symmetric monoidal structure $\underline{\sC}\totimes$. Let $f\colon U\rightarrow V$ be a map in $\baseCat$. If we have a cofibre sequence $A \rightarrow B\rightarrow C$ in $f^*\underline{\sC}$, then applying $f_{\otimes}$ yields a cofibre sequence in $\underline{\sC}$
    \[\underline{\colim}_{\underline{J}_f}f_{\otimes}(A\rightarrow B)\longrightarrow f_{\otimes}B\longrightarrow f_{\otimes}C.\]
\end{prop}
\begin{proof}
    We write $\partial \colon \underline{\Lambda}^2_0 \rightarrow f^*\underline{\sC}$ for the diagram $(0\leftarrow A \rightarrow B)$ for the diagram defining the cofibre $C$. Since $f_{\otimes}$ was assumed to be distributive, we have a colimit diagram
    \[\delta\tcocone\colon (f_*\underline{\Lambda}^2_0)\tcocone \longrightarrow f_*\big((\underline{\Lambda}^2_0)\tcocone\big)\xlongrightarrow{f_*(\partial\tcocone)} f_*f^*\underline{\sC} \xlongrightarrow{f_{\otimes}} \underline{\sC} \]
    and so we obtain that $f_{\otimes}C$ is the colimit of the diagram $\delta\colon f_*\underline{\Lambda}^2_0\rightarrow \underline{\sC}$. Now applying the colimit decomposition \cref{cor:(co)limitDecompositionAlongPushouts} on the pushout from \cref{obs:pushoutDecompositionOfHornPowers}, we obtain the pushout diagram in $\underline{\sC}$ (where we have suppressed the restriction functors for readability)
    \begin{center}
        \begin{tikzcd}
            \underline{\colim}_{\underline{J}_f}\delta \rar\dar\ar[dr, phantom, very near end, "\ulcorner"] & \underline{\colim}_{\underline{M}_f}\delta\dar\\
            \underline{\colim}_{\underline{B}_f}\delta \rar & f_{\otimes}C.
        \end{tikzcd}
    \end{center}
    But then $\underline{M}_f$ has a final object and the evaluation at the final object of the functor $\underline{M}_f\subseteq f_*\underline{\Lambda}^2_0\xrightarrow{\delta} \underline{\sC}$ is $f_{\otimes}B$, hence we get $\underline{\colim}_{\underline{M}_f}\delta\simeq f_{\otimes}B$. Next, by distributivity, tensoring with the zero object yields the zero object, and so since by definition $\underline{F}_f$ contains at least one $01$ in the tuples description from \cref{cons:subcategoriesOfHornPowers}, we see that the restricted functor $\underline{F}_f\subseteq f_*\underline{\Lambda}^2_0\xrightarrow{\delta}\underline{\sC}$ has constant value the zero object. Therefore, all in all,  by \cref{cor:F_fB_fAdjunction}, we have that $\underline{\colim}_{\underline{B}_f}\delta\simeq \underline{\colim}_{\underline{F}_f}\delta\simeq 0$. Finally, by definition of our notation, $\underline{\colim}_{\underline{J}_f}f_{\otimes}(A\rightarrow B)=\underline{\colim}_{\underline{J}_f}\delta$. Combining all of these gives the claimed cofibre sequence.
\end{proof}

\subsection{Special case of index 2 quotients in the equivariant setting}\label{subsection:specialCaseIndex2Cubes}

In preparation for our application to the algebraic K--theory of 2--groups in the next section, we will be considering the special case of $\baseCat = \orbit_G$ and where all parametrised colimits in sight come from an index 2 subgroup $H\leq G$. As we shall see, the special facts in the following observation conspire to make this situation particularly simple.  

\begin{obs}\label{obs:Index2GroupTheory}
    Let us collect all the elementary group theory we will need here. Let $H\leq G$ be an index 2 subgroup and $K\leq G$ be another subgroup.
    \begin{enumerate}
        \item If the inclusion $H\cap K\leq K$ is proper, then it is also an index 2 inclusion. This is because, fixing any $x\in K\backslash H$, we know by $|G/H|=2$ that for any other $k\in K$, $k=xh$ for some $h\in H$. But then $h = x^{-1}k\in K$ and so $h\in H\cap K$, whence $|K/H\cap K|=2$.  
        \item By the usual double coset decomposition, we know that $G/K\times G/H\cong \coprod_{g\in K\backslash G/H}G/H^g\cap K=\coprod_{g\in K\backslash G/H}G/H\cap K$ where we have used that $H\leq G$ was a normal subgroup since $|G/H|=2$. But then, again by $|G/H|=2$, there are only two possibilities for what the set $K\backslash G/H$ can be, whence
        \begin{equation*}
            G/K\times G/H = \begin{cases}
                (G/K)^{\coprod2} & \text{ if } K\leq H\\
                G/K\cap H & \text{ if } K\not\leq H
            \end{cases}
        \end{equation*}
    \end{enumerate}
\end{obs}

\begin{prop}\label{prop:specialIndex2Identification}
    Let $H\leq G$ have index 2 and let $w\colon G/H\rightarrow G/G$ be the unique map. In this case, the map $\varphi \colon (w_!w^*\underline{\ast})\tcone \hookrightarrow w_*w^*\underline{\Delta}^1$ from \cref{cons:singletonInclusion} induces via \cref{prop:universalConing} an equivalence $\overline{\varphi}\colon ((w_!w^*\underline{\ast})\tcone)\tcocone \rightarrow w_*w^*\underline{\Delta}^1$. Equivalently, in the notation of \cref{cons:subcategoriesOfHornPowers}, the map $\varphi$ induces an equivalence $\varphi \colon (w_!w^*\underline{\ast})\tcone \xrightarrow{\simeq} \underline{J}_w$.
\end{prop}
\begin{proof}
    Equivalences can be checked fibrewise, and so fixing a subgroup $K\leq G$ and writing $r\colon G/K\rightarrow G/G$ for the unique map, we will prove by induction on the order of $G$ that 
    \begin{equation}\label{eqn:inductiveProofMap}
        r^*\overline{\varphi}\colon r^*((w_!w^*\underline{\ast})\tcone)\tcocone\simeq ((r^*w_!w^*\underline{\ast})\tcone)\tcocone \longrightarrow r^*w_*w^*\underline{\Delta}^1
    \end{equation}
    is an equivalence. As the base case, the statement is vacuously true when $|G|=1$. Now, first suppose that $K\lneq G$. By the dichotomy in \cref{obs:Index2GroupTheory} (2), we have the two pullbacks 
    \begin{center}
        \begin{tikzcd}
            G/K\coprod G/K\ar[dr, phantom ,very near start, "\lrcorner"]\rar["i\sqcup gi"]\dar["\id\sqcup \id"'] & G/H\dar["w"] && G/H\cap K\ar[dr, phantom ,very near start, "\lrcorner"]\rar["j"]\dar["\ell"'] & G/H\dar["w"]\\
            G/K\rar["r"'] & G/G && G/K\rar["r"'] & G/G
        \end{tikzcd}
    \end{center}
    according as $K\leq H$ or $K\not\leq H$. In the case $K\leq H$, we see that $r^*\psi\colon r^*w_!w^*\underline{\ast}\rightarrow r^*w_*w^*\underline{\Delta}^1$ is identified with the ordinary singleton inclusion
    \[(i^*\sqcup i^*g^*)w^*\underline{\ast} \simeq r^*\underline{\ast}\sqcup r^*\underline{\ast}\longrightarrow r^*\underline{\Delta}^1\times r^*\underline{\Delta}^1\simeq (i^*\times i^*g^*)w^*\underline{\Delta}\]
    since the map $\psi$ was stable under basechange by \cref{rmk:singletonInclusionBasechange}, and the map $r^*\underline{\ast}\sqcup r^*\underline{\ast}\rightarrow r^*\underline{\Delta}^1\times r^*\underline{\Delta}^1$ is the ordinary singleton inclusion by \cref{example:singletonInclusionForOrdinaryCubes}. Therefore, in this case we indeed have that \cref{eqn:inductiveProofMap} is an equivalence. As for the case $K\not\leq H$, writing $z\colon G/H\cap K\rightarrow G/G$ for the unique map, we have the identification of $r^*\psi$ with
    $\psi\colon \ell_!z^*\underline{\ast}\longrightarrow \ell_*z^*\underline{\Delta}^1$ again by stability of singleton inclusions with basechange \cref{rmk:singletonInclusionBasechange}. But then by \cref{obs:Index2GroupTheory} (1) we know that $H\cap K\leq K$ was an index 2 inclusion and $|K|<|G|$, and so by the inductive hypothesis, \cref{eqn:inductiveProofMap} is also an equivalence, finishing the proof for this case.

    Finally, for the case when $K= G$, we write $i\colon \{G/G\}\hookrightarrow \orbit_G\op$ for the inclusion. We would like to argue that applying $i^*$ (which is the evaluation at $G/G$) to $\overline{\varphi}$ gives an equivalence in $\cat$. Now, we know that $i^*w_*w^*\underline{\Delta}^1\simeq \Delta^1$. On the other hand, $i^*w_!\underline{\ast}\simeq \emptyset$ by \cref{obs:topFixedPointsOfIndexedCoproducts}. Therefore, we see now that $i^*\overline{\varphi}$ is identified with the map $(\emptyset^{\triangleleft})^{\triangleright}\longrightarrow\Delta^1$
    which is clearly an equivalence. This completes the inductive step and hence the proof of the proposition.
\end{proof}

\section{Noncommutative motives and equivariant algebraic K--theory}\label{section:twoVariantsK-theory}
In the final section, we put together all the general theory developed above to treat the parametrised version of algebraic K--theory.  Building upon  \cref{sec4:perfectStables}, we shall  introduce  split Karoubi  sequences in the parametrised context in \cref{subsection:splitKaroubiPushouts} and deduce their properties from the unparametrised context via \cref{mainMackeyInclusion}. Using these sequences, we construct two variants of parametrised noncommutative motives $\underline{\nmot}_{\baseCat}\pointwise$ and $\underline{\nmot}_{\baseCat}^{\mathrm{nm}}$, called \textit{pointwise} and \textit{normed} motives, respectively. The former will be the one with a desirable universal property and which corepresents K--theory, whereas the latter will be the one that may be  refined multiplicatively for formal reasons (cf. \cref{univPropStableMotives,MotivicCorepresentability} resp. \cref{monoidalityOfTMotives}). Next, we shall employ the theory of parametrised cubes laid down in \cref{section:theoryOfParametrisedCubes} to show that these two types of motives coincide in the equivariant case when $G$ is a 2--group, i.e. $|G|=2^n$ for some $n$. Finally, we  combine this with \cref{monoidalBorelificationPrinciple} to treat the case of Swan K--theory.

\subsection{Split Karoubi sequences and additive functors} \label{subsection:splitKaroubiSequences} \label{subsection:splitKaroubiPushouts}
In our setting, the notion of (split) Karoubi sequences, so central in giving a universal characterisation of algebraic K--theory, will simply be a direct adaptation of those of \cite{nineAuthorsII} in light of \cref{mainMackeyInclusion}. 

\begin{defn}\label{definitionKaroubi}
A sequence $\underline{\sC} \xrightarrow{i} \underline{\D} \xrightarrow{p} \underline{\E}$ in $\underline{\cat}\Tperfect_{\baseCat}$ with vanishing composite is called a \textit{Karoubi sequence} if it is both a fibre and cofibre sequence. It is moreover said to be a \textit{split} Karoubi sequence if it can be completed to ${\baseCat}$-adjunctions 
 \begin{equation}\label{eqn:splitKaroubiSequence}
    \begin{tikzcd}
    {\underline{\sC}}  \rar["i" description, hook] & {\underline{\D}} \rar["p" description, two heads] \lar[shift right = 3, "q"', two heads]\lar[shift left = 3, "r", two heads]& {\underline{\E}} \lar[shift right = 3, "\ell"', hook]\lar[shift left = 3, "j", hook]
    \end{tikzcd}
    \end{equation}
where an arrow stacked above another denotes being a left adjoint. If we only have the left adjoints $q$ and $\ell$ (resp. only right adjoints $r$ and $j$), then we say that the Karoubi sequence is left--split (resp. right--split).
\end{defn}

\begin{rmk}\label{(co)fibreCreatedPointwiseInMackey}
Since $\cat\perfect$ is semiadditive, (co)limits in $\mackey_{\baseCat}(\cat\perfect)$ are computed pointwise  by \cite[Cor. 6.7.1]{barwick1}. On the other hand, \cite[$\S A.1, A.2$]{nineAuthorsII} give us very good control of the fibre and cofibre sequences in $\cat\perfect$ in terms of (split) Karoubi sequences. Hence, in conjunction with  \cref{MackeyInclusionCreatesFibreCofibres}, we will have a good control of the parametrised (split) Karoubi sequences as defined above. 
\end{rmk}

\begin{prop}\label{workhorseAdjunctionYoga}
Suppose we have a sequence $\underline{\sC} \xrightarrow{i} \underline{\D} \xrightarrow{p} \underline{\E}$ in $\underline{\cat}\Tperfect_{\baseCat}$ with vanishing composite. Then the following conditions are equivalent for the given sequence: 
\begin{enumerate}
    \item it is a fibre sequence, and $p$ admits a fully faithful left (resp. right) adjoint $\ell$,
    \item it is a cofibre sequence, and $i$ is fully faithful and admits a left (resp. right) adjoint $q$. 
\end{enumerate}
Furthermore, if (1) and (2) hold, then both the original sequence and the left (resp. right) sequence $\underline{\E}\xrightarrow{\ell}\underline{\D}\xrightarrow{q}\underline{\sC}$ are Karoubi sequences.
\end{prop}
\begin{proof}
Since the inclusion $\underline{\cat}\Tperfect_{\baseCat} \subset \underline{\cmonoid}_{\baseCat}(\cat\perfect)$ preserves and reflects fibres and cofibres by \cref{MackeyInclusionCreatesFibreCofibres}, and since these are pointwise in $\mackey_{\baseCat}(\cat\perfect)$ by the remark above, we can check the Karoubiness of these sequences by checking fibrewise. This immediately reduces the result to \cite[Lem. A.2.5]{nineAuthorsII}.
\end{proof}

\begin{cor}\label{cor:KaroubiFibreDescription}
    Suppose we have a cofibre sequence  $\underline{\sC}\xrightarrow{i}\underline{\D}\rightarrow\underline{\E}$ in $\underline{\cat}\Tperfect_{\baseCat}$ such that $i$ was fully faithful. Then it is already a Karoubi sequence.
\end{cor}
\begin{proof}
    Again, by appealing to \cref{mainMackeyInclusion}, we may prove this statement in $\mackey_{\baseCat}(\cat\perfect)$, where it is known by \cite[Cor. A.1.10]{nineAuthorsII}.
\end{proof}


\begin{lem}\label{indexedProductsPreserveAdjunctions} Let $f \colon W \rightarrow V$ be in $\baseCat$. Then $f_*$ sends a split Karoubi sequence 
in $\big(\underline{\cat}\Tperfect_{\baseCat}\big)_W$ to one in $\big(\underline{\cat}\Tperfect_{\baseCat}\big)_V$.
\end{lem}
\begin{proof}
We saw in \cref{SemiadditivityOfPresentableStables} that $\underline{\cat}_{\baseCat}\Tperfect$ is ${\baseCat}$-semiadditive, and so $f_! \simeq f_*$. Hence $f_*$ preserves (co)fibre sequences and we have bifibre sequences in the three directions above. Furthermore, \cite[Lem. 4.3.2]{kaifPresentable} says that the desired three layers of sequences are all adjoints of each other, and hence they form a split Karoubi sequence.
\end{proof}

\begin{fact}[Split Karoubi classification]\label{fact:splitKaroubiClassification}
    Suppose we have a split Karoubi sequence as in \cref{eqn:splitKaroubiSequence}. This can then be recovered as the pullback 
    \begin{center}
        \begin{tikzcd}
            \underline{\D} \rar["q\rightarrow qjp"]\dar["p"']\ar[dr, phantom, very near start, "\lrcorner"] & \underline{\sC}^{\Delta^1}\dar["\target"]\\
            \underline{\E} \rar["qj"] & \underline{\sC}
        \end{tikzcd}
    \end{center}
    whose vertical fibres are then $\underline{\sC}$. This can be deduced from the analogous nonparametrised result in $\cat\perfect$, recorded for instance in \cite[Prop. A.2.12]{nineAuthorsII}, since the inclusion $\cat\Tperfect_{\baseCat}\subset \mackey_{\baseCat}(\cat\perfect)$ preserves and reflects (co)limits by \cref{mainMackeyInclusion}, and since (co)limits in $\mackey_{\baseCat}(\cat\perfect)$ are computed pointwise by \cref{(co)fibreCreatedPointwiseInMackey}.
\end{fact}

We learnt of the following observation from the forthcoming  \cite[Lem. 1.1.4]{nineAuthorsIV}, which was written in the more structured setting of Poincar\'{e} categories. We include the argument for $\cat\perfect$ here for the reader's convenience.

\begin{lem}\label{lem:CDHAdjunctionViaClassification}
    There is an adjunction $L : \func(\Delta^1, \cat\perfect) \rightleftarrows \cat\perfect : R$ where $L(\sC \xrightarrow{f} \D) \simeq \sC\times_{\D}\D^{\Delta^1}$ and $R({\E}) \simeq (\E^{\Delta^1}\rightarrow  {\E})$, where both $\D^{\Delta^1}\rightarrow \D$ and $\E^{\Delta^1}\rightarrow\E$ are the target maps. Moreover, the right adjoint $R$ preserves all colimits.
\end{lem}
\begin{proof}
    To see that we have such an adjunction, we would like to construct an equivalence 
    \begin{equation}\label{eqn:CDHAdjunction}
        \map_{\func(\Delta^1,\cat\perfect)}(\sC\xrightarrow{f}\D, \E^{\Delta^1}\xrightarrow{\target}\E) \simeq \map_{\cat\perfect}(\sC\times_{\D}\D^{\Delta^1},\E)
    \end{equation}
    natural in all the variables. For this, consider the split Karoubi sequence
    $\D \hookrightarrow \sC\times_{\D}\D^{\Delta^1} \twoheadrightarrow \sC$ classified by $\sC\xrightarrow{f}\D$ via \cref{fact:splitKaroubiClassification}. Since  $\func\exact(-,\E)$ preserves split Karoubi sequences, we obtain another split Karoubi sequence
    $\func\exact(\sC,\E) \hookrightarrow \func\exact(\sC\times_{\D}\D^{\Delta^1},\E) \twoheadrightarrow \func\exact(\D,\E)$.
    By the classification \cref{fact:splitKaroubiClassification} yet again, we then obtain an equivalence
    $\func\exact(\sC\times_{\D}\D^{\Delta^1},\E)\simeq \func\exact(\D,\E)\times_{\func\exact(\sC,\E)}\func\exact(\sC,\E)^{\Delta^1}$.
    Applying core groupoids $(-)^{\simeq}$ to this yields the equivalence \cref{eqn:CDHAdjunction}, which is clearly natural in all the inputs, as wanted. Finally, to see that the right adjoint preserves colimits, we just need to argue that the functor $(-)^{\Delta^1}$ commutes with colimits. Using that $\cat\perfect\simeq \presentable_{L,\mathrm{st},\omega}\simeq \presentable_{R,\mathrm{st},\omega\text{-filt}}\op$ and that the faithful inclusion $\presentable_{R,\mathrm{st},\omega\text{-filt}}\subset \widehat{\cat}$ creates limits, we obtain
    \[\ind(\colim_{j\in J}\E_j)^{\Delta^1}\simeq (\colim^{\presentable_L}_{j\in J}\ind\E_j)^{\Delta^1}\simeq (\lim^{\presentable_R}_{j\in J\op}\ind\E_j)^{\Delta^1}\simeq \lim^{\presentable_R}_{j\in J\op}(\ind\E_j)^{\Delta^1}\simeq \colim_{j\in J}^{\presentable_L}\ind(\E_j)^{\Delta^1}\]
    where the outer equivalences are by \cite[Prop. 5.3.5.15]{lurieHTT}, whence the desired conclusion.
\end{proof}

\begin{cons}\label{pullbackOfCompactCategories}
Since $\func(\Delta^1, \cat\perfect)$ and $\cat\perfect$ are semiadditive, the left adjoint $L$ from \cref{lem:CDHAdjunctionViaClassification} preserves finite products. Hence we can apply $\underline{\cmonoid}_{\baseCat}$ to obtain a ${\baseCat}$-adjunction $\underline{L}_{\baseCat} : \underline{\func}(\Delta^1, \underline{\cmonoid}_{\baseCat}(\cat\perfect)) \rightleftarrows \underline{\cmonoid}_{\baseCat}(\cat\perfect) : \underline{R}_{\baseCat}$ where the ${\baseCat}$-right adjoint preserves all fibrewise colimits. Furthermore, both adjoints clearly restrict to $\underline{\cat}\Tperfect_{\baseCat}\subset \underline{\cmonoid}_{\baseCat}(\cat\perfect)$ so that we get the $\baseCat$--adjunction
$\underline{L}_{\baseCat} : \underline{\func}(\Delta^1, \underline{\cat}\Tperfect_{\baseCat}) \rightleftarrows \underline{\cat}\Tperfect_{\baseCat} : \underline{R}_{\baseCat}$.
By the preservation and reflection of fibrewise (co)limits from \cref{mainMackeyInclusion}, the right adjoint here also preserves all fibrewise colimits, and so in particular, $\kappa$--filtered colimits. Thus, $\underline{L}_{\baseCat}$ preserves $\kappa$-compact objects for all regular cardinals $\kappa$. Thus, if $(\underline{\sC} \xrightarrow{f} \underline{\D})$ is a ${\baseCat}$-exact functor between $\kappa$-compact ${\baseCat}$-perfect stable categories, then  $\underline{\sC}{\times}_{\underline{\D}}\underline{\D}^{\Delta^1}$ is $\kappa$-compact too. 
\end{cons}

\begin{cor}\label{splitKaroubiSequencesAsFilteredColimits}
For any regular cardinal $\kappa$ there is a small set $S_{\kappa}$ of split Karoubi sequences on $\kappa$-compact ${\baseCat}$-perfect-stable categories such that any split Karoubi sequence in $\underline{\cat}_{\baseCat}\Tperfect$ can be written as a fibrewise $\kappa$-filtered colimit of sequences in $S_{\kappa}$.
\end{cor}
\begin{proof}
    First of all, for any regular cardinal $\kappa$, since $\cat\Tperfect_{\baseCat}$ is $\kappa$--compactly generated by \cref{compactGenerationOfCatPerf}, we have by \cite[Prop. 5.3.5.15]{lurieHTT} that $\func(\Delta^1, \cat\Tperfect_{\baseCat})\simeq \ind_{\kappa}\func(\Delta^1,(\cat\Tperfect_{\baseCat})^{\kappa})$.  On the other hand, by $\kappa$--compact generation, we also have that finite limits in $\cat\Tperfect_{\baseCat}$ commute with $\kappa$--filtered colimits: this is because equivalences can be checked by mapping out of a set of $\kappa$--compact objects by virtue of $\kappa$--compact generation, and mapping out of these commutes with all $\kappa$--filtered colimits and all limits;  hence, we can reduce this commutation statement to the category of spaces where it is true. 

    Combining the   previous paragraph with the fact that the functor  $(\underline{\E}\xrightarrow{f}\underline{\sC}) \mapsto \underline{\E}\times_{\underline{\sC}}\underline{\sC}^{\Delta^1}$ preserves $\kappa$--compactness by \cref{pullbackOfCompactCategories}, we obtain that if we are given a split Karoubi sequence in $\cat\Tperfect_{\baseCat}$ as in \cref{definitionKaroubi}, then we may write $\underline{\E}\xrightarrow{qj}\underline{\sC}$ as a $\kappa$--filtered colimit $\colim_{a\in I}(\underline{\E}_{a}\xrightarrow{(qj)_a} \underline{\sC}_a)$ where $\underline{\sC}_a, \underline{\E}_a\in(\cat\Tperfect_{\baseCat})^{\kappa}$. Thus, the map of pullback squares 
    \begin{center}
        \begin{tikzcd}
            \underline{\D}_a \rar\dar["p_a"']\ar[dr, phantom, very near start, "\lrcorner"] & \underline{\sC}_a^{\Delta^1}\dar["\target"]
            \ar[drr, phantom, "\longrightarrow"]& & \underline{\D} \rar["q\rightarrow qjp"]\dar["p"']\ar[dr, phantom, very near start, "\lrcorner"] & \underline{\sC}^{\Delta^1}\dar["\target"]\\
            \underline{\E}_a \rar["(qj)_a"] & \underline{\sC}_a && \underline{\E} \rar["qj"] & \underline{\sC}
        \end{tikzcd}
    \end{center}
    is an equivalence upon applying $\colim_{a\in I}$, where $\underline{\D}_a$ is also in $(\cat\Tperfect_{\baseCat})^{\kappa}$. In this way, we have written the given split Karoubi sequence as a $\kappa$--filtered colimit of split Karoubi sequences of $\kappa$--compact objects $\underline{\sC}_a\hookrightarrow \underline{\D}_a\twoheadrightarrow \underline{\E}_a$, as wanted.
\end{proof}

Using these, we can now define the additivity of a functor as follows:

\begin{defn}\label{defn:additiveFunctors}
Let $\underline{\A}$ be a ${\baseCat}$--stable category. A ${\baseCat}$-functor $\underline{\cat}\Tperfect_{\baseCat} \rightarrow \underline{\A}$ is said to be \textit{additive} if it sends split Karoubi sequences to fibre sequences and preserves the final objects. We write $\underline{\func}\add_{\baseCat}(\underline{\cat}\Tperfect_{\baseCat}, \underline{\A}) \subseteq \underline{\func}_{\baseCat}(\underline{\cat}\Tperfect_{\baseCat}, \underline{\A})$ for the ${\baseCat}$-full subcategory of such. 
\end{defn}

\begin{obs}[Waldhausen sequences]\label{obs:waldhausenSplitting}
    Write $s\colon \Delta^0\rightarrow \Delta^1$ and $t\colon \Delta^0\rightarrow \Delta^1$ for the inclusion of the source and the target, respectively. One of Waldhausen's many key original insights  is the following: for $\underline{\sC}\in\cat\Tperfect_{\baseCat}$, we have the split Karoubi sequence
    \begin{equation}\label{eqn:waldhausenSequence}
    \begin{tikzcd}
        \underline{\sC}\rar[hook, "s_*" description]& \underline{\sC}^{\Delta^1} \rar[two heads, "t^*" description] \lar[shift right = 3, "s^*"', two heads]  \lar[shift left = 3,two heads, "\fib"] & \underline{\sC}. \lar[shift right = 3, "t_!"', hook]  \lar[shift left = 3, hook, "t_*"]
    \end{tikzcd}
\end{equation}
    This means that for any additive $F \colon \underline{\cat}\Tperfect\rightarrow \myuline{\A}$, we have the split fibre sequence 
    \begin{equation*}
        \begin{tikzcd}
            F(\underline{\sC}) \rar["F(s_*)"] & F(\underline{\sC}^{\Delta^1})\rar["F(t^*)"]\lar[bend left = 40, "F(\fib)"] & F(\underline{\sC})\lar[bend left = 40, "F(t_*)"]
        \end{tikzcd}
    \end{equation*}
    which yields the equivalence  $F(\fib)\times F(t^*) \colon F(\underline{\sC}^{\Delta^1}) \simeq F(\underline{\sC})\times F(\underline{\sC}) : F(s_*)+F(t_*)$.
    This is of foundational importance as we will see in the next basic observation which may be seen as ``trickling down'' the additivity property through a level of decategorification. As in the nonparametrised situation, it is a straightforward matter to see using \cref{TTensorsOfPresentables} that all such sequences are obtained by applying $\underline{\sC}\otimes-$ to the version of \cref{eqn:waldhausenSequence} for $\myuline{\spectra}\tomega$.
\end{obs}

\begin{lem}[Waldhausen's trick, {\cite[Rmk. 2.7.6 (ii)]{nineAuthorsII}}]\label{rmk:waldhausenTrick}
    Let $F\colon \underline{\cat}\Tperfect\rightarrow\underline{\A}$ be an additive invariant. If we have a cofibre sequence $\alpha \Rightarrow \beta \Rightarrow \gamma$ of maps $\underline{\sC} \rightarrow \underline{\D}$ in $\cat\Tperfect_{\baseCat}$, then we have an equivalence of morphisms $F\beta \simeq F\alpha\oplus F\gamma \colon F\underline{\sC}\rightarrow F\underline{\D}$. 
\end{lem}
\begin{proof}
    The key is that both natural transformations $(\beta\Rightarrow \gamma), (\alpha\oplus \gamma \Rightarrow \gamma) \colon \underline{\sC} \longrightarrow \underline{\D}^{\Delta^1}$    have the same fibres, i.e. $\alpha$. Hence, applying $F$ and postcomposing further with the equivalence $F(\fib)\times F(t^*)\colon F(\underline{\D}^{\Delta^1})\xrightarrow{\simeq} F(\underline{\D})\times F(\underline{\D})$ from \cref{obs:waldhausenSplitting} yields that the two morphisms $F(\beta\Rightarrow \gamma), \: F(\alpha\oplus\gamma\Rightarrow \gamma) \colon F(\underline{\sC})\rightarrow F(\underline{\D}^{\Delta^1})$ are equivalent. Finally, postcomposing now these two equivalent morphisms with $F(s^*)\colon F(\underline{\D}^{\Delta^1}) \rightarrow F(\underline{\D})$ shows that we have an equivalence of morphisms
    $F(\beta) \simeq F(\alpha\oplus \gamma) \simeq F(\alpha)\oplus F(\gamma) \colon F(\underline{\sC}) \rightarrow F(\underline{\D})$
    as was to be shown.
\end{proof}

\begin{rmk}\label{rmk:waldhausenRightSplit}
    In fact, the proof of the second part of the lemma above shows that any additive functor already sends left--split or right--split Karoubi sequences to fibre sequences.
\end{rmk}

Next, we record how stable valued additive functors interact with certain pushout diagrams. This will be needed in our analysis of equivariant algebraic K--theory in \cref{subsection:PointwiseEqualsNormed2-Groups}.

\begin{defn}\label{splitKaroubiPushout}
A square in $\underline{\cat}\Tperfect$ 
\begin{center}
    \begin{tikzcd}
    \underline{\A} \rar \dar[hook] \ar[dr, phantom, very near end, "\scalebox{1.5}{$\ulcorner$}"]& \underline{\B} \dar[hook]\\
    \underline{\sC} \rar\uar[bend right = 30, dashed] & \underline{\proper}\uar[bend right = 30, dashed]
    \end{tikzcd}
\end{center}
is said to be a \textit{right-split Karoubi pushout} if it is a pushout diagram and the vertical arrows are right-split Karoubi inclusions.
\end{defn}

The following lemma gives the source of right-split Karoubi pushouts that will concern us.

\begin{lem}\label{sourceOfRightSplitKaroubiPushouts}
Suppose we have a span $\underline{\sC}\hookleftarrow \underline{\A}\rightarrow\underline{\B}$ in $\underline{\cat}\Tperfect$ where the left morphism is fully faithful, and both morphisms have right adjoints. Then the pushout in $\underline{\cat}\Tperfect$  is a right-split Karoubi pushout.
\end{lem}
\begin{proof}
We work in the presentable setting via the equivalence $\underline{\cat}\Tperfect \simeq \underline{\presentable}_{L, \mathrm{st}, \omega}$ from \cref{TPresentableIdempotentCorrespondence}. Recall that colimits in $\underline{\presentable}_{L, \mathrm{st},\omega}$ are computed as limits in $\underline{\presentable}_{R, \mathrm{st},\omega\text{-filt}}$ by \cref{prop:antiequivalencePrLPrR}. And so we get the solid pushout in $\underline{\presentable}_{L, \mathrm{st}, \omega}$ and dashed pullback  in $\underline{\presentable}_{R, \mathrm{st},\omega\text{-filt}}$
\begin{center}
    \begin{tikzcd}
    \underline{\ind}(\underline{\A}) \rar \dar[hook] \ar[dr, phantom, very near end, "\scalebox{1.5}{$\ulcorner$}"]& \underline{\ind}(\underline{\B}) \dar\lar[bend left = 30, dashed]\\
    \underline{\ind}(\underline{\sC}) \rar\uar[bend right = 30, dashed] & \underline{\ind}(\underline{\proper}). \uar[bend right = 30, dashed]\lar[bend left = 30, dashed]
    \end{tikzcd}
\end{center}
Since limits in both $\underline{\presentable}_{L,\mathrm{st},\omega}$ and $\underline{\presentable}_{R,\mathrm{st},\omega\text{-filt}}$ are computed underlying by \cref{lem:limitClosureOfPr_R,prop:limitClosureOfCatPerfInCategories}, and since the top and left dashed maps are themselves compact--preserving left adjoints by our hypothesis, the bottom and right dashed maps are also compact--preserving left adjoints (i.e. the dashed square is also a pullback in $\underline{\presentable}_{L,\mathrm{st},\omega}$).  Moreover, since sections pull back to sections and since $\underline{\ind}(\underline{\A}) \rightarrow \underline{\ind}(\underline{\sC})$ is a section of $\underline{\ind}(\underline{\sC}) \rightarrow \underline{\ind}(\underline{\A})$, we see that $\underline{\ind}(\underline{\B}) \rightarrow \underline{\ind}(\underline{\proper})$ is a section of $\underline{\ind}(\underline{\proper})\rightarrow \underline{\ind}(\underline{\B})$. This pair being adjoint to each other then  implies that $\underline{\ind}(\underline{\B})\rightarrow \underline{\ind}(\underline{\proper})$ is fully faithful. Since all maps in sight preserve compacts, we may apply $(-)\tomega$ to the solid diagram to get a right--split Karoubi pushout.
\end{proof}

The following result is where our stability hypothesis comes in.

\begin{lem}\label{motivicSplitKaroubiPushout}
Let $F\colon \underline{\cat}\Tperfect\rightarrow \underline{\M}$ be an additive $\baseCat$--functor where $\underline{\M}$ is $\baseCat$--stable. If we have a right-split Karoubi pushout as in \cref{splitKaroubiPushout}, then
\begin{center}
    \begin{tikzcd}
    F(\underline{\A}) \rar \dar & F(\underline{\B}) \dar\\
    F(\underline{\sC}) \rar & F(\underline{\proper}) 
    \end{tikzcd}
\end{center}
is a pushout in $\underline{\M}$.
\end{lem}
\begin{proof}

First we extend the diagram with $\underline{\E} := \cofib(\underline{\B} \hookrightarrow \underline{\proper})$ to obtain
\begin{center}
    \begin{tikzcd}
    \underline{\A} \rar \dar[hook] \ar[dr, phantom, very near end, "\scalebox{1.5}{$\ulcorner$}"]& \underline{\B} \dar[hook]\ar[dr, phantom, very near end, "\scalebox{1.5}{$\ulcorner$}"] \rar & 0 \dar\\
    \underline{\sC} \rar\uar[bend right = 30, dashed] & \underline{\proper}\uar[bend right = 30, dashed]\rar& \underline{\E}.
    \end{tikzcd}
\end{center}
Since taking cofibres of right-split Karoubi inclusions give right-split Karoubi sequences by \cref{workhorseAdjunctionYoga}, we get right-split Karoubi sequences $\underline{\A}\hookrightarrow \underline{\sC}\twoheadrightarrow \underline{\E}$ and $\underline{\B}\hookrightarrow \underline{\proper}\twoheadrightarrow \underline{\E}$.
Hence, by  \cref{rmk:waldhausenRightSplit}, the maps $F(\underline{\sC})/F(\underline{\A}) \rightarrow F(\underline{\E})$ and $F(\underline{\proper})/F(\underline{\B}) \rightarrow F(\underline{\E})$ are equivalences. Now consider the horizontal maps of vertical cofibre sequences
\begin{center}
    \begin{tikzcd}
    F(\underline{\A}) \rar \dar & F(\underline{\B}) \dar \rar & 0 \dar\\
    F(\underline{\sC}) \dar\rar& F(\underline{\proper}) \dar\rar& F(\underline{\E})\dar[equal]\\
    F(\underline{\sC})/F(\underline{\A}) \rar[dashed]\ar[rr, bend right = 30, "\simeq"] & F(\underline{\proper})/F(\underline{\B}) \rar["\simeq"] & F(\underline{\E}) 
    \end{tikzcd}
\end{center}
where the equivalences are by the previous sentence. Hence, the dashed map is an equivalence too. On the other hand, we have this map of cofibre sequences
\begin{center}
    \begin{tikzcd}
    F(\underline{\B}) \rar\dar[equal] & F(\underline{\B})\coprod_{F(\underline{\A})}F(\underline{\sC}) \rar\dar & F(\underline{\sC})/F(\underline{\A})\dar["\simeq"]\\
    F(\underline{\B}) \rar & F(\underline{\proper}) \rar & F(\underline{\proper})/F(\underline{\B})
    \end{tikzcd}
\end{center}
Since we are working stably, this implies that the middle vertical is an equivalence.
\end{proof}

\subsection{Two variants of noncommutative motives}\label{subsection:variant2}\label{subsection:noncommutativeMotives}
We now follow closely the methods of \cite{nineAuthorsIV} in constructing the noncommutative motives in our setting. Essentially all the proofs of this section are direct parametrised modifications of their arguments and we are grateful to them for sharing a draft of their upcoming work.

\begin{hypothesis}\label{hypothesis2:norm_finiteness}
    Note that, under \cref{hypothesis:normed_spectra} on $\baseCat$, the $\baseCat$--category $\myuline{\presentable}_{L,\mathrm{\baseCat}\mathrm{-st}}$  attains a $\baseCat$--symmetric monoidal structure. We say that $\baseCat$ satisfies the \textit{norm finiteness hypothesis} if, additionally, the nonfull subcategory  $\myuline{\presentable}_{L,\mathrm{\baseCat}\mathrm{-st},\omega}\subset\myuline{\presentable}_{L,\mathrm{\baseCat}\mathrm{-st}}$ is closed under the multiplicative norms and so inherits a $\baseCat$--symmetric monoidal structure from the ambient category. Under this assumption, the $\baseCat$--category $\underline{\cat}\Tperfect_{\baseCat}$ thus attains a $\baseCat$--symmetric monoidal structure under the equivalence $\underline{\cat}\Tperfect_{\baseCat}\simeq \myuline{\presentable}_{L,\mathrm{\baseCat}\mathrm{-st},\omega}$ with tensor unit being the $\baseCat$--category $\myuline{\spectra}^{\omega}$ of compact $\baseCat$--spectra. In forthcoming work \cite{brankoKaifNatalie}, we show that this hypothesis is satisfied for the case of $\baseCat$ being $\orbit_G$, and so the results in this subsection hold for our main applications below in the case of equivariant algebraic K--theory.
\end{hypothesis}

\begin{nota}
Let $\kappa$ be a regular cardinal. We write $\underline{\widetilde{\cat}}_{\baseCat}^{\myuline{\mathrm{perf}},  \kappa}$ for the smallest ${\baseCat}$-symmetric monoidal subcategory of $\underline{\cat}_{\baseCat}\Tperfect$ containing $(\underline{\cat}_{\baseCat}\Tperfect)\tkappa$. In particular, since $(\underline{\cat}_{\baseCat}\Tperfect)\tkappa$ is small by \cref{compactGenerationOfCatPerf}, $\underline{\widetilde{\cat}}_{\baseCat}^{\myuline{\mathrm{perf}},  \kappa}$ is also small. We need this slight enlargement for the technical reason that we do not know a priori that $(\underline{\cat}_{\baseCat}\Tperfect)\tkappa$ inherits the ${\baseCat}$-symmetric monoidal structure of $\underline{\cat}_{\baseCat}\Tperfect$ since it is not clear that the multiplicative norms preserve parametrised-$\kappa$-compact objects.
\end{nota}

\begin{defn}
Let $\kappa$ be a regular cardinal. Let $\R\botPointwiseKappa$ be the collection of diagrams in $\underline{\widetilde{\cat}}_{\baseCat}^{\myuline{\mathrm{perf}},  \kappa} \subseteq \underline{\presheaf}_{\baseCat}(\underline{\widetilde{\cat}}_{\baseCat}^{\myuline{\mathrm{perf}},  \kappa})$ consisting of:
\begin{itemize}
    \item the diagram $\tconstant_{\baseCat}(\emptyset)\tcocone = \terminalTCat \rightarrow \underline{\widetilde{\cat}}_{\baseCat}^{\myuline{\mathrm{perf}},  \kappa}$ picking the zero category (i.e. the initial object),
    \item all split Karoubi sequences in $(\underline{\cat}_{\baseCat}\Tperfect)\tkappa\subseteq \underline{\widetilde{\cat}}_{\baseCat}^{\myuline{\mathrm{perf}},  \kappa} $.
\end{itemize}
When $\baseCat$ satisfies \cref{hypothesis2:norm_finiteness} so that $\underline{\cat}\Tperfect$ attains a $\baseCat$--symmetric monoidal structure, let $\R\botNormedKappa$ be the closure of $\R\botPointwiseKappa$  under $f_{\otimes}$ for arbitrary maps $f : U \rightarrow V$ in $\underline{\finite}_{\baseCat}$.
\end{defn}

Now, using the construction and notation from \cref{parametrisedPresentableDwyerKanLocalisation}, we may define the following intermediate notions of noncommutative motives.

\begin{defn}
Let $\kappa$ be a regular cardinal. We define \textit{unstable pointwise $ \kappa$-motives} $\underline{\nmot}_{\baseCat}^{\mathrm{pw},\unstable,   \kappa}$ to be $L_{\R\botPointwiseKappa}\underline{\presheaf}_{\baseCat}(\underline{\widetilde{\cat}}_{\baseCat}^{\myuline{\mathrm{perf}},  \kappa})$. When $\baseCat$ satisfies \cref{hypothesis2:norm_finiteness}, we define \textit{unstable normed $ \kappa$-motives} $\underline{\nmot}_{\baseCat}^{\mathrm{nm},\unstable,   \kappa}$ to be $L_{\R\botNormedKappa}\underline{\presheaf}_{\baseCat}(\underline{\widetilde{\cat}}_{\baseCat}^{\myuline{\mathrm{perf}},  \kappa})$.
\end{defn}

\noindent Since $\R\botPointwiseKappa$, $\R\botNormedKappa$ are small as $\underline{\widetilde{\cat}}_{\baseCat}^{\myuline{\mathrm{perf}},  \kappa}$ is, $\underline{\nmot}^{\mathrm{pw}, \unstable,   \kappa}_{\baseCat}$ and $\underline{\nmot}^{\mathrm{nm},\unstable,   \kappa}_{\baseCat}$ are  ${\baseCat}$-presentable.

\begin{rmk}\label{rmk:specialFormOfR_pw}
    For the purposes of capturing the notion of additivity, we may without loss of generality let $\R\botPointwise$ be the tensor ideal generated by the map $\myuline{\spectra}^{\underline{\omega},\Delta^1}/\myuline{\spectra}\rightarrow \myuline{\spectra}$ associated to the Waldhausen sequence for $\myuline{\spectra}\tomega$. To see this, we argue now that a functor $F\colon \underline{\cat}\Tperfect\rightarrow \underline{\A}$ which sends sequences of the form \cref{eqn:waldhausenSequence} to fibre sequences already sends any split Karoubi sequence (see \cref{eqn:splitKaroubiSequence} for notation) to a fibre sequence. To wit, given such a split Karoubi sequence, it is easy to see that we have a cofibre sequence $iq\Rightarrow \id_{\underline{\D}}\Rightarrow jp$ in $\underline{\func}\Texact(\underline{\D},\underline{\D})$. Hence, by \cref{rmk:waldhausenTrick}, we get a splitting 
    \[\id_{F\underline{\D}}\simeq F(i)\circ F(q)\oplus F(j)\circ F(p)\colon F\underline{\D}\rightarrow F\underline{\D}\]
    which implies that the sequence $F\underline{\sC}\rightarrow F\underline{\D}\rightarrow F\underline{\E}$ is a fibre sequence in $\underline{\A}$, as wanted. Therefore, it is enough to require that $\R\botPointwise$ consists only of the Waldhausen sequences \cref{eqn:waldhausenSequence}, which by the last sentence in \cref{obs:waldhausenSplitting}, is a tensor ideal generated by the Waldhausen sequence for $\myuline{\spectra}\tomega$. This completes the proof of the lemma.
\end{rmk}

\begin{nota}
Write $j^{\kappa}_{\unstable}\colon \underline{\widetilde{\cat}}_{\baseCat}^{\myuline{\mathrm{perf}},  \kappa}\xhookrightarrow{y_{\unstable}^{\kappa}}\underline{\presheaf}_{\baseCat}(\underline{\widetilde{\cat}}_{\baseCat}^{\myuline{\mathrm{perf}},  \kappa})\rightarrow \underline{\nmot}_{\baseCat}^{\mathrm{pw},\unstable,   \kappa}$ for the canonical composition. Since split Karoubi sequences were already cofibre sequences in $\underline{\widetilde{\cat}}_{\baseCat}^{\myuline{\mathrm{perf}},  \kappa}$ by definition, we get from \cite[Thm. 6.4.2]{kaifPresentable} that this functor is ${\baseCat}$-fully faithful.
\end{nota}

Recall now the notion of additive $\baseCat$--functors from \cref{defn:additiveFunctors}.

\begin{prop}\label{universalPropertyUnstableKappaMotives}
For every ${\baseCat}$-cocomplete category $
\underline{\E}$,
$(j^{\kappa}_{\unstable})^* : \underline{\func}^L_{\baseCat}(\underline{\nmot}_{\baseCat}^{\mathrm{pw},\unstable,  \kappa}, \underline{\E}) \rightarrow \underline{\func}_{\baseCat}\add(\underline{\widetilde{\cat}}_{\baseCat}^{\myuline{\mathrm{perf}},  \kappa}, \underline{\E})$ is an equivalence.
\end{prop}
\begin{proof}
This is immediate by \cref{parametrisedPresentableDwyerKanLocalisation} and the universal property of presheaves.
\end{proof}

\begin{cons}\label{cons:bigUnstablePointwiseMotives}
Let $\kappa \leq \kappa'$ be two regular cardinals. Then the composition $\underline{\widetilde{\cat}}_{\baseCat}^{\myuline{\mathrm{perf}},  \kappa}\subseteq \underline{\widetilde{\cat}}_{\baseCat}^{\myuline{\mathrm{perf}},  \kappa'} \hookrightarrow \underline{\nmot}_{\baseCat}^{\mathrm{pw},\unstable,  \kappa'}$ preserves initial objects and sends split Karoubi sequences to cofibre sequences. Hence by \cref{universalPropertyUnstableKappaMotives} we obtain a  ${\baseCat}$-colimit-preserving functor $\underline{\nmot}_{\baseCat}^{\mathrm{pw},\unstable,  \kappa} \rightarrow \underline{\nmot}_{\baseCat}^{\mathrm{pw}, \unstable,  \kappa'}$. This is ${\baseCat}$-fully faithful since it sends compact-generators to compact objects and is ${\baseCat}$-fully faithful on these. Similar considerations also apply when we replace motives with presheaves. From these, and since $\underline{{\cat}}_{\baseCat}\Tperfect \simeq \colim_{\kappa}\underline{\widetilde{\cat}}_{\baseCat}^{\myuline{\mathrm{perf}},  \kappa}$, we then define 
\[\underline{\nmot}_{\baseCat}^{\mathrm{pw},\unstable} \coloneqq \colim_{\kappa}\underline{\nmot}_{\baseCat}^{\mathrm{pw},\unstable,  \kappa}\quad\quad \underline{\presheaf}_{\baseCat}(\underline{\cat}\Tperfect)\coloneqq \colim_{\kappa}\underline{\presheaf}_{\baseCat}(\underline{\widetilde{\cat}}_{\baseCat}^{\myuline{\mathrm{perf}},  \kappa})\]
Applying $\colim_{\kappa}$ on all functors in sight give us a ${\baseCat}$-fully faithful functor
\[j_{\unstable} \colon  \underline{{\cat}}_{\baseCat}\Tperfect \xhookrightarrow{y_{\unstable}} \underline{\presheaf}_{\baseCat}(\underline{\cat}\Tperfect) \longrightarrow \underline{\nmot}_{\baseCat}^{\mathrm{pw},  \unstable}\]
Since the poset of regular cardinals is a large category and each of $\underline{\nmot}_{\baseCat}^{\mathrm{pw},\unstable,  \kappa}$ is large, we deduce that $\underline{\nmot}_{\baseCat}^{\mathrm{pw}, \unstable}$ is a large ${\baseCat}$-presentable category since large unions of large sets is large. We refer to \cite{nineAuthorsIV} for a more thorough discussion of set-theoretic considerations.
\end{cons}

\begin{prop}\label{universalPropertyUnstableBigMotives}
For a ${\baseCat}$--(co)complete category $\underline{\E}$, $(j_{\unstable})^* : \underline{\func}^L_{\baseCat}(\underline{\nmot}_{\baseCat}^{\mathrm{pw}, \unstable}, \underline{\E}) \rightarrow \underline{\func}_{\baseCat}\add(\underline{\cat}\Tperfect_{\baseCat}, \underline{\E})$ is an equivalence.
\end{prop}
\begin{proof}
By \cref{splitKaroubiSequencesAsFilteredColimits} we have $\underline{\func}\add_{\baseCat}(\underline{\cat}\Tperfect_{\baseCat}, \underline{\E}) \simeq {\lim}_{\kappa}\underline{\func}\add_{\baseCat}(\underline{{\cat}}_{\baseCat}^{\myuline{\mathrm{perf}},  \kappa}, \underline{\E})$ by restricting from the tautological equivalence $\underline{\func}_{\baseCat}(\underline{\cat}\Tperfect_{\baseCat}, \underline{\E}) \simeq {\lim}_{\kappa}\underline{\func}_{\baseCat}(\underline{{\cat}}_{\baseCat}^{\myuline{\mathrm{perf}},  \kappa}, \underline{\E})$.
But we also have the tautological equivalence $\underline{\func}^L_{\baseCat}(\underline{\nmot}_{\baseCat}^{\mathrm{pw},\unstable}, \underline{\E}) \simeq {\lim}_{\kappa} \underline{\func}^L_{\baseCat}(\underline{\nmot}_{\baseCat}^{\mathrm{pw},\unstable,  \kappa}, \underline{\E})$. Therefore we can apply \cref{universalPropertyUnstableKappaMotives} to conclude.
\end{proof}

\begin{cons}[Big stable motives]\label{cons:stableNoncommutativeMotives}
Define the ${\baseCat}$-presentable-stable category of \textit{parametrised noncommutative motives} to be
$\underline{\nmot}_{\baseCat}\pointwise \coloneqq \myuline{\spectra}_{\baseCat}(\underline{\nmot}^{ \mathrm{pw},\unstable}_{\baseCat})$. This yields
\[\Z_{\mathrm{pw}} \colon \underline{\cat}\Tperfect_{\baseCat} \xhookrightarrow{j_{\unstable}} \underline{\nmot}^{\mathrm{pw},\unstable}_{\baseCat} \xrightarrow{\canonical} \underline{\nmot}_{\baseCat}\pointwise.\]
Since $\baseCat$--stabilisation is a left adjoint in $\underline{\presentable}_{\baseCat, L}$, we also have
$\underline{\nmot}_{\baseCat}\pointwise \simeq \colim_{\kappa}\underline{\nmot}_{\baseCat}^{\mathrm{pw}, \kappa}$ where $\underline{\nmot}_{\baseCat}^{\mathrm{pw}, \kappa}\coloneqq \myuline{\spectra}_{\baseCat}(\underline{\nmot}_{\baseCat}^{\mathrm{pw},\unstable,  \kappa})$. We then obtain commuting composites
\begin{center}
    \begin{tikzcd}
        & \underline{\nmot}^{\mathrm{pw},\unstable, \kappa}_{\baseCat}\ar[dr, "\canonical"]\\
        \Z_{\mathrm{pw},\kappa} \colon\:\: \underline{\widetilde{\cat}}^{\myuline{\mathrm{perf}}, \kappa}_{\baseCat} \rar["\sU_{\kappa}"']\ar[ur, hook, "j_{\unstable}"] &\underline{\presheaf}^{\tstable}_{\baseCat}(\underline{\widetilde{\cat}}^{\myuline{\mathrm{perf}}, \kappa}_{\baseCat}) \ar[r, "\lambda_{\kappa}"']& \underline{\nmot}^{\mathrm{pw},\kappa}_{\baseCat}
    \end{tikzcd}
\end{center}
where $\sU_{\kappa}$ is the composite $$\underline{\widetilde{\cat}}^{\myuline{\mathrm{perf}}, \kappa}_{\baseCat} \hookrightarrow \underline{\presheaf}_{\baseCat}(\underline{\widetilde{\cat}}^{\myuline{\mathrm{perf}}, \kappa}_{\baseCat}) \xrightarrow{\canonical} \underline{\presheaf}^{\tstable}_{\baseCat}(\underline{\widetilde{\cat}}^{\myuline{\mathrm{perf}}, \kappa}_{\baseCat}) \coloneqq \myuline{\spectra}\otimes \underline{\presheaf}_{\baseCat}(\underline{\widetilde{\cat}}^{\myuline{\mathrm{perf}}, \kappa}_{\baseCat})$$
We will use this second description to handle monoidal matters later.
\end{cons}

\begin{thm}[Universal property of pointwise stable motives]\label{univPropStableMotives}
For every ${\baseCat}$-presentable-stable category $\underline{\E}$, 
$\Z^*_{\mathrm{pw}} \colon \underline{\func}^L_{\baseCat}(\underline{\nmot}_{\baseCat}^{\mathrm{pw}}, \underline{\E}) \rightarrow \underline{\func}_{\baseCat}\add(\underline{\cat}\Tperfect_{\baseCat}, \underline{\E})$ is an equivalence. 
\end{thm}
\begin{proof}
This is an immediate consequence of  \cref{universalPropertyUnstableBigMotives} and \cref{parametrisedSemiadditivisationSmashing}. 
\end{proof}

\begin{cons}[Pointwise connective algebraic $\KTheory$-theory] \label{parametrisedAlgebraicKTheory} \label{nota:pointwiseKTheory}
Recall from \cite{BGT13} that connective algebraic K-theory is given by the finite product-preserving functor
\[\K : \cat\perfect \xrightarrow{\Qcons_{\bullet}} \func(\mathbf{\Delta}\op, \cat\perfect) \xrightarrow{(-)^{\simeq}} \func(\mathbf{\Delta}\op, \spc) \xrightarrow{\colim} \spc \]
where $\underline{\Qcons}_{n}\underline{\sC} \simeq \underline{\func}_{\baseCat}(\twistedArrow(\Delta^n),\underline{\sC})$, Quillen's Q--construction. Since $\cmonoid(\spc) \rightarrow \spc$ preserves sifted colimits by \cite[{$\S3.2.3$}]{lurieHA}, it in particular preserves geometric realisations. Hence the geometric realisation used above to define $\K$ acquires a canonical commutative monoid structure because $(-)^{\simeq}\colon \cat\perfect\rightarrow \spc$ factors through $\forget\colon \cmonoid(\spc)\rightarrow \spc$.
Thus we can apply the ${\baseCat}$-cofree  \cref{cofreeParametrisations} and ${\baseCat}$-semiadditivise to get 
\[\underline{\K}_{\baseCat}\pointwise : \underline{\cmonoid}_{\baseCat}(\cat\perfect) \longrightarrow \underline{\cmonoid}_{\baseCat}(\spc)\] which we call the \textit{pointwise parametrised algebraic K-theory space}. Levelwise, this looks like \small
\begin{equation*}
    \begin{split}
        \mackey_{\baseCat}(\K) : \:\: \func^{\times}(\effBurn(\baseCat), \cat\perfect) & \xrightarrow{\underline{\Qcons}_{\bullet}} \func(\simplex\op, \func^{\times}(\effBurn(\baseCat), \cat\perfect))\\
        & \xrightarrow{(-)^{\simeq}} \func(\simplex\op, \func^{\times}(\effBurn(\baseCat), \cmonoid(\spc)))\\
        & \xrightarrow{\colim} \func^{\times}(\effBurn(\baseCat), \cmonoid(\spc))\\
    \end{split}
\end{equation*}\normalsize
We will have use of this description soon in analysing motivic suspensions. Note also that $\underline{\K}_{\baseCat}\pointwise$ is an additive theory by the usual unparametrised additivity theorem and since we define split Karoubi sequences in $\underline{\cmonoid}_{\baseCat}(\cat\perfect)$ as those that are pointwise split Karoubi in the usual sense. Moreover, one can deloop the algebraic K-theory space $\K$ to get an \textit{algebraic K-theory spectrum } $\KTheory : \cat\perfect \rightarrow \spectra$ which is the spectrum associated to the prespectrum whose $n$-th term is $\colim_{\bullet \in (\mathbf{\Delta}\op)^n}(\mathrm{Q}_{\bullet}\sC)^{\simeq}$ (cf. \cite[{$\S7.2$}]{BGT13} or Waldhausen's original treatment \cite{waldhausen} for more details using the equivalent $\Sdot_{\bullet}$--construction), and we write $\underline{\KTheory}_{\baseCat}\pointwise$ for the analogous pointwise K--theory spectrum. By construction, we have a factorisation $\underline{\KTheory}_{\baseCat}\pointwise\colon \underline{\cat}\Tperfect_{\baseCat} \subseteq \underline{\mackey}_{\baseCat}(\cat\perfect) \xrightarrow{\underline{\mackey}_{\baseCat}(\KTheory)}\myuline{\mackey}_{\baseCat}(\spectra)\simeq \myuline{\spectra}_{\baseCat}$.
\end{cons}

\begin{lem}\label{QFunCommutation}
Let $\underline{\sC}, \underline{\D}\in \underline{\cat}\Tperfect_{\baseCat}$. Then $\underline{\func}\Texact_{\baseCat}(\underline{\D}, \underline{\Qcons}_n\underline{\sC})\simeq \underline{\Qcons}_n\underline{\func}_{\baseCat}\Texact(\underline{\D}, \underline{\sC})$.
\end{lem}
\begin{proof}
Since $\underline{\Qcons}_n\underline{\sC}\simeq \underline{\func}_{\baseCat}(\twistedArrow(\Delta^n),\underline{\sC})$, we get $\underline{\func}_{\baseCat}(\underline{\D}, \underline{\Qcons}_n\underline{\sC}) \simeq  \underline{\Qcons}_n\underline{\func}_{\baseCat}(\underline{\D}, \underline{\sC})$ from \cref{parametrisedCotensors} (1). But then, both $\underline{\Qcons}_n\underline{\sC}$ and $\underline{\func}_{\baseCat}(\underline{\D}, \underline{\sC})$ inherit ${\baseCat}$-(co)limits from $\underline{\sC}$ (the former by \cref{parametrisedCotensors} (2)), and so clearly we obtain the statement required.
\end{proof}

\begin{lem}[``{\cite[{$\S7.3$}]{BGT13}}'']\label{motivicSuspensions}
Let $\underline{\sC}\in \underline{\cat}\Tperfect_{\baseCat}$. Then $\colim_{\bullet\in \mathbf{\Delta}\op} y_{\unstable} \underline{\Qcons}_{\bullet}(\underline{\sC})$ is already motivically local and $\colim_{\bullet\in \mathbf{\Delta}\op} j_{\unstable} \underline{\Qcons}_{\bullet}(\underline{\sC}) \simeq \Sigma j_{\unstable}(\underline{\sC})\in \underline{\nmot}_{\baseCat}^{\mathrm{pw},\unstable}$.
\end{lem}
\begin{proof}
To see the first part, let $\underline{\D} \in \underline{\cat}\Tperfect$. Then note that  \small
\begin{equation*}
    \begin{split}
        \myuline{\map}_{\underline{\presheaf}_{\baseCat}}(y_{\unstable}\underline{\D}, \colim_{\bullet\in \mathbf{\Delta}\op} y_{\unstable} \underline{\Qcons}_{\bullet}(\underline{\sC})) &\simeq \colim_{\bullet\in \mathbf{\Delta}\op} \myuline{\map}_{\underline{\presheaf}_{\baseCat}}(y_{\unstable}\underline{\D}, y_{\unstable} \underline{\Qcons}_{\bullet}(\underline{\sC}))\\
        &\simeq \colim_{\bullet\in \mathbf{\Delta}\op} \underline{\func}\Texact(\underline{\D},  \underline{\Qcons}_{\bullet}(\underline{\sC}))^{\simeq}\\
        &\simeq \colim_{\bullet\in \mathbf{\Delta}\op} \big(\underline{\Qcons}_{\bullet}\underline{\func}\Texact(\underline{\D}, \underline{\sC})\big)^{\simeq}  =: \myuline{\K}_{\baseCat}(\underline{\func}\Texact(\underline{\D}, \underline{\sC}))
    \end{split}
\end{equation*} \normalsize
and hence, since $\underline{\func}\Texact(-,\underline{\sC})$ preserves split Karoubi sequences and since $\myuline{\K}_{\baseCat}$ is additive, we obtain that indeed $\colim_{\bullet\in \mathbf{\Delta}\op} y_{\unstable} \underline{\Qcons}_{\bullet}(\underline{\sC})$ is motivically local as claimed.

For the second part, recall we have the simplicial split Karoubi sequence $\underline{\sC} \rightarrow \underline{\decalage}_{\bullet}\underline{\sC}\rightarrow \underline{\Qcons}_{\bullet}\underline{\sC}$    where we have adopted the terminology d\'{e}calage from \cite[{Lem. 2.4.7}]{nineAuthorsII}. The construction $\underline{\decalage}_{\bullet}\underline{\sC}$ is also called the simplicial path object in \cite[{Proof of Prop. 7.17}]{BGT13}. Now since $j_{\unstable} : \underline{\cat}\Tperfect_{\baseCat} \rightarrow \underline{\nmot}^{\mathrm{pw},\unstable}_{\baseCat}$ sends split Karoubi sequences to cofibre sequences by definition of unstable motives, and cofibre sequences are stable under colimits, we can apply $j_{\unstable}$ to the simplicial split Karoubi sequence and take geometric realisation in $\underline{\nmot}^{\mathrm{pw},\unstable}_{\baseCat}$ to get a cofibre sequence $j_{\unstable}(\underline{\sC}) \rightarrow \colim_{n\in\mathbf{\Delta}\op} j_{\unstable} \underline{\decalage}_{\bullet}\underline{\sC} \rightarrow \colim_{\bullet\in \mathbf{\Delta}\op}j_{\unstable}\underline{\Qcons}_{\bullet}\underline{\sC}$ in $\underline{\nmot}^{\mathrm{pw},\unstable}_{\baseCat}$.   
    But then we know that the middle term is always augmented over 0 and so is zero, hence the last term is a suspension of the first term, as required.
\end{proof}

\begin{thm}[Motivic corepresentability of pointwise $\KTheory$-theory]\label{MotivicCorepresentability}
Let $\underline{\sC}, \underline{\D} \in \underline{\cat}\Tperfect_{\baseCat}$. Then there is a natural equivalence 
\[\myuline{\mapsp}_{\underline{\nmot}_{\baseCat}^{\mathrm{pw}}}(\Z_{\mathrm{pw}}\underline{\sC},  \Z_{\mathrm{pw}}\underline{\D}) \simeq \underline{\KTheory}_{\baseCat}\pointwise(\underline{\func}\Texact(\underline{\sC}, \underline{\D})).\] In particular, $\underline{\KTheory}_{\baseCat}\pointwise$ is corepresented by $\Z_{\mathrm{pw}}\big(\myuline{\spectra}^ {\underline{\omega}}\big)$ by \cref{colimitGenerationSpectra}.
\end{thm}
\begin{proof}
Firstly, note that in $\underline{\nmot}_{\baseCat}^{\mathrm{pw},\unstable}$, $\Sigma^nj_{\unstable}\D \simeq \colim_{\bullet\in (\mathbf{\Delta}\op)^n}j_{\unstable} \underline{\Qcons}_{\bullet}\D$ since 
\begin{equation*}
    \begin{split}
        \Sigma^nj_{\unstable}\D \simeq \Sigma^{n-1}\big(\colim_{\bullet\in \mathbf{\Delta}\op}j_{\unstable} \underline{\Qcons}_{\bullet}\D\big) &\simeq  \colim_{\bullet\in \mathbf{\Delta}\op}\big(\Sigma^{n-1}j_{\unstable} \underline{\Qcons}_{\bullet}\D\big)\simeq \colim_{\bullet\in \mathbf{\Delta}\op}\big(\Sigma^{n-2}\big(\colim_{\bullet\in \mathbf{\Delta}\op}j_{\unstable} \underline{\Qcons}_{\bullet}\D\big)\big)
    \end{split}
\end{equation*}
and so on by \cref{motivicSuspensions}.  Now, the left hand parametrised spectrum in the theorem statement is the one associated to the prespectrum whose $n$-th term, for $n\geq 1$, is
\begin{equation*}
    \begin{split}
        \myuline{\map}_{\underline{\nmot}_{\baseCat}^{\mathrm{pw},\unstable}}(j_{\unstable}\underline{\sC},  \Sigma^nj_{\unstable}\underline{\D}) &\simeq \myuline{\map}_{\underline{\nmot}_{\baseCat}^{\mathrm{pw},\unstable}}(j_{\unstable}\underline{\sC},  \colim_{\bullet\in (\mathbf{\Delta}\op)^n}j_{\unstable} \underline{\Qcons}_{\bullet}\underline{\D}) \\
        &\simeq \myuline{\map}_{\underline{\presheaf}_{\baseCat}}(y_{\unstable}\underline{\sC},  \colim_{\bullet\in (\mathbf{\Delta}\op)^n}j_{\unstable} \underline{\Qcons}_{\bullet}\underline{\D}) \\
        &\simeq \colim_{\bullet\in (\mathbf{\Delta}\op)^n}\myuline{\map}_{\underline{\presheaf}_{\baseCat}}(y_{\unstable}\underline{\sC},  j_{\unstable} \underline{\Qcons}_{\bullet}\underline{\D})\\
        &\simeq \colim_{\bullet\in (\mathbf{\Delta}\op)^n}\myuline{\map}_{\underline{\cat}\Tperfect_{\baseCat}}(\underline{\sC},   \underline{\Qcons}_{\bullet}\underline{\D})\\
        &\simeq \colim_{\bullet\in (\mathbf{\Delta}\op)^n}\big(\underline{\Qcons}_{\bullet}\underline{\func}\Texact(\underline{\sC}, \underline{\D})\big)^{\simeq}\\
        &\simeq \loops\Sigma^n\underline{\KTheory}_{\baseCat}\pointwise\big(\underline{\func}\Texact(\underline{\sC}, \underline{\D})\big)
    \end{split}
\end{equation*}
where the second equivalence is since for $n\geq 1$, $\colim_{\bullet\in (\mathbf{\Delta}\op)^n}y_{\unstable} \underline{\Qcons}_{\bullet}\D$ is already in $\underline{\nmot}_{\baseCat}^{\mathrm{pw},\unstable}$ by \cref{motivicSuspensions}; the fourth since $j_{\unstable}$ is ${\baseCat}$-fully faithful; the fifth by \cref{QFunCommutation}; and the last by definition of $\underline{\KTheory}_{\baseCat}\pointwise$. Hence both parametrised spectra in the statement have equivalent associated spectra, giving the desired conclusion.
\end{proof}

As in \cref{cons:bigUnstablePointwiseMotives}, we can construct $\underline{\nmot}_{\baseCat}^{\mathrm{nm},\unstable}$, and we denote the canonical maps by $k_{\unstable}\colon \underline{\cat}\Tperfect\rightarrow \underline{\nmot}^{\mathrm{nm},\unstable}_{\baseCat}$ and $\Z_{\mathrm{nm}}\colon \underline{\cat}\Tperfect\rightarrow \underline{\nmot}^{\mathrm{nm}}_{\baseCat}$.  By definition of $\R_{\mathrm{nm}}$, the functors $k_{\unstable}$ and $\Z_{\mathrm{nm}}$ are in particular additive.

\begin{prop}\label{TSymmetricMonoidalityUnstableMotives}
Suppose $\baseCat$ satisfies \cref{hypothesis2:norm_finiteness}.  There is a ${\baseCat}$-symmetric monoidal structure on $\underline{\nmot}^{\mathrm{nm},\unstable}_{\baseCat}$ such that $k_{\unstable} : \underline{\cat}\Tperfect_{\baseCat}\longrightarrow\underline{\nmot}^{\mathrm{nm},\unstable}_{\baseCat}$ refines canonically to a ${\baseCat}$-symmetric monoidal functor.
\end{prop}
\begin{proof}
We first argue for the case of small motives. From \cite[Prop. 6.0.12]{nardinShah}, the Yoneda embedding $\underline{\widetilde{\cat}}_{\baseCat}^{\myuline{\mathrm{perf}}, \kappa} \hookrightarrow\underline{\presheaf}_{\baseCat}(\underline{\widetilde{\cat}}_{\baseCat}^{\myuline{\mathrm{perf}}, \kappa})$  refines to a ${\baseCat}$-symmetric monoidal functor. Furthermore, since the $\baseCat$--cartesian symmetric monoidal on $\underline{\spc}_{\baseCat}$ is $\baseCat$--distributive by \cite[Prop. 3.2.5]{nardinShah}, we see by \cite[Thm. 3.2.6]{nardinShah} that the $\baseCat$--Day convolution symmetric monoidal structure on $\underline{\presheaf}_{\baseCat}(\underline{\widetilde{\cat}}_{\baseCat}^{\myuline{\mathrm{perf}},  \kappa})$ is $\baseCat$--distributive. Hence, by \cref{parametrisedPresentableDwyerKanLocalisation} and by construction of $\R_{\mathrm{nm}}$,  $\underline{\nmot}_{\baseCat}^{\mathrm{nm},\unstable,   \kappa}$ attains a canonical $\baseCat$--symmetric monoidal structure which affords a refinement of $k_{\unstable}$ to a $\baseCat$--symmetric monoidal functor. This completes the case of small motives.

Now for the case of the big motives, applying again \cite[Cor. 6.0.12]{nardinShah}, we get that the ${\baseCat}$-symmetric monoidal inclusion $\underline{\widetilde{\cat}}_{\baseCat}^{\myuline{\mathrm{perf}}, \kappa} \subseteq \underline{\widetilde{\cat}}_{\baseCat}^{\myuline{\mathrm{perf}}, \kappa'}$ induces a ${\baseCat}$-symmetric monoidal refinement of $\underline{\presheaf}_{\baseCat}(\underline{\widetilde{\cat}}_{\baseCat}^{\myuline{\mathrm{perf}}, \kappa}) \rightarrow \underline{\presheaf}_{\baseCat}(\underline{\widetilde{\cat}}_{\baseCat}^{\myuline{\mathrm{perf}}, \kappa'})$. On the other hand, \cref{parametrisedPresentableDwyerKanLocalisation} (2) implies that this induces a ${\baseCat}$-symmetric monoidal refinement of $\underline{\nmot}_{\baseCat}^{\mathrm{nm},\unstable,  \kappa} \rightarrow \underline{\nmot}_{\baseCat}^{\mathrm{nm},\unstable,  \kappa'}$. Thus since filtered colimits of ${\baseCat}$-symmetric monoidal categories are formed underlying by the straightforward parametrised analogue of \cite[{$\S3.2.3$}]{lurieHA},  we obtain a canonical ${\baseCat}$-symmetric monoidal structure on $\underline{\nmot}_{\baseCat}^{\mathrm{nm},{\unstable}}$ together with a ${\baseCat}$-symmetric monoidal refinement of $\underline{\cat}_{\baseCat}\Tperfect \rightarrow \underline{\nmot}_{\baseCat}^{\mathrm{nm},{\unstable}}$. 
\end{proof}

\begin{prop}[Monoidality of normed motives]\label{monoidalityOfTMotives}
Suppose $\baseCat$ satisfies \cref{hypothesis2:norm_finiteness}.  The ${\baseCat}$-functor $\Z_{\mathrm{nm}} : \underline{\cat}\Tperfect_{\baseCat} \rightarrow \underline{\nmot}_{\baseCat}^{\mathrm{nm}}$  refines to a ${\baseCat}$-symmetric monoidal functor.
\end{prop}
\begin{proof}
The functor $k_{\unstable} : \underline{\cat}\Tperfect_{\baseCat} \rightarrow \underline{\nmot}_{\baseCat}^{\mathrm{nm},\unstable}$ is canonically  ${\baseCat}$-symmetric monoidal  by \cref{TSymmetricMonoidalityUnstableMotives}. Moreover, by \cite[Lem. 4.2.3]{kaifPresentable} and \cref{parametrisedSemiadditivisationSmashing}, $\underline{\nmot}^{\mathrm{nm},{\unstable}}_{\baseCat} \rightarrow \underline{\nmot}_{\baseCat}^{\mathrm{nm}}$ also refines uniquely to a  ${\baseCat}$-symmetric monoidal functor.
\end{proof}

Unlike in the pointwise situation where algebraic K--theory is a construction and its corepresentability in motives is a result, we now \textit{define} the \textit{normed} parametrised algebraic K--theory to be that which is corepresented by the unit in normed motives.

\begin{defn}\label{defn:normedKTheory}
Suppose $\baseCat$ satisfies \cref{hypothesis2:norm_finiteness}. The \textit{normed parametrised algebraic K--theory spectrum} $\underline{\KTheory}_{\baseCat}^{\normed}$ is defined as $\underline{\KTheory}_{\baseCat}^{\normed} \colon \underline{\cat}\Tperfect_{\baseCat} \xrightarrow{\Z_{\normed}} \underline{\nmot}_{\baseCat}^{\normed} \xrightarrow{\myuline{\mapsp}(\unit{,} -)}\myuline{\spectra}_{\baseCat}$.
\end{defn}

\begin{obs}\label{pointwiseNormedComparison}
Let us now highlight some points based on all our considerations so far: (1) The ${\baseCat}$--functor $\underline{\KTheory}_{\baseCat}^{\normed} \colon \underline{\cat}\Tperfect_{\baseCat} \rightarrow \myuline{\spectra}_{\baseCat}$ canonically refines to a ${\baseCat}$--lax symmetric monoidal functor because $\myuline{\mapsp}(\unit, -)$ canonically refines to such. (2) Since $\Z_{\mathrm{nm}} \colon \underline{\cat}\Tperfect_{\baseCat} \rightarrow \underline{\nmot}_{\baseCat}^{\mathrm{nm}}$ is additive, by \cref{univPropStableMotives}, we obtain a canonical comparison map $\fanpsi \colon \underline{\nmot}_{\baseCat}\pointwise \longrightarrow \underline{\nmot}_{\baseCat}^{\mathrm{nm}}$ which also yields a transformation of additive functors
    $\big(\underline{\KTheory}_{\baseCat}\pointwise \Rightarrow \underline{\KTheory}_{\baseCat}^{\normed}\big)\colon \underline{\cat}\Tperfect_{\baseCat}\longrightarrow\myuline{\spectra}_{\baseCat}$. We do not know in general if these comparison maps are equivalences. In the rest of the article, we prove that this \textit{is} so in the equivariant case  for $G$ a 2-group.\footnote{In follow-up joint work \cite{hilmanRamzi} with Ramzi, we prove it for all finite groups $G$.} (3) It might be tempting to try to prove directly that the map $\underline{\KTheory}_{\baseCat}\pointwise \Rightarrow \underline{\KTheory}_{\baseCat}^{\normed}$ is an equivalence by proving \cref{MotivicCorepresentability} in the case of normed motives. When one traces through this strategy, the key hiccup is in proving \cref{motivicSuspensions} that the Q--construction is already motivically local where we have used crucially that the parametrised algebraic K--theory space functor $\underline{\K}_{\baseCat}$ is additive, i.e. that it interacts well with maps in $\R_{\mathrm{pw}}$. What is needed in the normed setting is that it interacts well with maps in $\R_{\mathrm{nm}}\supseteq \R_{\mathrm{pw}}$. Indeed, this is what we will show in the equivariant case for 2--groups in \cref{section:equivariantK-TheoryFor2-Groups}.
\end{obs}

\subsection{Pointwise K--theory is normed for 2-groups} \label{subsection:PointwiseEqualsNormed2-Groups}\label{section:equivariantK-TheoryFor2-Groups}
In this subsection, we specialise the considerations of \cref{subsection:noncommutativeMotives} to the case of $\baseCat = \orbit_G$ where $G$ is a finite group, giving $G$--equivariant algebraic K--theory. As explained in \cref{hypothesis:normed_spectra,hypothesis2:norm_finiteness}, $\orbit_G$ satisfies the hypothesis therein and so the discussion about norm structures apply here. The end goal is to show \cref{mainTheorem:pointwiseEqualsNormedFor2-Groups} that $\underline{\KTheory}_G$ refines to the structure of a normed ring $G$--spectrum when $G$ is a 2--group. In other words, we will show that the comparison map 
$\Psi\colon \underline{\nmot}_G^{\mathrm{pw}}\rightarrow\underline{\nmot}_G^{\mathrm{nm}} $ and hence also $\underline{\KTheory}_G\pointwise\Longrightarrow\underline{\KTheory}_G^{\mathrm{{nm}}}$
from \cref{pointwiseNormedComparison} (2) are equivalences. First, recall the notations $\sU$ and $\lambda$ from \cref{cons:stableNoncommutativeMotives}. Let us also take stock of the theory developed in \cref{section:theoryOfParametrisedCubes}. While we have opted to use the more efficient stars and shrieks notation there, it would be beneficial now to switch to the more conventional notation in equivariant homotopy theory to keep track of the groups involved.

\begin{nota}\label{nota:specialNotationsForC_2}
    For a subgroup $H\leq G$ and $w\colon G/H\rightarrow G/G$ the unique map, we will from now on denote $w^*, w_!, w_*$ by $\res^G_H, \ind^G_H, \prod_{G/H}$ respectively. In the presence of $G$--symmetric monoidal structures, we will write $\norm^G_H$ for $w_{\otimes}$. When $|G/H|=2$, we will write the colimit of the $C_2$--pushouts discussed in \cref{example:C_2Pushouts} by $B\underline{\amalg}_AB$. In particular, by \cref{prop:pullbackFormulaForConeDiagrams} in the case when $\underline{I}=w_!\underline{\ast}$, we may express $B\underline{\amalg}_AB$ as the \textit{ordinary} pushout
    \begin{center}
        \begin{tikzcd}
            \ind^G_H\res^G_HA \rar["\varepsilon"]\dar["\ind^G_Hf"']\ar[dr, phantom, very near end, "\ulcorner"] & A\dar\\
            \ind^G_HB\rar & B\underline{\amalg}_AB
        \end{tikzcd}
    \end{center}
    Finally, given a cofibre sequence $A \rightarrow B\rightarrow C$ of $H$--objects in a pointed $G$--category with a $G$--distributive symmetric monoidal structure, a straightforward combination  of \cref{normsOfCofibres} and \cref{prop:specialIndex2Identification} gives us a cofibre sequence of $G$--objects
    \[A\otimes B\underline{\amalg}_{\norm^G_HA}B\otimes A\longrightarrow \norm^G_HB\longrightarrow \norm^G_HC\]
\end{nota}

We aim to prove that $\R_{\mathrm{pw}}$ is $\underline{\otimes}$--multiplicatively closed, i.e. if $H\leq G$ and we have a split Karoubi in $\cat\Tperfect_H$ as in \cref{eqn:splitKaroubiSequence}, then  $\norm^G_H(\sU(\underline{\D})/\sU(\underline{\sC})) \rightarrow \sU(\norm^G_H\underline{\E})$ induced by the $\localisation$-equivalence $\sU(\underline{\D})/\sU(\underline{\sC}) \rightarrow \sU(\underline{\E})$ is itself a $\localisation$-equivalence. This would show that the inclusion $\R\botPointwiseKappa \subseteq \R\botNormedKappa $ (cf. \cref{section:twoVariantsK-theory}) is an identification, and so the map $\fanpsi : \underline{\nmot}\pointwise_G \rightarrow \underline{\nmot}_G^{\mathrm{nm}}$ is an equivalence. Since size issues will not play a role in our discussions here, we will suppress any mention of $\kappa$.

\begin{cor}\label{pushoutProductExpansion}
Let $H\leq G$ with $|G/H|=2$. Suppose we have a pushout $Y=X\cup_AB$ in a $G$--distributive symmetric monoidal $G$-stable category $\underline{\sC}$.
Then we have the pushout
\begin{center}
    \begin{tikzcd}
    A\otimes B\underline{\amalg}_{\norm^G_HA}B\otimes A \rar \dar\ar[dr, phantom, very near end, "\scalebox{1.5}{$\ulcorner$}"]& \norm^G_HB\dar\\
    X\otimes Y\underline{\amalg}_{\norm^G_HX}Y\otimes X \rar & \norm^G_HY.
    \end{tikzcd}
\end{center}
\end{cor}
\begin{proof}
Writing $C$ for $\cofib(A\rightarrow B) \simeq \cofib(X\rightarrow Y)$, we get from the $G/H$-distributivity of $\norm^G_H$ together with \cref{nota:specialNotationsForC_2} that we have the map of cofibre sequences
\begin{center}
    \begin{tikzcd}
    A\otimes B\underline{\amalg}_{\norm^G_HA}B\otimes A \rar \dar& \norm^G_HB\dar\rar & \norm^G_HC \dar[equal] \\
    X\otimes Y\underline{\amalg}_{\norm^G_HX}Y\otimes X\rar & \norm^G_HY \rar & \norm^G_HC
    \end{tikzcd}
\end{center}
and so since $\underline{\sC}$ was stable, the left square is a fibrewise pushout.
\end{proof}

\begin{lem}\label{motivesCommuteWithC_2Pushouts}
Suppose $H\leq G$ with $|G/H| = 2$, and $\underline{\A} \xrightarrow{i} \underline{\B}$ is a split Karoubi inclusion in $\underline{\cat}\Tperfect_H$. Then the canonical map
$\Z(\underline{\A}\otimes\underline{\B})\underline{\amalg}_{\Z(\norm^G_H\underline{\A})}\Z(\underline{\B}\otimes\underline{\A}) \longrightarrow \Z(\underline{\A}\otimes\underline{\B}\underline{\amalg}_{\norm^G_H\underline{\A}}\underline{\B}\otimes\underline{\A})$
is an equivalence in $\underline{\nmot}_G$.
\end{lem}
\begin{proof}
By \cref{nota:specialNotationsForC_2} we have the left pushout in
\[
    \begin{tikzcd}
    \ind^G_H\res^G_H\norm^G_H\underline{\A} \ar[dr, phantom, very near end, "\scalebox{1.5}{$\ulcorner$}"]\rar["\varepsilon"]\dar[hook] & \norm^G_H\underline{\A}\dar[hook]\\
    \ind^G_H(\underline{\A}\otimes\underline{\B}) \rar\uar[bend right = 30, dashed] & \underline{\A}\otimes\underline{\B}\underline{\amalg}_{\norm^G_HA}\underline{\B}\otimes\underline{\A} \uar[bend right = 30, dashed]
    \end{tikzcd}
    \hspace{2mm}
    \begin{tikzcd}
    \ind^G_H\res^G_H\Z(\norm^G_H\underline{\A}) \ar[dr, phantom, very near end, "\scalebox{1.5}{$\ulcorner$}"]\rar["\varepsilon"]\dar & \Z(\norm^G_H\underline{\A})\dar\\
    \ind^G_H\Z(\underline{\A}\otimes\underline{\B}) \rar & \Z(\underline{\A}\otimes\underline{\B}\underline{\amalg}_{\norm^G_HA}\underline{\B}\otimes\underline{\A})
    \end{tikzcd}
\]
which is moreover a right-split Karoubi pushout by \cref{sourceOfRightSplitKaroubiPushouts}. Hence by \cref{motivicSplitKaroubiPushout} we obtain the right pushout square  above from which, using \cref{nota:specialNotationsForC_2}, we may conclude.
\end{proof}

Next, recall the notion of saturation closure from \cite[Def. 6.3.5]{kaifPresentable}.

\begin{obs}\label{obs:saturationTensorIdeal}
    Let $\underline{\sC}$ be equipped with a $G$--distributive symmetric monoidal structure. Suppose that $S$ were a tensor ideal, i.e. if for any $Z\in\underline{\sC}$ and $f\colon A\rightarrow B$ in $S$, we have that $\id_Z\otimes f$ is also in $S$. Then by  a standard argument, we know that $\overline{S}$ is also a tensor ideal. To wit, consider the collection $U\subseteq \overline{S}$ of morphisms $f$ such that $\id_Z\otimes f$ is again in $\overline{S}$ for any $Z\in\underline{\sC}$. By hypothesis on $S$, we know that $S\subseteq U$. Moreover, since $Z\otimes-$ commutes with all colimits by $G$--distributivity of the symmetric monoidal structure, it is easy to see that the three axioms in \cite[Def. 6.3.5]{kaifPresentable} are satisfied so that $U$ is a $G$--strongly saturated collection containing $S$, whence $U=\overline{S}$ by minimality of $\overline{S}$.
\end{obs}

\begin{lem}\label{saturationInduction}
Let $H\leq G$ with $|G/H|=2$ and $\underline{\sC}$ a $G$--distributive symmetric monoidal $G$-stable category. Suppose $S$ is a collection of morphisms in $\underline{\sC}$ which is a tensor ideal and $\overline{S}$ its $G$-strong saturation. If $\norm^G_H$ sends morphisms in $S$ to morphisms in $\overline{S}$, then $\norm^G_H$ also preserves all morphisms in the saturation $\overline{S}$.
\end{lem}
\begin{proof}
Write $U\subseteq \overline{S}$ for the collection of morphisms which get sent to a morphism in $\overline{S}$ by $\norm^G_H$. By hypothesis, $S\subseteq U$. We claim that $U$ is $G$--strongly saturated. The 2-out-of-3 property is clear, and so we only have to check the first two axioms. To see axiom (2), we need to show that if $\partial : \underline{J} \rightarrow \underline{\func}_H(\Delta^1,\underline{\sC})$ is a diagram that is pointwise in the full subcategory  $\underline{\func}^{U}_H(\Delta^1,\underline{\sC})$, then $\norm^G_H\underline{\colim}_{\underline{J}}\partial \in \func_G^{\overline{S}}(\Delta^1,\underline{\sC}).$ For this, recall by $G/H$-distributivity that $\norm^G_H\underline{\colim}_{\underline{J}}\partial$ is computed as the cone point of the $G$-colimit diagram
\[(\prod_{G/H}\underline{J})^{\underline{\triangleright}} \rightarrow \prod_{G/H}(\underline{J}^{\underline{\triangleright}}) \xrightarrow{\prod_{G/H}\partial} \prod_{G/H}\underline{\func}_H(\Delta^1,\underline{\sC}) \xrightarrow{\norm^G_H} \underline{\func}_G(\Delta^1,\underline{\sC}) \] Now the hypothesis on $\partial$ ensures that, when restricted to $\prod_{G/H}\underline{J}$, this composite lands in $\underline{\func}_G^{\overline{S}}(\Delta^1,\underline{\sC}) \subseteq \underline{\func}_G(\Delta^1,\underline{\sC})$ and since by definition $\underline{\func}_G^{\overline{S}}(\Delta^1,\underline{\sC})$ is closed under $G$-colimits, we obtain that the cone point $\norm^G_H\underline{\colim}_J\partial$ is indeed in $\underline{\func}_G^{\overline{S}}(\Delta^1,\underline{\sC})$ as required.

Finally, to see axiom (3), suppose we have a pushout of $H$--objects in $\underline{\sC}$ 
\begin{center}
    \begin{tikzcd}
    A \rar \dar\ar[dr, phantom, very near end, "\scalebox{1.5}{$\ulcorner$}"]& B\dar\\
    X \rar & Y
    \end{tikzcd}
\end{center}
where the left vertical is in $U$ (and so, by definition of saturation, the right vertical is in $\overline{S}$). Then by  \cref{pushoutProductExpansion} we obtain the pushout square
\[
    \begin{tikzcd}
    A\otimes B\underline{\amalg}_{\norm^G_HA}B\otimes A \rar \dar\ar[dr, phantom, very near end, "\scalebox{1.5}{$\ulcorner$}"]& \norm^G_HB\dar\\
    X\otimes Y\underline{\amalg}_{\norm^G_HX}Y\otimes X \rar & \norm^G_HY
    \end{tikzcd}
    \hspace{4mm}
    \begin{tikzcd}
     \ind^G_H\res^G_H\norm^G_HA\rar["\varepsilon"] \dar\ar[dr, phantom, very near end, "\scalebox{1.5}{$\ulcorner$}"]& \norm^G_HA\dar\\
    \ind^G_H(A\otimes B)\rar & A\otimes B\underline{\amalg}_{\norm^G_HA}A\otimes B
    \end{tikzcd}
\]
Hence if we can show that the left vertical map is in $\overline{S}$, then by definition, the right vertical map will be in $\overline{S}$ too. For this, by \cref{nota:specialNotationsForC_2} we have the right pushout above
and similarly for $X\otimes Y\underline{\amalg}_{\norm^G_HX}X\otimes Y$. Since the respective maps on the upper three terms between the ones for the pair $(A, B)$ and the ones for the pair $(X, Y)$ are all in $\overline{S}$ by hypothesis (the bottom left uses that $\overline{S}$ is a tensor ideal by \cref{obs:saturationTensorIdeal}), so is the induced map $A\otimes B\underline{\amalg}_{\norm^G_HA}A\otimes B \rightarrow X\otimes Y\underline{\amalg}_{\norm^G_HX}X\otimes Y$ by axiom (2) of \cite[Def. 6.3.5]{kaifPresentable}. Therefore, the left pushout above gives that $\norm^G_HB\rightarrow \norm^G_H Y$ is also in $\overline{S}$ as required.
\end{proof}

\begin{lem}\label{lem:saturationTensorIdealInduction}
    Let $H\leq G$ with $|G/H|=2$ and $\underline{\sC}$ a $G$--distributive symmetric monoidal $G$-stable category. Let $T$ be a collection of morphisms and $S\supseteq T$ the smallest tensor ideal containing $T$. If $\norm^G_H$ sends morphisms in $T$ to morphisms in $\overline{S}$, then $\norm^G_H$ also preserves all morphisms in $\overline{S}$. 
\end{lem}
\begin{proof}
    By \cref{saturationInduction}, it suffices to show that $\norm^G_H$ sends morphisms in $S$ to $\overline{S}$. So let $U\subseteq S$ be the collection of morphisms which get sent to $\overline{S}$ under $\norm^G_H$. By hypothesis, this contains $T$, and it is easy to check that it is also a tensor ideal because $\overline{S}$ is again a tensor ideal by \cref{obs:saturationTensorIdeal}. Hence by minimality we see that $U=S$ as required.
\end{proof}

To state the next result, observe that by \cref{prop:leftKanExtensionFunctorialityPresentables} we have the identification $\myuline{\spectra}_H^{\Delta^1}\underline{\amalg}^{s_!}_{\myuline{\spectra}_G}\myuline{\spectra}_H^{\Delta^1} \simeq \underline{\func}_G(\Delta^1\underline{\amalg}_{\Delta^0}\Delta^1, \myuline{\spectra}_G)$. Furthermore, by \cref{presheafCommutesWithTensor}, we also have the identification $\bigotimes_{G/H}\underline{\func}_H(\Delta^1, \myuline{\spectra}_H) \simeq \underline{\func}_G(\prod_{G/H}\Delta^1, \myuline{\spectra}_G)$.

\begin{lem}\label{technicalLeftKanExtensionThing}
Let $s : \Delta^0 \hookrightarrow \Delta^1$ be the source inclusion, $H\leq G$ with $|G/H|=2$, and $\varphi : \Delta^1\underline{\amalg}_{\Delta^0}\Delta^1 \hookrightarrow \prod_{G/H}\Delta^1$ the inclusion from \cref{prop:specialIndex2Identification}. Then the functor $$\underline{\func}_G(\Delta^1\underline{\amalg}_{\Delta^0}\Delta^1, \myuline{\spectra}_G) \rightarrow \underline{\func}_G(\prod_{G/H}\Delta^1, \myuline{\spectra}_G)$$ induced by $\bigotimes_{G/H}(\myuline{\spectra}_H \xrightarrow{s_!} \underline{\func}_H(\Delta^1, \myuline{\spectra}_H))$ is given by left Kan extension along the inclusion $\varphi$, and so in particular preserves $\omega$--compact objects and is $G$-fully faithful since $\varphi$ is $G$-fully faithful.
\end{lem}
\begin{proof}
    Write $\partial \colon \underline{\Delta}^1\rightarrow \underline{\cat}_H$ for the map classifying $\underline{\Delta}^0 \xrightarrow{s} \underline{\Delta}^1$. By \cref{presheafCommutesWithTensor}, we obtain the commuting square in 
    \begin{center}
        \begin{tikzcd}
            \prod_{G/H}\underline{\Delta}^1\rar["\prod_{G/H}\partial"] & \prod_{G/H}\underline{\cat}_H \ar[rrr,"{\prod_{G/H}\underline{\func}(-, \myuline{\spectra})}"] \dar["\underline{\times}"']& &&\prod_{G/H}\underline{\presentable}_{H,L, \mathrm{st}}\dar["\underline{\otimes}"]\\
            & \underline{\cat}_G \ar[rrr,"{\underline{\func}(-, \myuline{\spectra})}"] & &&\underline{\presentable}_{G,L, \mathrm{st}}
        \end{tikzcd}
    \end{center}
    Note, importantly, that it is easy to see that the essential image of these compositions factor through $\underline{\presentable}_{G,L, \mathrm{st},\omega}\subset\underline{\presentable}_{G,L, \mathrm{st}}$. By the commutativity of this diagram and the identification $\myuline{\spectra}_H^{\Delta^1}\underline{\amalg}^{s_!}_{\myuline{\spectra}_G}\myuline{\spectra}_H^{\Delta^1} \simeq \underline{\func}_G(\Delta^1\underline{\amalg}_{\Delta^0}\Delta^1, \myuline{\spectra}_G)$ from \cref{prop:leftKanExtensionFunctorialityPresentables}, applying $\underline{\func}(-,\myuline{\spectra})$ to the diagram in $\underline{\cat}_G$
    \begin{center}
        \begin{tikzcd}
            \underline{\Delta}^0 \rar["s",hook]\dar["s"',hook]\ar[dr,phantom,very near end, "\underline{\ulcorner}"] & \underline{\Delta}^1\dar\ar[ddr, "s\times \id", bend left = 30,hook]\\
            \underline{\Delta}^1 \rar\ar[drr, "\id\times s"', bend right = 20,hook] & \underline{\Delta}^1\underline{\amalg}_{\underline{\Delta}^0}\underline{\Delta}^1\ar[dr, dashed, "\varphi", hook]\\
            && \prod_{G/H}\underline{\Delta}^1
        \end{tikzcd}
    \end{center}
    where the outer diagram is encoded by $\prod_{G/H}\underline{\Delta}^1$--shaped diagram $\prod_{G/H}\partial$ then yields the desired statement.  
\end{proof}

For the next result, recall the notation from \cref{obs:waldhausenSplitting} as well as \cite[Prop. 6.3.6]{kaifPresentable} which says that strong saturations are the same as motivic equivalences in the current setting.

\begin{prop}\label{arrowSpectraCase}
Let $H\leq G$ be a subgroup of index 2. Then $\norm^G_H$ sends the morphism $t^* : \sU((\myuline{\spectra}_H\tomega)^{\Delta^1})/\sU(\myuline{\spectra}_H\tomega) \rightarrow \sU(\myuline{\spectra}_H\tomega)$ in $\R\botPointwise$ to a morphism in $\overline{\R}\botPointwise$.
\end{prop}
\begin{proof}
Note that we have the factorisation $s_*\colon \myuline{\spectra}_H \xrightarrow{s_!}\myuline{\spectra}_H^{\Delta^1}\xrightarrow[\simeq]{\cofib} \myuline{\spectra}_H$, which we learnt from Achim Krause. Hence applying $\norm^G_H$ to the whole square and applying the identification as in \cref{pushoutProductExpansion}, we get in turn the diagram
\begin{center}\adjustbox{scale=0.9}{
    \begin{tikzcd}
    \myuline{\spectra}_H^{\Delta^1}\underline{\amalg}^{s_!}_{\myuline{\spectra}_G}\myuline{\spectra}_H^{\Delta^1} \simeq \underline{\func}_G(\Delta^1\underline{\amalg}_{\Delta^0}\Delta^1, \myuline{\spectra}_G)\dar["\simeq"']\rar["\varphi_!", hook] & \underline{\func}_G(\prod_{G/H}\Delta^1, \myuline{\spectra}_G)\simeq \norm^G_H(\myuline{\spectra}_H^{\Delta^1}) \dar["\simeq" ]\\
    \myuline{\spectra}_H^{\Delta^1}\underline{\amalg}^{s_*}_{\myuline{\spectra}_G}\myuline{\spectra}_H^{\Delta^1} \rar & \norm^G_H(\myuline{\spectra}_H^{\Delta^1})
    \end{tikzcd}
    }
\end{center}
where the $G/H$-pushout on the top left is with respect to the $s_!$ diagram and the bottom left is with respect to the $s_*$ diagram. Since, by \cref{technicalLeftKanExtensionThing}, the top arrow is $\varphi_!$ which preserves $\omega$--compact objects and is $G$-fully faithful, so is the bottom arrow. Therefore, together with the $G/H$-distributivity of $\norm^G_H$, we obtain from \cref{nota:specialNotationsForC_2} the following solid cofibre sequence in $\presentable_{G,L,\mathrm{st},\omega}\simeq \cat_G\Tperfect$
\begin{center}
    \begin{tikzcd}
    \myuline{\spectra}_H^{\Delta^1}\underline{\amalg}^{s_*}_{\myuline{\spectra}_G}\myuline{\spectra}_H^{\Delta^1} \rar[hook] & \norm^G_H(\myuline{\spectra}_H^{\Delta^1}) \rar[two heads, "(t\times t)^*"]  & \norm^G_H\myuline{\spectra}_H \simeq \myuline{\spectra}_G \lar[bend right = 40, "(t\times t)_!"', hook, dashed] \lar[bend left = 40, hook,dashed]
    \end{tikzcd}
\end{center}
which is moreover Karoubi by \cref{cor:KaroubiFibreDescription}. Note that this is then automatically split by \cref{workhorseAdjunctionYoga} since the right hand Karoubi projection admits the dashed adjoints by \cite[Lem. 4.4.3]{kaifPresentable}, where everything in sight preserves $\omega$--compact objects. Hence, upon applying $(-)\tomega$ and by definition of the motivic localisation, the diagonal map in 
\begin{center}
    \begin{tikzcd}
    \frac{\sU(\norm^G_H(\myuline{\spectra}_H^{\underline{\omega},\Delta^1}))}{\sU(\myuline{\spectra}_H^{\underline{\omega},\Delta^1})\underline{\amalg}^{s_*}_{\sU(\myuline{\spectra}_G\tomega)}\sU(\myuline{\spectra}_H^{\underline{\omega},\Delta^1})} \dar\rar& \sU\big(\norm^G_H\myuline{\spectra}_H\tomega\big)\\
    \frac{\sU(\norm^G_H(\myuline{\spectra}_H^{\underline{\omega},\Delta^1}))}{\sU(\myuline{\spectra}_H^{\underline{\omega},\Delta^1}\underline{\amalg}^{s_*}_{\myuline{\spectra}_G\tomega}\myuline{\spectra}_H^{\underline{\omega},\Delta^1})} \ar[ur] 
    \end{tikzcd}
\end{center}
is a morphism in $\R\botPointwise$. So to show that the top horizontal map is in $\overline{\R}\botPointwise$, it will suffice to show that the left vertical map is in $\overline{\R}\botPointwise$: this is merely the observation that we have, by definition, a map of cofibre sequences in $\underline{\presheaf}^{\tstable}(\underline{\cat}\Tperfect)$
\begin{center}
    \begin{tikzcd}
    \sU(\myuline{\spectra}_H^{\underline{\omega},\Delta^1})\underline{\amalg}^{s_*}_{\sU(\myuline{\spectra}_G\tomega)}\sU(\myuline{\spectra}_H^{\underline{\omega},\Delta^1}) \rar\dar  &  \sU(\norm^G_H(\myuline{\spectra}_H^{\underline{\omega},\Delta^1}))  \rar\dar[equal]  &     \frac{\sU(\norm^G_H(\myuline{\spectra}_H^{\underline{\omega},\Delta^1}))}{\sU(\myuline{\spectra}_H^{\underline{\omega},\Delta^1})\underline{\amalg}^{s_*}_{\sU(\myuline{\spectra}_G\tomega)}\sU(\myuline{\spectra}_H^{\underline{\omega},\Delta^1})}\dar\\
    \sU(\myuline{\spectra}_H^{\underline{\omega},\Delta^1}\underline{\amalg}^{s_*}_{\myuline{\spectra}_G\tomega}\myuline{\spectra}_H^{\underline{\omega},\Delta^1})  \rar   &     \sU(\norm^G_H(\myuline{\spectra}_H^{\underline{\omega},\Delta^1}))   \rar   &   \frac{\sU(\norm^G_H(\myuline{\spectra}_H^{\underline{\omega},\Delta^1}))}{\sU(\myuline{\spectra}_H^{\underline{\omega},\Delta^1}\underline{\amalg}^{s_*}_{\myuline{\spectra}_G\tomega}\myuline{\spectra}_H^{\underline{\omega},\Delta^1})}
    \end{tikzcd}
\end{center}
and the left vertical is in $\overline{\R}\botPointwise$ by \cref{motivesCommuteWithC_2Pushouts}, and hence the right vertical is in $\overline{\R}\botPointwise$ too.
\end{proof}

\begin{lem}\label{subgroupIndex2PreserveSaturations}
Let $H\leq G$ with $|G/H|=2$. Then $\norm^G_H$ preserves morphisms in $\overline{\R}\botPointwise$.
\end{lem}
\begin{proof}
By \cref{rmk:specialFormOfR_pw}, we could have replaced $\R\botPointwise$ with the tensor ideal generated by the map $t^*\colon \sU(\myuline{\spectra}^{\underline{\omega},\Delta^1})/\sU(\myuline{\spectra}\tomega)\rightarrow \sU(\myuline{\spectra}\tomega)$ coming from the Waldhausen split Karoubi sequence.  By \cref{lem:saturationTensorIdealInduction}, it suffices to show that $\norm^G_H$ sends this morphism to one $\overline{\R}\botPointwise$. This input is in turn supplied by \cref{arrowSpectraCase}.
\end{proof}

The final ingredient to the main theorem is the following standard fact in group theory. 

\begin{fact}\label{groupTheoryNormalSeries}
Let $p$ be a prime, $G$ be a p-group, and $H \leq G$ a subgroup. There is a normal series $H = N_0 \triangleleft N_1 \triangleleft \cdots \triangleleft N_k = G$ such that the quotients $N_m/N_{m-1}\cong C_p$ for all $m$.
\end{fact}

\begin{thm}\label{mainTheorem:pointwiseEqualsNormedFor2-Groups}
Let $G$ be a 2-group. The inclusion $\R\botPointwiseKappa\subseteq \R\botNormedKappa$ is an identification, and hence the comparison $\fanpsi : \underline{\nmot}_G\pointwise \rightarrow \underline{\nmot}_G^{\mathrm{nm}}$ from \cref{pointwiseNormedComparison} is an equivalence. 
\end{thm}
\begin{proof}
Let $H \leq G$ be a subgroup. We need to show that $\norm^G_H$ preserves $\localisation$-equivalences. Equivalently, by \cite[Prop. 6.3.6]{kaifPresentable}, we need to show $\norm^G_H$ preserves  morphisms in $\overline{\R}\botPointwise$, the $G$-strong saturation of $\R\botPointwise = \langle\underline{\sC} \xrightarrow{s_*}\underline{\sC}^{\Delta^1} \xrightarrow{t^*} \underline{\sC}\rangle$. By \cref{groupTheoryNormalSeries}, let $H = N_0 \triangleleft N_1 \triangleleft \cdots \triangleleft N_k = G$ be a $C_2$-normal series. Since $\norm^G_H \simeq \norm^{N_k}_{N_{k-1}} \circ \cdots \circ \norm^{N_1}_{N_0}$, it would suffice to show that $\norm^{N_m}_{N_{m-1}}$ preserves morphisms in $\overline{\R}\botPointwise$. But then $N_{m-1}\triangleleft N_{m}$ is a normal inclusion of index 2, and so this assertion is true by \cref{subgroupIndex2PreserveSaturations}.
\end{proof}

In view of \cref{pointwiseNormedComparison}, the following is now an immediate consequence of the theorem.

\begin{cor}\label{mainCorollary:K-theoryOf2-Groups}
Let $G$ be a 2--group. Then $\underline{\KTheory}\pointwise_G \Rightarrow \underline{\KTheory}_G^{\mathrm{nm}} : \underline{\cat}\Tperfect_G \longrightarrow \myuline{\spectra}_G$  is an equivalence. In particular,  $\underline{\KTheory}\pointwise_G$ refines to the a $G$--lax symmetric monoidal structure and induces 
$\underline{\KTheory}_G : \calg_G(\underline{\cat}\Tperfect_G) \longrightarrow \calg_G(\myuline{\spectra}_G)$.
\end{cor}

Having performed a general analysis of normed equivariant algebraic K--theory, we record here a large source of examples via \cref{monoidalBorelificationPrinciple} coming from categories with $G$--actions. 

\begin{prop}\label{borelificationPrincipleForCatPerf}
Let $G$ be a finite group. The functor $\eval_{G/e} : \underline{\cat}\Tperfect_G \rightarrow \underline{\borel}(\cat\perfect)$ canonically refines to a $G$--symmetric monoidal functor $\eval_{G/e} : (\underline{\cat}\Tperfect_G)\totimes \rightarrow \underline{\borel}((\cat\perfect)^{\otimes})$.
Moreover, it admits a $G$--fully faithful right adjoint $\underline{\borel}(\cat\perfect) \hookrightarrow \underline{\cat}\Tperfect_G$.
\end{prop}
\begin{proof}
By \cref{monoidalBorelificationPrinciple} (2), we are left to show that $\eval_{G/e}$ is the unit of the Bousfield localisation in the $\widehat{\cat}$ version of \cref{prop:masterBorelificationLocalisation}. As noted there, this is fibrewise induced by taking homotopy fixed points in the target of the $H$--equivariant map $\res : \cat\Tperfect_H \rightarrow \cat\perfect$ to yield $\eval : \cat\Tperfect_H \rightarrow (\cat\perfect)^{hH}\simeq \func(BH, \cat\perfect)$ as desired. We now immediately obtain that the $G$--right adjoint is as claimed because fibrewise the adjunction is given by the dashed lift
\begin{center}
    \begin{tikzcd}
    \cat\Tperfect_H \rar["\eval_{H/e}", shift left = 1 ] \dar[hook]& \func(BH, \cat\perfect) \lar[shift left = 1, hook, dashed] \ar[dl, shift left = 4, xshift = 2, hook] \\
    \mackey_H(\cat\perfect) \ar[ur, shift right = 2, xshift=2,  "\eval_{H/e}"]
    \end{tikzcd}
\end{center}
for which the diagonal Bousfield localisation is given for instance by \cite[$\S8$]{barwick2} (compare also with \cref{monoidalBorelificationInTheLarge}), and the vertical inclusion is by \cref{mainMackeyInclusion}. 
\end{proof}

\begin{cor}\label{penultimateGLaxSymmetricMonoidal}
Let $G$ be a 2--group. Then the $G$--functor
\begin{equation}\label{eqn:equivariantKTheoryComposite}
    \begin{tikzcd}
        \underline{\KTheory}_G : \underline{\borel}_G(\cat\perfect) \rar[hook] &  \underline{\cat}\Tperfect_G \subset \underline{\mackey}_G(\cat\perfect) \ar[rr,"\underline{\mackey}_G(\KTheory)"] && \myuline{\spectra}_G
    \end{tikzcd}
\end{equation}
canonically refines to a $G$--lax symmetric monoidal functor
$\underline{\KTheory}_G \colon \underline{\borel}((\cat\perfect)^{\otimes}) \longrightarrow \myuline{\spectra}_G^{\underline{\otimes}}$. In particular, we obtain a functor ${\KTheory}_G \colon \func\big(BG, \calg(\cat\perfect)\big)\Tcore \longrightarrow \calg_G\big(\myuline{\spectra}_G^{\underline{\otimes}}\big)\Tcore$.
\end{cor}
\begin{proof}
That the functor $\underline{\cat}\Tperfect_G\rightarrow \myuline{\spectra}_G$ refines to a $G$--lax symmetric monoidal functor is by \cref{mainCorollary:K-theoryOf2-Groups}. That the inclusion $\underline{\borel}_G(\cat\perfect) \hookrightarrow  \underline{\cat}\Tperfect_G$ refines to a $G$--lax symmetric monoidal functor is because it is right adjoint to $\eval_{G/e}$, which in turn enhances to a $G$--symmetric monoidal functor by \cref{borelificationPrincipleForCatPerf}. Finally, applying $\calg_G$ to this composite and \cref{asafCalgMonad} (3) gives the last statement. 
\end{proof}



\printbibliography[heading=bibintoc]
\end{document}